\documentclass[10pt,a4paper]{amsart}
\usepackage[latin1]{inputenc}
\usepackage{amsmath}
\usepackage{amsfonts}
\usepackage{amssymb}
\usepackage{amsthm}
\usepackage[usenames, dvipsnames]{color}
\usepackage{tikz,epic}
\usepackage[all]{xy}
\usepackage{geometry,pgfplots}

\definecolor{liens}{rgb}{1,0,0}
\usepackage[colorlinks=true, linkcolor=blue, 
hyperfootnotes=true,citecolor=blue,urlcolor=black]{hyperref}

\renewcommand{\geq}{\geqslant}
\renewcommand{\leq}{\leqslant}

\newtheorem{theo}{Theorem}
\newtheorem{lemma}[theo]{Lemma}
\newtheorem{propo}[theo]{Proposition}
\newtheorem{coro}[theo]{Corollary}
\newtheorem{assumption}[theo]{Assumption}

\theoremstyle{definition}
\newtheorem{defi}[theo]{Definition}

\theoremstyle{remark}
\newtheorem{rem}[theo]{Remark}
\newtheorem{ex}[theo]{Example}

\def\iup{{\widetilde{\iota}}}

\def\Z{\mathbb{Z}}
\def\C{\mathbb{C}}
\def\R{\mathbb{R}}
\def\Q{\mathbb{Q}}
\def\N{{\mathbb N}}

\def\P1{\mathbb{P}^{1}}

\def\beq{\begin{equation}}
\def\eeq{\end{equation}}

\def\Etproj{\overline{E}_{t}}

\def\P2{\mathbb{P}^{2}}

\def\rx{r_{x}}
\def\ry{r_{y}}

\def\P1{\mathbb{P}^{1}}

\def\mer{\mathcal{M}{\it er}}
\def\PPP{\mathcal{P}}

\begin{document}
\title[On the algebraic nature of weighted quadrant walks]{Differential transcendence \& algebraicity criteria for the series counting weighted quadrant walks}
\author{Thomas Dreyfus}
\address{Institut de Recherche Math\'ematique Avanc\'ee, UMR 7501, Universit\'e de Strasbourg et CNRS, 7 rue Ren\'e Descartes, 67084 Strasbourg, France}
\email{dreyfus@math.unistra.fr}

\author{Kilian Raschel}
\address{Institut Denis Poisson, UMR 7013, Universit\'e de Tours, Universit\'e d'Orl\'eans et CNRS, Parc de Grandmont, 37200 Tours, France}
\email{raschel@math.cnrs.fr}


\thanks{This project has received funding from the European Research Council (ERC) under the European Union's Horizon 2020 research and innovation programme under the Grant Agreement No 648132 and the Grant Agreement No 759702.}

\date{\today}

\bibliographystyle{amsalpha}

\begin{abstract}
We consider weighted small step walks in the positive quadrant, and provide algebraicity and differential transcendence results for the underlying generating functions: we prove that depending on the probabilities of allowed steps, certain of the generating functions are algebraic over the field of rational functions, while some others do not satisfy any algebraic differential equation with rational function coefficients. Our techniques involve differential Galois theory for difference equations as well as complex analysis (Weierstrass parameterization of elliptic curves). We also extend to the weighted case many key intermediate results, as a theorem of analytic continuation of the generating functions. 
\end{abstract}

\maketitle
\setcounter{tocdepth}{1}

\sloppy
\section*{Introduction}

Take a walk with small steps in the positive quadrant $\mathbb{Z}_{\geq0}^{2}=\{0,1,2,\ldots\}^2$, that is a succession of points
\begin{equation*}
     P_0, P_1,\ldots ,P_k,
\end{equation*}     
where each $P_n$ lies in the quarter plane, where the moves (or steps) $P_{n+1}-P_{n}$ belong to a finite step set $\mathcal S\subset\{0,\pm 1\}^{2}$ which has been chosen a priori, and the probability to move in the direction $P_{n+1}-P_{n}=(i,j)$ is equal to some weight-parameter $d_{i,j}$, with $\sum_{(i,j)\in\mathcal S}d_{i,j}=1$. The following picture is an example of  such path:\smallskip

\begin{center}
\begin{tikzpicture}[scale=.8, baseline=(current bounding box.center)]
\foreach \x in {0,1,2,3,4,5,6,7,8,9,10}
  \foreach \y in {0,1,2,3,4}
    \fill(\x,\y) circle[radius=0pt];
\draw (0,0)--(10,0);
\draw (0,0)--(0,4);

\draw[->](0,0)--(1,1);
\draw[->](1,1)--(1,0);
\draw[->](1,0)--(0,1);
\draw[->](0,1)--(1,2);
\draw[->](1,2)--(2,1);
\draw[->](2,1)--(2,0);
\draw[->](2,0)--(3,1);
\draw[->](3,1)--(3,0);
\draw[->](3,0)--(4,1);
\draw[->](4,1)--(3,2);
\draw[->](3,2)--(2,3);
\draw[->](2,3)--(2,2);
\draw[->](2,2)--(3,3);
\draw[->](3,3)--(4,2);
\draw[->](4,2)--(4,1);
\draw[->](4,1)--(5,0);
\draw[->](5,0)--(6,1);
\draw[->](6,1)--(6,0);
\draw[->](6,0)--(7,1);
\draw[->](7,1)--(8,0);
\draw[->](8,0)--(9,1);
\draw[->](9,1)--(9,0);
\draw[->](9,0)--(10,1);
\draw[->](10,1)--(9,2);
\draw[->](9,2)--(8,3);
\draw[->](8,3)--(8,2);
\end{tikzpicture}
\quad\quad
$\mathcal S=\left\{\begin{tikzpicture}[scale=.4, baseline=(current bounding box.center)]
\foreach \x in {-1,0,1} \foreach \y in {-1,0,1} \fill(\x,\y) circle[radius=2pt];
\draw[thick,->](0,0)--(-1,1);
\draw[thick,->](0,0)--(1,1);
\draw[thick,->](0,0)--(1,-1);
\draw[thick,->](0,0)--(0,-1);
\end{tikzpicture}\right\}$ \quad $P_0 =(0,0)$
\end{center}

\medskip

 Such objects are very natural both in combinatorics and probability theory: they are interesting for themselves and also because they are strongly related to other discrete structures, see \cite{BMM,DeWa-15} and references therein.

Our main object of investigation is the probability 
\begin{equation}
\label{eq:def_qijk}
     \mathbb P[P_0{\stackrel{k}{\longrightarrow}} (i,j)]  
\end{equation}
that the walk started at $P_0$ be at some generic position $(i,j)$ after the $k$th step, with all intermediate points $P_n$ remaining in the cone. More specifically we shall turn our attention to the generating function (or counting function)
\begin{equation}
\label{eq:GF}
     Q(x,y;t)=\sum_{i,j,k\geq0} \mathbb P[P_0\stackrel{k}{\longrightarrow} (i,j)] x^{i}y^{j}t^{k}.
\end{equation}
We are interesting in classifying the algebraic nature of the above series: to which of the following classes does the function \eqref{eq:GF} belong to:
\begin{multline}
\label{eq:classes}
    \{\text{rational}\}\subset\{\text{algebraic}\}\subset\{\text{holonomic}\}\subset\{\text{differentially algebraic}\} \\ \text{v.s.}\quad \{\text{differentially transcendental}\}?
\end{multline}
Rational and algebraic functions are classical notions. By $Q(x,y;t)$ holonomic (resp.\ differentially algebraic) we mean that all of $x\mapsto Q(x,y;t)$, $y\mapsto Q(x,y;t)$ and $t\mapsto Q(x,y;t)$ satisfy a linear (resp.\ algebraic) differential equation with coefficients in $\C(x)$, $\C(y)$ and $\C(t)$, respectively; see Definition \ref{defi1} for a more precise statement. We say that $Q(x,y;t)$ is differentially transcendental if it is not differentially algebraic.
Our main results give sufficient conditions on the weights $d_{i,j}$ to characterize the algebraic nature \eqref{eq:classes} of the counting function \eqref{eq:GF}.

\subsection*{Motivations to consider models of weighted walks}
In this article we shall go beyond the classical hypothesis consisting in studying unweighted walks, that is walks with $d_{i,j}=1/\vert\mathcal{S}\vert$ for all $(i,j)\in \mathcal{S}$. Indeed, motivations to consider weighted models are multiple: first, they offer a natural framework to generalize the numerous results established for unweighted lattice walks, see our bibliography for a nonexhaustive list of works concerning unweighted quadrant walks, especially \cite{BMM,BRS,BostanKauersTheCompleteGenerating,DHRS,FaRa-10,KurkRasch,KuRa-15,Mi-09,MelcMish,MR09}. Second, some models of unweighted walks in dimension $3$ happen to be, after projection, equivalent to models of 2D weighted walks \cite{BoBMKaMe-16}. Needless to mention that lattice walks in 3D represent a particularly challenging topic, see \cite{BoBMKaMe-16,DuHoWa-16}. Third, these models with weights yield results in probability theory, where the hypothesis to have only uniform probabilities (case of unweighted lattice walks) is too restrictive. Fourth, since there exist infinitely many weighted models (compare with only $79$ unweighted small step models!), case-by-case reasonings should be excluded, and in some sense only intrinsic arguments merge up, like in \cite{FIM,DHRS2}. 

\subsection*{Literature} 
There is a large literature on (mostly unweighted) walks in the quarter plane, focusing on various probabilistic and combinatorial aspects. Two main questions have attracted the attention of the mathematical community: 
\begin{itemize}
     \item finding a closed-form expression for the probability \eqref{eq:def_qijk}, or equivalently for the series \eqref{eq:GF};
     \item characterizing the algebraic nature of the series \eqref{eq:GF}, according to the classes depicted in \eqref{eq:classes}. 
\end{itemize}
The first question, combinatorial in nature, should not put the second one in the shade: knowing the nature of $Q(x,y;t)$ has consequences on the asymptotic behavior of the coefficients, and further allows to apprehend the complexity of these lattice paths problems (to illustrate this fact, let us remind that unconstrained walks are associated with rational generating functions, while walks confined to a half-plane admit algebraic counting functions \cite{BMPe-00}). This is the second question that we shall consider in the present work. 

To summarize the main results obtained so far in the literature, one can say that for unweighted quadrant models, the generating function \eqref{eq:GF} is holonomic (third class of functions in \eqref{eq:classes}) if and only if a certain group of transformations (simply related to the weights, see \eqref{eq:generators_group}) is finite; note that models having a finite group are models to which a generalization of the well-known reflection principle applies. This is a very satisfactory result, as it connects combinatorial aspects to geometric features. Moreover, there are various tools for verifying whether given parameters lead to a finite or infinite group \cite{BMM,FaRa-10,KauersYatchak}.

Going back to the algebraic nature of the counting function, the pioneering result is \cite{FIM} (Chapter 4 of that book), which states that if the group is finite, the function \eqref{eq:GF} is holonomic, and even algebraic provided that some further condition be satisfied. Then Mishna \cite{Mi-09}, Mishna and Rechnitzer \cite{MR09}, Mishna and Melczer \cite{MelcMish} observed that there exist infinite group models such that the series is nonholonomic. Bousquet-M\'elou and Mishna \cite{BMM}, Bostan and Kauers \cite{BostanKauersTheCompleteGenerating} proved that the series is holonomic for all unweighted quadrant models with finite group. The converse statement is shown in \cite{KurkRasch,BRS}: for all infinite group models, the series is nonholonomic. The combination of all these works yields the aforementioned equivalence between finite group and holonomic generating function.

The question of differential algebraicity was approached more recently. Bernardi, Bousquet-M\'elou and the second author of the present paper showed \cite{BBMR16} that despite being nonholonomic, $9$ unweighted quadrant models are differentially algebraic. In \cite{DHRS,DHRS2}, the first author of this paper, Hardouin, Roques and Singer proved that all $47$ remaining infinite group models are differentially transcendental. See \cite{FaRa-10,KauersYatchak,DuHoWa-16} for related studies.

\subsection*{Main results}
The above recap shows how actively the combinatorial community took possession of these quadrant walk problems. It also illustrates that within a relatively small class of problems (only $79$ unweighted different models!), there exists a remarkable variety of behaviors. This certainly explains the vivid interest in this model. 

In this article we bring three main contributions, building on the recent works \cite{FIM,KurkRasch,KuRa-15,DHR,DHRS,DHRS2} and mixing techniques coming from complex analysis and Galois theory. The {\it first contribution} is about the techniques: along the way of proving our other contributions we generalize a certain number of results of \cite{FIM} (stationary probabilistic case) and \cite{KurkRasch} (unweighted combinatorial case). In particular we prove an analytic continuation result of the generating functions to a certain elliptic curve, see Theorem \ref{theo:analytic_continuation}.

The differential Galois results of \cite{DHRS} are applied in the setting of unweighted walks with infinite group, and the only serious obstruction to their extension to the weighted case was to prove the analytic continuation result of Theorem \ref{theo:analytic_continuation}. Accordingly the results of \cite{DHRS} may be extended to the weighted case. Our {\it second contribution} is to provide differential transcendence sufficient conditions for infinite group weighted walks, see Theorems \ref{thm:generic1},~\ref{thm:generic2} and~\ref{thm:generic3}. Those are consequences of a more general result, Proposition \ref{prop:caract_hyperalg}, coming from \cite{DHR}, which is a criteria (i.e., a necessary and sufficient condition) for differential transcendence. The latter is however not totally explicit in terms of the parameters $d_{i,j}$, contrary to the (easily verified) sufficient conditions. Note that the proofs here are omitted since they are similar to \cite{DHRS}.

Our {\it third result} is about models having a finite group. Theorem \ref{theo:alg_crit1} and Corollary~\ref{coro:alg_crit2} show that the generating functions are then holonomic, and even algebraic if and only if a certain quantity vanishes (namely, the alternating sum of the monomial $xy$ under the orbit of the group).

\subsection*{Structure of the paper}
\begin{itemize}
     \item Section \ref{sec:kernel}: statement of the kernel functional equation \eqref{eq:funcequ} satisfied by the generating function, study of the zero set defined by the kernel. Results in that section generalize results in \cite{FIM,KurkRasch}, at several places with new and minimal proofs.
    
     \item Section \ref{sec:uniformization}: elliptic parametrization of the zero set of the kernel and continuation of the generating functions. Results in this section are importantly used in Sections~\ref{sec:transcendence} and \ref{sec:algebraicity}.
     \item Section \ref{sec:transcendence}: statement of Theorems \ref{thm:generic1}, \ref{thm:generic2} and \ref{thm:generic3}, giving sufficient conditions for differential transcendence of the counting series.
     \item Section \ref{sec:algebraicity}: Theorem \ref{theo:alg_crit1} and Corollary \ref{coro:alg_crit2} on algebraicity criteria for the generating functions in the finite group case.
\end{itemize}
\noindent {\bf Acknowledgements.}  We would like to thank Alin Bostan, Charlotte Hardouin, Manuel Kauers, Julien Roques and Michael Singer for discussions and support of this work. We express our deepest gratitude to the anonymous referee, whose numerous remarks and suggestions have led us to greatly improve our work.

\section{Kernel of the walk}
\label{sec:kernel}

\subsection{Functional equation}
Weighted walks with small steps in the quarter plane are sums of steps taken in a step set $\mathcal S$, itself being a subset of 
$\{
\begin{tikzpicture}[scale=0.3]
;
\draw[thick,->](0,0)--(-1,0);
\end{tikzpicture}
,
\begin{tikzpicture}[scale=0.3]
;
\draw[thick,->](0,0)--(-1,1);
\end{tikzpicture}
 ,
 \begin{tikzpicture}[scale=0.3]
;
\draw[thick,->](0,0)--(0,1);
\end{tikzpicture}
 ,\begin{tikzpicture}[scale=0.3]
;
\draw[thick,->](0,0)--(1,1);
\end{tikzpicture},
\begin{tikzpicture}[scale=0.3]
;
\draw[thick,->](0,0)--(1,0);
\end{tikzpicture}
,
\begin{tikzpicture}[scale=0.3]
;
\draw[thick,->](0,0)--(1,-1);
\end{tikzpicture}
,
\begin{tikzpicture}[scale=0.3]
;
\draw[thick,->](0,0)--(0,-1);
\end{tikzpicture}
,
\begin{tikzpicture}[scale=0.3]
;
\draw[thick,->](0,0)--(-1,-1);
\end{tikzpicture}
\}$,
or alternatively
\begin{equation*}
\begin{tikzpicture}[scale=.4, baseline=(current bounding box.center)]
\foreach \x in {-1,0,1} \foreach \y in {-1,0,1} \fill(\x,\y) circle[radius=2pt];
\draw[thick,->](0,0)--(0,-1);
\draw[thick,->](0,0)--(-1,-1);
\draw[thick,->](0,0)--(-1,0);
\draw[thick,->](0,0)--(-1,1);
\draw[thick,->](0,0)--(0,1);
\draw[thick,->](0,0)--(1,0);
\draw[thick,->](0,0)--(1,1);
\draw[thick,->](0,0)--(1,-1);
\end{tikzpicture}\smallskip
\end{equation*}
Steps will be identified with pairs $(i,j)\in\{0,\pm 1\}^{2}\backslash\{(0,0)\}$. 
For $(i,j)\in\{0,\pm 1\}^{2}$, let $d_{i,j}\in  [0,1]$ with $\sum d_{i,j}=1$.
We consider quadrant walks starting from $P_0=(0,0)\in\Z_{\geq 0}^{2}$, which at each time move in the direction $(i,j)$ (resp.\ stay at the same position) with probability $d_{i,j}$ (resp.\ $d_{0,0}$). A walk will be called unweighted if $d_{0,0}=0$ and if in addition all nonzero $d_{i,j}$ take the same value. 

As said in the introduction, we will mainly focus on the probability $\mathbb P[P_0{\stackrel{k}{\longrightarrow}} (i,j)]$ that the walk be at position $(i,j)$ after $k$ steps, starting from $P_0$ and with all intermediate points $P_1,\ldots,P_{k-1}$ in the quarter plane. The corresponding trivariate generating function $Q(x,y;t)$ is defined in~\eqref{eq:GF}. Being the generating function of probabilities, $Q(x,y;t)$ converges for all $(x,y,t)\in \C^{3}$ such that $\vert x\vert,\vert y\vert \leq 1$ and $\vert t\vert < 1$. Note that in several papers, as in \cite{BMM}, it is not assumed that $\sum d_{i,j}=1$. However, after a rescaling of the $t$-variable, we may always reduce to this case.
 
The kernel of the walk is the polynomial defined by 
$
K(x,y;t):=xy (1-t S(x,y)),
$
where $S(x,y)$ denotes the jump polynomial
\begin{equation}
\label{eq:def_S}
S(x,y) =\sum_{(i,j)\in \{0,\pm 1\}^{2}} d_{i,j}x^i y^j
= A_{-1}(x) \frac{1}{y} +A_{0}(x)+ A_{1}(x) y
=  B_{-1}(y) \frac{1}{x} +B_{0}(y)+ B_{1}(y) x,
\end{equation}
and $A_{i}(x) \in x^{-1}\R[x]$, $B_{i}(y) \in y^{-1}\R[y]$. Define further
\begin{equation}
\label{eq:def_F1_F2}
     F^{1}(x;t):= K(x,0;t)Q(x,0;t) \quad\text{and}\quad F^{2}(y;t):= K(0,y;t)Q(0,y;t).
\end{equation}
The following is an adaptation of \cite[Lemma~4]{BMM} to our context; the proof is omitted since it is exactly the same as in \cite{BMM,DHRS2}.
\begin{lemma}
The generating function $Q(x,y;t)$ introduced in \eqref{eq:GF} satisfies the functional equation
\begin{equation} 
\label{eq:funcequ}
     K(x,y;t)Q(x,y;t)=F^{1}(x;t) +F^{2}(y;t)-K(0,0;t) Q(0,0;t)+xy.
\end{equation}
\end{lemma}

\subsection{Nondegenerate walks}
From now, let us fix $t\in (0,1)$.
The kernel curve $\Etproj$ is defined as the zero set in $\P1(\C)^2$ of the homogeneous polynomial
\begin{equation} 
\label{eq:kernelwalk}
     \overline{K}(x_0,x_1,y_0,y_1;t)= x_0x_1y_0y_1 -t \sum_{i,j=0}^2 d_{i-1,j-1} x_0^{i} x_1^{2-i}y_0^j y_1^{2-j},
\end{equation}
namely,
\begin{equation}
\label{eq:expression_Et}
     \Etproj=\{([x_{0}\!:\!x_{1}],[y_{0}\!:\!y_{1}])\in\P1(\C)^2: \overline{K}(x_0,x_1,y_0,y_1;t)=0\}.
\end{equation}
Working in $\P1(\mathbb C)$ rather than on $\mathbb C$ appears to be particularly convenient and allows avoiding tedious discussions (as in particular on the number of branch points, see \cite[Section 2.1]{KurkRasch}).
 To simplify our notation, for $x=[x_{0}\!:\!x_{1}]$ and $y=[y_{0}\!:\!y_{1}]\in \P1(\C)$, we shall alternatively write $\overline{K}(x,y;t)$, $\overline{K}(x,y_0,y_1;t)$ and $\overline{K}(x_0,x_1,y;t)$ instead of $\overline{K}(x_0,x_1,y_0,y_1;t)$.

Following \cite{FIM} we introduce the concept of degenerate model.
\begin{defi}
\label{def:degeneate}
A walk is called degenerate if one of the following assertions holds:
\begin{itemize}
     \item $(x,y)\mapsto K(x,y;t)$ is reducible over $\C[x,y]$;
     \item $(x,y)\mapsto K(x,y;t)$ has not bidegree $(2,2)$, that is $x$- or $y$-degree smaller than or equal to~$1$.
\end{itemize}
\end{defi}
By \eqref{eq:def_S} or \eqref{eq:kernelwalk}, the terms in $x^{2}$ (resp.\ $y^{2}$) in   $K(x,y;t)/t$ form a polynomial of degree at most two in $y$ (resp.\ $x$), which does not depend upon $t$. So $(x,y)\mapsto K(x,y;t)$ has not bidegree $(2,2)$ for one particular $t\in \C^{*}$ if and only if it has not bidegree $(2,2)$ for every value of $t\in \C^{*}$. 

On the other hand,  the polynomial $K(x,y;t)$ might be irreducible for some values of $t\in \C^{*}$ and reducible for other values of $t$, so that the notion of degenerate model a priori depends on $t$. However, we will see in Proposition \ref{prop:singcases} and its proof that a walk is degenerate for a specific value of $t\in (0,1)$ if and only if it is degenerate for every $t\in(0,1)$, so that the values of $t$ yielding a factorization of $K$ are outside of $(0,1)$. Note that by \cite[Proposition 1.2]{DHRS2}, it was already known that the latter values were algebraic over $\Q(d_{i,j})$.\\ 

It is natural to ask whether one can have a geometric understanding of degenerate models, to what Proposition \ref{prop:singcases} below answers. An analogue of it has been proved in \cite[Lemma 2.3.2]{FIM} in the case $t=1$ and in \cite[Proposition 1.2]{DHRS2} when $t$ is transcendental over $\Q(d_{i,j})$. It is noteworthy that the step sets leading to degenerate models are exactly the same if $t$ is a free variable in $(0,1)$ (our results), if $t$ is transcendental over $\Q(d_{i,j})$ (results of \cite{DHRS2}) or if $t$ is fixed to be equal to $1$ (\cite{FIM}).

\begin{propo}\label{prop:singcases}
A walk (or model) is degenerate if and only if at least one of the following holds:
\begin{enumerate}
\item \label{case1} All the weights are $0$ except maybe  $\{d_{1,1},d_{0,0},d_{-1,-1}\}$ or  $\{d_{-1,1},d_{0,0},d_{1,-1}\}$. This corresponds to walks with steps supported in one of the following configurations:
$$\begin{tikzpicture}[scale=.4, baseline=(current bounding box.center)]
\foreach \x in {-1,0,1} \foreach \y in {-1,0,1} \fill(\x,\y) circle[radius=2pt];
\draw[thick,->](0,0)--(-1,-1);
\draw[thick,->](0,0)--(1,1);
\end{tikzpicture}
\qquad\qquad
\begin{tikzpicture}[scale=.4, baseline=(current bounding box.center)]
\foreach \x in {-1,0,1} \foreach \y in {-1,0,1} \fill(\x,\y) circle[radius=2pt];
\draw[thick,->](0,0)--(1,-1);
\draw[thick,->](0,0)--(-1,1);
\end{tikzpicture}
$$
\item \label{case2} There exists $i\in \{- 1,1\}$ such that $d_{i,-1}=d_{i,0}=d_{i,1}=0$. This corresponds to (half-space) walks with steps supported in one of the following configurations:
$$\begin{tikzpicture}[scale=.4, baseline=(current bounding box.center)]
\foreach \x in {-1,0,1} \foreach \y in {-1,0,1} \fill(\x,\y) circle[radius=2pt];
\draw[thick,->](0,0)--(0,-1);
\draw[thick,->](0,0)--(1,-1);
\draw[thick,->](0,0)--(1,0);
\draw[thick,->](0,0)--(1,1);
\draw[thick,->](0,0)--(0,1);
\end{tikzpicture}\qquad\qquad
\begin{tikzpicture}[scale=.4, baseline=(current bounding box.center)]
\foreach \x in {-1,0,1} \foreach \y in {-1,0,1} \fill(\x,\y) circle[radius=2pt];
\draw[thick,->](0,0)--(0,-1);
\draw[thick,->](0,0)--(-1,-1);
\draw[thick,->](0,0)--(-1,0);
\draw[thick,->](0,0)--(-1,1);
\draw[thick,->](0,0)--(0,1);
\end{tikzpicture}
$$
\item \label{case3} There exists $j\in \{-1, 1\}$ such that $d_{-1,j}=d_{0,j}=d_{1,j}=0$. This corresponds to (half-space) walks with steps supported in one of the following configurations:
$$\begin{tikzpicture}[scale=.4, baseline=(current bounding box.center)]
\foreach \x in {-1,0,1} \foreach \y in {-1,0,1} \fill(\x,\y) circle[radius=2pt];
\draw[thick,->](0,0)--(-1,0);
\draw[thick,->](0,0)--(-1,1);
\draw[thick,->](0,0)--(0,1);
\draw[thick,->](0,0)--(1,1);
\draw[thick,->](0,0)--(1,0);
\end{tikzpicture} 
\qquad\qquad
\begin{tikzpicture}[scale=.4, baseline=(current bounding box.center)]
\foreach \x in {-1,0,1} \foreach \y in {-1,0,1} \fill(\x,\y) circle[radius=2pt];
\draw[thick,->](0,0)--(-1,0);
\draw[thick,->](0,0)--(-1,-1);
\draw[thick,->](0,0)--(0,-1);
\draw[thick,->](0,0)--(1,-1);
\draw[thick,->](0,0)--(1,0);
\end{tikzpicture}
$$
\end{enumerate}
\end{propo}

With $A_i$ and $B_i$ defined in \eqref{eq:def_S}, introduce 
\begin{equation*}
     \widetilde{A}_{i}=xA_{i}\in \R[x]\quad \text{and} \quad 
     \widetilde{B}_{i}=yB_{i}\in \R[y].
\end{equation*}
Note that Case \ref{case2} (resp.\ Case \ref{case3}) of Proposition~\ref{prop:singcases} holds if and only if at least one of $\widetilde{B}_{\pm1}$ (resp.\ $\widetilde{A}_{\pm1}$) is identically zero. Before proving Proposition~\ref{prop:singcases}, we state and show an intermediate lemma. 

\begin{lemma}\label{lem:posneg}
Assume that none of $\widetilde{A}_{\pm1},\widetilde{B}_{\pm1}$ is identically zero, and remind that $t\in (0,1)$. Let $i\in\{-1,1\}$ be fixed. The following holds:
\begin{itemize}
     \item If $\widetilde{A}_{i}(x)=t\widetilde{A}_{0}(x)-x=0$, then $x=0$;
     \item If $\widetilde{B}_{i}(y)=t\widetilde{B}_{0}(y)-y=0$, then $y=0$.
\end{itemize}
\end{lemma}

\begin{proof}
The proof is elementary: as we shall see the roots of the polynomials $\widetilde{A}_{\pm 1}(x)$ and $\widetilde{B}_{\pm 1}(y)$ have nonpositive real parts, while the roots of the polynomials $t\widetilde{A}_{0}(x)-x$ and $t\widetilde{B}_{0}(y)-y$ are real and nonnegative. If they coincide they must be zero.
 
Let us prove only the first item of Lemma \ref{lem:posneg}, as the other one is similar. Since $t\in (0,1)$, the polynomial $t\widetilde{A}_{0}(x)-x$ is not identically zero. First, let us show that the roots of 
\begin{equation*}
     t\widetilde{A}_{0}(x)-x=d_{-1,0}t+(d_{0,0}t-1)x+d_{1,0}tx^{2}
\end{equation*}
are real and nonnegative. Indeed,
\begin{equation*}
     t\widetilde{A}_{0}(0)\geq 0\quad \text{and}\quad t\widetilde{A}_{0}(1)-1<0,
\end{equation*}
since $t\in (0,1)$ and $d_{-1,0}+d_{0,0}+d_{1,0}\in [0,1]$. If $d_{1,0}>0$ then $\lim_{x\to +\infty}t\widetilde{A}_{0}(x)-x=+\infty$ and the mean value theorem implies that the two roots of the quadratic real polynomial $t\widetilde{A}_{0}(x)-x$ are real nonnegative. If $d_{1,0}=0$, then $\widetilde{A}_{0}(x)$ has degree exactly one and the mean value theorem implies the existence of one real root in $[0,1)$.   

 On the contrary, the coefficients of the quadratic polynomial $\widetilde{A}_{i}(x)$ are all positive, and thus its roots are necessarily real nonpositive, or complex conjugate numbers with nonpositive real parts. It shows that a common root of $\widetilde{A}_{i}(x)$ and $t\widetilde{A}_{0}(x)-x$ has to be $x=0$.
\end{proof}

\begin{proof}[Proof of Proposition \ref{prop:singcases}]
The proof extends that of \cite[Proposition 1.2]{DHRS2} in the situation where $t$ is not necessarily transcendental over $\Q(d_{i,j})$. We are going to detail (and generalize) only the parts of the proof in \cite{DHRS2} where the transcendence hypothesis is used.  \par 

First, the proof that Cases (\ref{case1}), (\ref{case2}) and (\ref{case3}) in the statement of Proposition \ref{prop:singcases} correspond to degenerate walks is exactly the same. More precisely:
\begin{itemize}
     \item The first (resp.\ the second) configuration of Case \eqref{case1} corresponds to a kernel which is a univariate polynomial in $xy$ (resp.\ a quadratic form in $(x,y)$), and it may be factorized.
     \item The first (resp.\ second) configuration of Case \eqref{case2} corresponds to a kernel with $x$-valuation (resp.\ $x$-degree) larger than or equal to $1$ (resp.\ smaller than or equal to $1$). Obviously, a kernel with $x$-degree larger than or equal to $1$ may be factorized by $x$, leading indeed to a degenerate case. 
     \item Similarly, the first (resp.\ second) configuration of Case~\eqref{case3} corresponds to a kernel with $y$-valuation (resp.\ $y$-degree) bigger than or equal to $1$ (resp.\ smaller than or equal to $1$).
\end{itemize}
Conversely, let us assume that the walk is degenerate for $t\in (0,1)$, but that we are not in Cases~(\ref{case2}) or (\ref{case3}), and let us prove that we are in Case (\ref{case1}). Our assumption implies that the kernel has bidegree $(2,2)$ and bivaluation $(0,0)$, that is $x$- and $y$-valuation equal to $0$. Therefore, according to the dichotomy of Definition \ref{def:degeneate}, $K(x,y;t)$ is reducible over $\C[x,y]$. Let us write a factorization 
\begin{equation}
\label{eq:factorization_K}
     K(x,y;t)=-f_{1}(x,y)f_{2}(x,y), 
\end{equation}
with nonconstant $f_{1}(x,y),f_{2}(x,y)\in \C[x,y]$. In the proof of \cite[Proposition 1.2]{DHRS2}, it is shown that $f_{1}(x,y)$ and $f_{2}(x,y)$ are irreducible polynomials of bidegree $(1,1)$, and to conclude the authors use the real structure of the kernel. Moreover, the assumption that $t$ is transcendental over $\Q(d_{i,j})$ is only used in the first part of the proof of \cite[Proposition 1.2]{DHRS2}. Thus, to extend their result in our context, we just have to show that for every $t\in (0,1)$: 
\begin{itemize}
     \item $f_{1}$ and $f_{2}$ in \eqref{eq:factorization_K} have bidegree $(1,1)$;
     \item $f_{1}$ and $f_{2}$ are irreducible in the ring $\C[x,y]$. 
\end{itemize}

\noindent\textit{We claim that both $f_{1}$ and $f_{2}$ in \eqref{eq:factorization_K} have bidegree $(1,1)$.} 
 Suppose that $f_{1}$ or $f_{2}$ does not  have bidegree $(1,1)$. Since $K$ is
 of bidegree at most $(2,2)$ then at least one of the $f_i$'s has degree $0$ in $x$ or $y$.
Up to interchange of  $x$ and $y$, and $f_1$ and $f_2$, we assume that $f_{1}(x,y)$ has $y$-degree $0$ and we denote it by $f_{1}(x)$. 

The factorization \eqref{eq:factorization_K} then reduces to ${K(x,y;t)=-f_{1}(x)f_{2}(x,y)}$, and we deduce that $f_{1}(x)$ is a common factor of the nonzero polynomials $\widetilde{A}_{-1}$, $t\widetilde{A}_{0}-x$ and $\widetilde{A}_{1}$ (these polynomials are all nonzero because we are not in Cases (\ref{case2}) and (\ref{case3}) of Proposition~\ref{prop:singcases}). By Lemma~\ref{lem:posneg}, $x=0$ is the only possible common root of the three previous polynomials.  
Hence $d_{-1,-1}=d_{-1,0}=d_{-1,1}=0$, which contradicts that Case \ref{case2} does not hold. Therefore $f_{1}(x)$ has degree $0$, contradicting that $f_{1}(x,y)$ is nonconstant and showing the claim. \\
  \par 
\noindent \textit{We claim that $f_1$ and $f_2$ are irreducible in the ring $\C[x,y]$.} If not, then $f_1(x,y) = (ax-b)(cy-d)$ for some $a,b,c,d \in \C$. Since $f_{1}(x,y)$ has bidegree $(1,1)$, we have $ac \neq 0$. Then it follows from \eqref{eq:def_S} that
\begin{equation*}
     0 = K\Bigl(\frac{b}{a},y;t\Bigr) =\frac{b}{a}y - t\left(\widetilde{A}_{-1}\Bigl(\frac{b}{a}\Bigr) + \widetilde{A}_0\Bigl(\frac{b}{a}\Bigr)y + \widetilde{A}_1\Bigl(\frac{b}{a}\Bigr)y^2\right).
\end{equation*}
In particular we find $\widetilde{A}_1(\frac{b}{a})=t \widetilde{A}_0(\frac{b}{a})-\frac{b}{a}=0$, hence by Lemma \ref{lem:posneg}, $b=0$. This contradicts the fact that $K$ has  {$x$-valuation} $0$. A similar argument shows that $f_2(x,y)$ is irreducible.  
\end{proof}

From now, we assume that the walk is nondegenerate. This discards only one-dimensional problems (resp.\ walks with support included in a half-plane), which are easier to study and whose generating functions are systematically algebraic, as explained in \cite[Section 2.1]{BMM}.

\subsection{Unit circles on the kernel curve}
Due to the natural domain of convergence of the power series $Q(x,y;t)$, the domains of $\Etproj$ where $\vert x\vert=1$ and $\vert y\vert=1$ are particularly interesting; some of their properties are studied in Lemmas \ref{lem3} and \ref{lem:curve_x1}. 

Hereafter we do the convention that the modulus of an element of $\P1(\C)\setminus \{[1\!:\!0]\}$ is the modulus of its corresponding complex number. Furthermore, the point $[1\!:\!0]$ has modulus strictly bigger than that of any other element in $\P1(\C)$.

\begin{lemma}
\label{lem3}
There are no $x,y\in  \mathbb{P}^{1}(\C)$ with $\vert x\vert=\vert y\vert=1$ such that $(x,y)\in \Etproj$.
\end{lemma}

\begin{proof}
Let $x,y\in \P1(\C)$ with $\vert x\vert =\vert y\vert =1$. They might be identified with the corresponding elements of $\C$. The triangular inequality yields $\vert S(x,y)\vert \leq 1$ (recall that $\sum d_{i,j}=1$) and finally $t\vert S(x,y)\vert <1 $. Since $\vert xy\vert =1$, we deduce the inequality ${K(x,y;t)\neq 0}$.
\end{proof}

Introduce the sets
\begin{equation}
\label{def:Gamma-x_Gamma-y}
     \Gamma_{x}=\Etproj\bigcap \{\vert x\vert = 1 \}\quad \text{and}\quad
     \Gamma_{y}=\Etproj\bigcap \{\vert y\vert = 1 \}.
\end{equation}
Let us see $y\mapsto \overline{K}(x,y;t)$ as a polynomial of degree two, and let (in what follows, $\sqrt {\,\cdot\,}$ denotes a fixed determination of the complex square root)
\begin{equation*}
     Y_{\pm}(x,t)=Y_{\pm}(x)=\dfrac{-tA_{0}(x)+1\pm \sqrt{(tA_{0}(x)-1)^{2}-4t^{2}A_{-1}(x)A_{1}(x)}}{2tA_{1}(x)}
\end{equation*}
be the two roots. 
For all $x\in \P1(\C)$, we have $(x,Y_{\pm}(x))\in \Etproj$. Similarly, let the two roots of $x\mapsto \overline{K}(x,y;t)$ be denoted by 
\begin{equation*}
     X_{\pm}(y,t)=X_{\pm}(y)=\dfrac{-tB_{0}(y)+1\pm \sqrt{(tB_{0}(y)-1)^{2}-4t^{2}B_{-1}(y)B_{1}(y)}}{2tB_{1}(y)}.
\end{equation*} 
For all $y\in \P1(\C)$, we have $(X_{\pm}(y),y)\in \Etproj$.\par 
The following result generalizes \cite[Lemma 2.3.4]{FIM}. It is illustrated on Figure \ref{fig:curves}. 

\begin{lemma}\label{lem:curve_x1}
Assume that the walk is nondegenerate. The set $\Gamma_{x}$ is composed of two disjoint connected (for the classical topology of $\P1(\C)^{2}$) paths: $\Gamma_{x}^{-}$ such that $(x,y)\in \Gamma_{x}^{-}\Rightarrow \vert y\vert <1$, and $\Gamma_{x}^{+}$ such that $(x,y)\in \Gamma_{x}^{+}\Rightarrow \vert y\vert >1$. A symmetric statement holds for $\Gamma_{y}$.
\end{lemma}

\begin{proof}
By definition
\begin{equation*}
     \Gamma_{x}=\{(x,Y_{-}(x)): \vert x\vert=1\}\bigcup \{(x,Y_{+}(x)): \vert x\vert=1\}.
\end{equation*}  
We need here to consider the $t$-dependency of $Y_{\pm}(x,t)$, which is defined for all $t\in (0,1)$. Note that $A_{1}(x)$ is not identically zero since otherwise the walk would be degenerate (Definition~\ref{def:degeneate}). The Taylor expansion at $t=0$ of $Y_{\pm}(x,t)$ gives
\begin{equation*}
     Y_{+}(x,t)=\frac{1}{tA_{1}(x)}+o(1/t),
\qquad Y_{-}(x,t)=\frac{t}{4}\frac{4A_{-1}(x)A_{1}(x)-A_{0}(x)^{2}}{A_{1}(x)}+o(t),
\end{equation*}     
proving that when $x$ is fixed with $A_{1}(x)\neq 0$, $Y_{+}(x,t)$ goes to infinity and $Y_{-}(x,t)$ goes to $0$ when $t$ goes to $0$. (Note, it may happen that the numerator $4A_{-1}(x)A_{1}(x)-A_{0}(x)^{2}$ be identically zero, but the conclusion remains.) Since the path $\{(x,Y_{\pm}(x,t)): \vert x\vert=1\}$ cannot intersect the unit disk by Lemma \ref{lem3}, we obtain that when $t$ is close to zero, the path $\{(x,Y_{-}(x,t)): \vert x\vert=1\}$ has $y$-coordinates  with modulus strictly smaller than $1$, while $\{(x,Y_{+}(x,t)): \vert x\vert=1\}$ has $y$-coordinates with modulus strictly bigger than $1$. So the result stated in Lemma \ref{lem:curve_x1} is correct for $t$ close to zero. Let us prove the result for an arbitrary $t$. To the contrary, assume that it is not the case. Since the two paths depend continuously upon $t\in (0,1)$, there must exists $t_{0}\in (0,1)$ such that one of the two paths intersect $\overline{E}_{t_{0}}\bigcap \{\vert y\vert = 1 \}$, thereby contradicting Lemma \ref{lem3}.
\end{proof}

\begin{figure}
\includegraphics[scale=0.3]{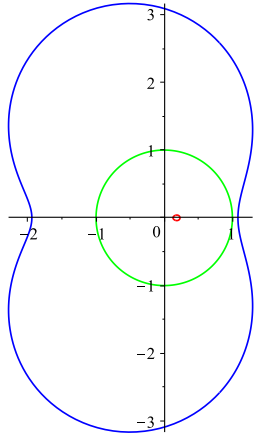}
\caption{The unit circle (in green) and the paths $Y_+(\{\vert x\vert=1\})$ (blue) and $Y_-(\{\vert x\vert=1\})$ (red), for the model with jumps $d_{-1,1}=\frac{1}{2}$, $d_{-1,0}=d_{0,-1}=d_{1,1}=1/6$ (see on the left of Figure \ref{fig:illustration_theorems_infinite}) and $t=0.96$}
\label{fig:curves}
\end{figure}

\subsection{Discriminants and genus of the walk}\label{secdisc}

Remind that we have assumed that the walk is not degenerate. Let us define the genus of the walk as the genus of the algebraic curve $\Etproj$. As we will see, the genus may be equal to zero or one, but we will only discuss nondegenerate walks of genus one, the genus zero case being considered in \cite{Mi-09,MR09,MelcMish,DHRS2}. Moreover, when $t$ is transcendental over $\Q(d_{i,j})$, every nondegenerate weighted walk of genus zero has a differentially transcendental generating function \cite{DHRS2}. 

\medskip

For $[x_0\!:\!x_1],[y_0\!:\!y_1]\in\P1(\C)$, denote by $\Delta^x_{[x_0:x_1]}$ and $\Delta^y_{[y_0:y_1]}$ (or $\Delta^x_{[x_0:x_1]}(t)$ and $\Delta^y_{[y_0: y_1]}(t)$ if it is convenient to emphasize the $t$-dependency) the discriminants of the second degree homogeneous polynomials $y \mapsto \overline{K}(x_0,x_1,y;t)$ and  $x \mapsto \overline{K}(x,y_0,y_1;t)$, respectively, i.e.,
\begin{multline}
\Delta^x_{[x_0:x_1]}=t^2 \Big( \big(d_{-1,0} x_1^2 -  \frac{1}{t} x_0x_1 +d_{0,0}x_0x_1 + d_{1,0}x_0^2\big)^2  \\
 - 4(d_{-1,1} x_1^2 + d_{0,1} x_0x_1 + d_{1,1}x_0^2)(d_{-1,-1} x_1^2 + d_{0,-1} x_0x_1 + d_{1,-1}x_0^2) \Big) \label{eq:expression_branch_points_x}
\end{multline}
and 
\begin{multline}
\Delta^y_{[y_0:y_1]}=t^2 \Big( \big(d_{0,-1} y_1^2 -  \frac{1}{t} y_0y_1 +d_{0,0}y_0y_1+ d_{0,1}y_0^2\big)^2  \\ 
- 4(d_{1,-1} y_1^2 + d_{1,0} y_0y_1 + d_{1,1}y_0^2)(d_{-1,-1} y_1^2 + d_{-1,0} y_0y_1 + d_{-1,1}y_0^2) \Big).\label{eq:expression_branch_points_y}
\end{multline}
The polynomial $\Delta^x_{[x_0:x_1]}$ (resp.\ $\Delta^y_{[y_0:y_1]}$)  is of degree four and so has four roots $a_1, a_2, a_3, a_4$ (resp.\ $b_1, b_2, b_3 , b_4$). We also introduce
\begin{equation}
\label{eq:expression_D_0}
     D(x):=\Delta^x_{[x:1]}=\sum_{j=0}^{4}\alpha_{j}x^{j}\quad \text{and}\quad E(y):=\Delta^y_{[y:1]}=\sum_{j=0}^{4}\beta_{j}y^{j}.
\end{equation}    
See Figure \ref{fig:ordering} for two different plots of $D$ (one when $a_4>0$, another one when $a_4<0$). Plainly, the signs of 
\begin{equation*}
     \alpha_{4}=\big(d_{1,0}^{2}-4d_{1,1}d_{1,-1}\big)t^{2}\quad
     \text{and}\quad \beta_4  =\big(d_{0,1}^{2}-4d_{1,1}d_{-1,1}\big)t^{2}
\end{equation*} 
are independent of $t\in (0,1)$. 
\begin{propo}[{\cite[\S  2.4.1, especially Proposition 2.4.3]{DuistQRT},\cite[Lemma 1.4]{DHRS2}}]\label{prop:genuscurvewalk}
Assume that the walk is not degenerate. The following facts are equivalent:
\begin{enumerate}
\item The algebraic curve $\Etproj$ is a genus one curve with no singularity\footnote{Let us recall the definition of singularities of $\Etproj$ in the affine chart $\C^{2}$. We may extend the notion of singularity in the other affine charts similarly. The point $P\in\Etproj \bigcap \C^{2}$ is a singularity if and only if the following partial derivatives both vanish at $P$:
\begin{equation*}
     \frac{\partial \overline{K}(x,1,y,1;t)}{\partial x}\quad \text{and}\quad \frac{\partial \overline{K}(x,1,y,1;t)}{\partial y}.
\end{equation*}
For instance, if $P=([a\!:\!1],[b\!:\!1])\in \Etproj$ with $ab\neq 0$, then $P$ is a singularity if and only if 
\begin{equation*}
     b -t \sum_{i=1,\,j=0}^2 id_{i-1,j-1} a^{i-1} b^j=a -t \sum_{i=0,\,j=1}^2 jd_{i-1,j-1} a^{i} b^{j-1}=0.
\end{equation*} 
}, i.e., $\Etproj$ is an elliptic curve;
\item The discriminant $\Delta^x_{[x_0:x_1]}$ has simple roots in $\mathbb{P}^{1}(\C)$;
\item The discriminant $\Delta^y_{[y_0: y_1]}$ has simple roots in $\mathbb{P}^{1}(\C)$.
\end{enumerate}
Otherwise, $\Etproj$ is a genus zero algebraic curve with exactly one singularity.
\end{propo}




When $t$ is transcendental over $\Q(d_{i,j})$, Lemma \ref{lem:doublezero}  provides a simple, geometric characterization of genus zero and genus one algebraic curves. This result will be generalized for every values of $t\in (0,1)$ in Proposition~\ref{coro1}.

\begin{lemma}[\cite{DHRS2}, Lemma 1.5]\label{lem:doublezero}  
Assume that $t$ is transcendental over $\Q(d_{i,j})$. A walk whose discriminant $\Delta^y_{[y_0: y_1]}$ has double roots is a model whose steps are supported in one of the following configurations:
$$\begin{tikzpicture}[scale=.4, baseline=(current bounding box.center)]
\foreach \x in {-1,0,1} \foreach \y in {-1,0,1} \fill(\x,\y) circle[radius=2pt];
\draw[thick,->](0,0)--(0,1);
\draw[thick,->](0,0)--(1,1);
\draw[thick,->](0,0)--(1,0);
\draw[thick,->](0,0)--(1,-1);
\draw[thick,->](0,0)--(0,-1);
\end{tikzpicture}\quad\quad \begin{tikzpicture}[scale=.4, baseline=(current bounding box.center)]
\foreach \x in {-1,0,1} \foreach \y in {-1,0,1} \fill(\x,\y) circle[radius=2pt];
\draw[thick,->](0,0)--(-1,1);
\draw[thick,->](0,0)--(0,1);
\draw[thick,->](0,0)--(1,1);
\draw[thick,->](0,0)--(1,0);
\draw[thick,->](0,0)--(1,-1);
\end{tikzpicture}\quad\quad
\begin{tikzpicture}[scale=.4, baseline=(current bounding box.center)]
\foreach \x in {-1,0,1} \foreach \y in {-1,0,1} \fill(\x,\y) circle[radius=2pt];
\draw[thick,->](0,0)--(-1,1);
\draw[thick,->](0,0)--(0,1);
\draw[thick,->](0,0)--(1,1);
\draw[thick,->](0,0)--(-1,0);
\draw[thick,->](0,0)--(1,0);
\end{tikzpicture}\quad\quad \begin{tikzpicture}[scale=.4, baseline=(current bounding box.center)]
\foreach \x in {-1,0,1} \foreach \y in {-1,0,1} \fill(\x,\y) circle[radius=2pt];
\draw[thick,->](0,0)--(1,1);
\draw[thick,->](0,0)--(1,0);
\draw[thick,->](0,0)--(-1,-1);
\draw[thick,->](0,0)--(0,-1);
\draw[thick,->](0,0)--(1,-1);
\end{tikzpicture}\quad\quad\begin{tikzpicture}[scale=.4, baseline=(current bounding box.center)]
\foreach \x in {-1,0,1} \foreach \y in {-1,0,1} \fill(\x,\y) circle[radius=2pt];
\draw[thick,->](0,0)--(-1,0);
\draw[thick,->](0,0)--(1,0);
\draw[thick,->](0,0)--(-1,-1);
\draw[thick,->](0,0)--(0,-1);
\draw[thick,->](0,0)--(1,-1);
\end{tikzpicture}\quad\quad\begin{tikzpicture}[scale=.4, baseline=(current bounding box.center)]
\foreach \x in {-1,0,1} \foreach \y in {-1,0,1} \fill(\x,\y) circle[radius=2pt];
\draw[thick,->](0,0)--(-1,1);
\draw[thick,->](0,0)--(1,1);
\draw[thick,->](0,0)--(-1,0);
\draw[thick,->](0,0)--(0,1);
\draw[thick,->](0,0)--(-1,-1);
\end{tikzpicture}\quad\quad\begin{tikzpicture}[scale=.4, baseline=(current bounding box.center)]
\foreach \x in {-1,0,1} \foreach \y in {-1,0,1} \fill(\x,\y) circle[radius=2pt];
\draw[thick,->](0,0)--(-1,1);
\draw[thick,->](0,0)--(0,1);
\draw[thick,->](0,0)--(-1,0);
\draw[thick,->](0,0)--(-1,-1);
\draw[thick,->](0,0)--(0,-1);
\end{tikzpicture}\quad\quad\begin{tikzpicture}[scale=.4, baseline=(current bounding box.center)]
\foreach \x in {-1,0,1} \foreach \y in {-1,0,1} \fill(\x,\y) circle[radius=2pt];
\draw[thick,->](0,0)--(-1,1);
\draw[thick,->](0,0)--(-1,0);
\draw[thick,->](0,0)--(-1,-1);
\draw[thick,->](0,0)--(0,-1);
\draw[thick,->](0,0)--(1,-1);
\end{tikzpicture}$$
Conversely, for any $t\in (0,1)$, if the steps are supported by one of the above configurations, then $\Delta^y_{[y_0: y_1]}$ has a double root.
\end{lemma}

When the $d_{i,j}$ form a configuration different to those listed in Lemma \ref{lem:doublezero}, it might be possible that for some values of $t$ algebraic over $\Q(d_{i,j})$, $\Etproj$ has genus zero, while for generic values, $\Etproj$ has genus one. Similarly to what occurs in the degenerate/nondegenerate situation (Proposition~\ref{prop:singcases}), we will see that these \textit{critical} values of $t$ are outside the interval $(0,1)$, as shows the following:

\begin{propo}\label{coro1}
The following dichotomy holds:
\begin{itemize}
     \item The walk is nondegenerate and, for all $t\in(0,1)$, $\Etproj$ is an elliptic curve;
     \item The steps are supported in one of the configurations of Lemma \ref{lem:doublezero}.
\end{itemize}
\end{propo}

\begin{rem}
Geometrically, the configurations such that the walk is nondegenerate and $\Etproj$ is an elliptic curve correspond to the situation where there are no three consecutive directions with weight zero, or equivalently, when the step set is not included in any half-space.\end{rem}

In the case $t=1$, it is proved in \cite[Lemma 2.3.10]{FIM} that, besides the models listed in Lemma \ref{lem:doublezero}, any nondegenerate model such that the drift is zero, i.e.,
\begin{equation*}
     \textstyle(\sum_{i}id_{i,j},\sum_{j}j d_{i,j})=(0,0),
\end{equation*}
has a curve $\Etproj$ of genus $0$ (and in particular is not elliptic). See Remark \ref{rem:critical_point} for further related comments.

The heart of the proof of Proposition \ref{coro1} is the following theorem.

\begin{theo}
\label{thm:disczeroes}
Assume that the walk is not in a configuration described in Lemma \ref{lem:doublezero}. Then for all $t\in (0,1)$, the four roots of $\Delta^x_{[x_0:x_1]}(t)$ are in $\mathbb{P}^{1}(\R)$, are distinct, and depend continuously upon $t$. Furthermore, two of them (say $a_{1}(t),a_{2}(t)$) satisfy $-1<a_{1}(t)<  a_{2}(t)<1$ and the other two (namely $a_{3}(t),a_{4}(t)$) satisfy $1< \vert a_{3}(t)\vert,\vert a_{4}(t)\vert$.  Finally, 
\begin{itemize}
     \item $\alpha_{4}>0$ implies that $a_{3}(t),a_{4}(t)$ are both positive;
     \item $\alpha_{4}=0$ implies that one of $a_{3}(t),a_{4}(t)$ is positive and the other one is $[1\!:\!0]$;
     \item $\alpha_{4}<0$ implies that one of $a_{3}(t),a_{4}(t)$ is positive and the other one is negative.
\end{itemize}
The same holds for  $\Delta^y_{[y_0: y_1]}(t)$.
\end{theo}

\begin{rem}
\label{rem:ordering}
As in \cite{FIM} we choose to order the $a_{i}(t)$ in such a way that the cycle of $\mathbb{P}^{1}(\R)$ starting from $-1$ and going to $+\infty$, and then from $-\infty$ to $-1$, crosses the $a_{i}(t)$ in the order $a_{1}(t),a_{2}(t),a_{3}(t),a_{4}(t)$, see Figure \ref{fig:ordering}. We order the $b_{i}(t)$ in the same way.
\end{rem}

With Theorem \ref{thm:disczeroes} as an assumption, we can now provide a proof of Proposition \ref{coro1}.
\begin{proof}[Proof of Proposition \ref{coro1}]
Assume that the model is in a half-space configuration of Lemma \ref{lem:doublezero}. Then by Lemma~\ref{lem:doublezero}, $\Delta^y_{[y_0: y_1]}$ has a double root for every $t\in (0,1)$, and Proposition \ref{prop:genuscurvewalk} implies that  the walk is degenerate, or for every $t\in (0,1)$ the algebraic curve $\Etproj$ is not an elliptic curve.  

Conversely, assume that the walk is not in a configuration of Lemma \ref{lem:doublezero}. Then, the walk is not degenerate by Proposition \ref{prop:singcases}. By Theorem \ref{thm:disczeroes}, for every $t\in (0,1)$ the roots of the discriminant are distinct, and Proposition \ref{prop:genuscurvewalk} implies that for any $t\in (0,1)$ the algebraic curve $\Etproj$ is an elliptic curve. 
\end{proof}

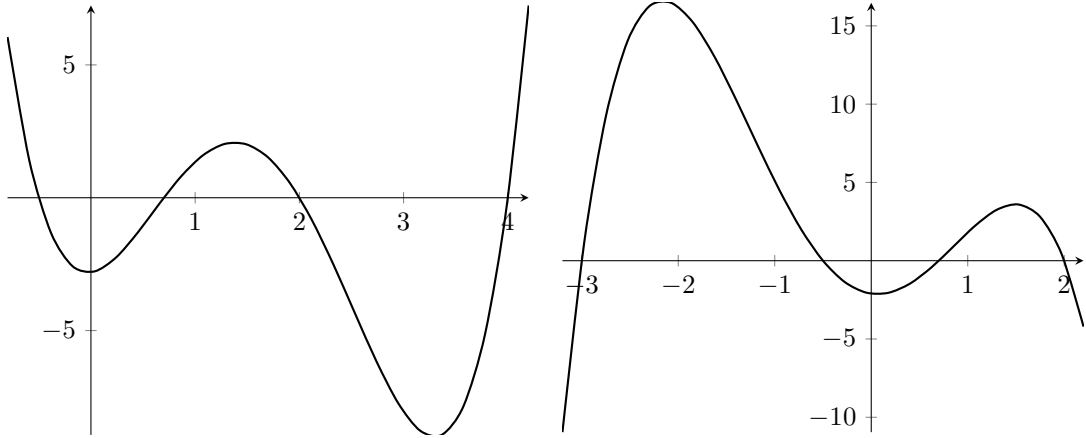
\begin{figure}
\begin{tikzpicture}
  \begin{axis} [axis lines=center]
    \addplot [domain=-0.8:4.2, smooth, thick] { x^4-6.2*x^3+8.85*x^2+.50*x-2.80 };
    \put(0,93){$a_1$}
    \put(47,93){$a_2$}
    \put(110,93){$a_3$}
    \put(176,93){$a_4$}
  \end{axis}
\end{tikzpicture}\quad
\begin{tikzpicture}
  \begin{axis} [axis lines=center]
    \addplot [domain=-3.2:2.2, smooth, thick] { -x^4-.8*x^3+6.55*x^2-.85*x-2.10 };
    \put(10,69){$a_4$}
    \put(79,69){$a_1$}
    \put(132,69){$a_2$}
    \put(174,69){$a_3$}
  \end{axis}
\end{tikzpicture}
\caption{Plot of $D(x)$ defined in \eqref{eq:expression_D_0} and ordering of $a_1,a_2,a_3,a_4$, see Remark~\ref{rem:ordering}. On the left, $a_4>0$ and on the right, $a_4<0$}
\label{fig:ordering}
\end{figure}
The rest of the subsection will be devoted to the proof of Theorem \ref{thm:disczeroes}. Let us do the proof for the discriminant $\Delta^y_{[y_0:y_1]}(t)$, the other case being clearly similar. Expanding the formula \eqref{eq:expression_branch_points_y}, $\Delta^y_{[y_0:y_1]}(t)$ is found to be equal to 
\begin{multline}
\label{eq:expanded_discriminant}
\big(d_{0,1}^{2}-4d_{1,1}d_{-1,1}\big)t^{2}y_{0}^{4}+\big(2t^{2}d_{0,1}d_{0,0}-2td_{0,1}-4t^{2}(d_{1,0}d_{-1,1}+d_{1,1}d_{-1,0})\big)y_{0}^{3}y_{1}\\+\big(1+t^{2}d_{0,0}^{2}+2t^{2}d_{0,-1}d_{0,1}-4t^{2}(d_{1,-1}d_{-1,1}+d_{1,0}d_{-1,0}+d_{1,1}d_{-1,-1})\big)y_{0}^{2} y_{1}^{2}\\
+\big( 2t^{2}d_{0,-1}d_{0,0}-2td_{0,-1}-4t^{2}(d_{1,-1}d_{-1,0}+d_{1,0}d_{-1,-1})\big) y_{0}y_{1}^{3}+
\big( d_{0,-1}^{2}-4d_{1,-1}d_{-1,-1}\big) t^{2}y_{1}^{4}.
\end{multline}
 
 Let us set 
\begin{equation}
\label{eq:def_B}
\left\{\begin{array}{rl}
     \displaystyle \widetilde{B}_{1}(y_{0},y_{1})\hspace{-2.5mm}&\displaystyle= d_{1,-1} y_1^2 + d_{1,0} y_0y_1 + d_{1,1}y_0^2,\\
     \displaystyle\mathbf{B}_{0}(y_{0},y_{1},t)\hspace{-2.5mm}&\displaystyle= d_{0,-1} y_1^2 -  \frac{1}{t} y_0y_1 +d_{0,0}y_0y_1+ d_{0,1}y_0^2,\\
    \displaystyle\widetilde{B}_{-1}(y_{0},y_{1})\hspace{-2.5mm}&\displaystyle=d_{-1,-1} y_1^2 + d_{-1,0} y_0y_1 + d_{-1,1}y_0^2.
\end{array}\right.
\end{equation}
With the above notations, the discriminant may be written 
\begin{align}
    \Delta^y_{[y_0:y_1]}(t)&=t^2 \mathbf{B}_{0}^2(y_{0},y_{1},t)- 4t^2\widetilde{B}_{1}(y_{0},y_{1})\,\widetilde{B}_{-1}(y_{0},y_{1}) \nonumber\\ 
   &=\left(t \mathbf{B}_{0}(y_{0},y_{1},t)+ 2t\sqrt{\widetilde{B}_{1}(y_{0},y_{1})\,\widetilde{B}_{-1}(y_{0},y_{1})}\right)\times \left(t \mathbf{B}_{0}(y_{0},y_{1},t)- 2t\sqrt{\widetilde{B}_{1}(y_{0},y_{1})\,\widetilde{B}_{-1}(y_{0},y_{1})}\right) \nonumber\\
&= \Delta^+(y_0,y_1,t)    \times \Delta^-(y_0,y_1,t),\label{eq:definition_Delta+-}
\end{align}
where the last equality is a definition. Obviously the zeros of the factors $\Delta^{\pm}(y_0,y_1,t)$ of \eqref{eq:definition_Delta+-} are roots of $\Delta^y_{[y_0:y_1]}(t)$. The innocent and trivial factorization \eqref{eq:definition_Delta+-} of the discriminant is the crucial point of the proof of Theorem \ref{thm:disczeroes}.

We are going to split the proof of Theorem \ref{thm:disczeroes} into several lemmas. Our strategy is to show that for every $t\in (0,1)$, $\Delta^{-}(y_0,y_1,t)$ has at least two distinct real zeros (resp.\ $\Delta^{+}(y_0,y_1,t)$ has at least two nonzero distinct real zeros), see Lemmas \ref{lem:roots_Delta+} and \ref{lem:roots_Delta-}, and that the four zeros obtained in this way are all distinct, see Lemma \ref{lem:Delta+NOTDelta-}. Figure \ref{fig:factors_a_4>0} represents an example of plot of $\Delta^-(y,1,t)$ and $\Delta^+(y,1,t)$, illustrating the preceding properties. 
 
\begin{lemma}
\label{lem:roots_Delta+}
Assume that the walk is not in a configuration described in Lemma~\ref{lem:doublezero}. For all $t\in (0,1)$, the factor $\Delta^+(y,1,t)$ has one zero in $(0,1)$ and one zero in $(1,+\infty)$, see Figure~\ref{fig:factors_a_4>0}. 
\end{lemma}

\begin{proof}
Let us first remark that both $\widetilde{B}_{1}(y,1)$ and $\widetilde{B}_{-1}(y,1)$ are nonnegative for ${y\in[0,+\infty)}$, which by \eqref{eq:definition_Delta+-} implies that $\Delta^+( y,1,t)$ and $\Delta^-( y,1,t)$ are real valued on $[0,+\infty)$. We shall now prove the following: for every ${t\in (0,1)}$, there exists $\varepsilon_{t}\in (0,1)$ such that
\begin{align} 
     \Delta^+(\varepsilon_{t},1,t)&>  0,\label{eq:limit_Delta+_0}\\
     \Delta^+(1,1,t)&<0,\label{eq:limit_Delta+_1}\\
     \lim_{y\to +\infty} \Delta^+( y,1,t)&=+\infty.\label{eq:limit_Delta+_inf}
\end{align}
Applying the mean value theorem to $\Delta^+( y,1,t)$ on $[\varepsilon_{t},+\infty)$ readily leads to Lemma~\ref{lem:roots_Delta+}. It thus remains to prove \eqref{eq:limit_Delta+_0}, \eqref{eq:limit_Delta+_1} and \eqref{eq:limit_Delta+_inf}.

The inequality \eqref{eq:limit_Delta+_0} is straightforward when $\Delta^+(0,1,t)=t(d_{0,-1}+2\sqrt{d_{-1,-1}d_{1,-1}})> 0$, since the function $y\mapsto \Delta^+(y,1,t)$ is continuous.

Assume now that ${\Delta^+(0,1,t)=0}$, in other words, that $0$ is a zero of $\Delta^+( y,1,t)$. In this situation, we shall prove the existence of another zero in $(0,1)$. We have ${d_{0,-1}+2\sqrt{d_{-1,-1}d_{1,-1}}=0}$, which implies $d_{0,-1}=d_{-1,-1}d_{1,-1}=0$. Assume that $d_{-1,-1}=0$. Then $d_{-1,0}$ and $d_{1,-1}$ are both nonzero since otherwise the walk would be in a configuration described in Lemma~\ref{lem:doublezero}. Then by \eqref{eq:def_B} and \eqref{eq:definition_Delta+-}, $\Delta^+(y,1,t)$ is equivalent at $y=0$ to $2t\sqrt{d_{-1,0}d_{1,-1}}\times \sqrt{y}$ with $2t\sqrt{d_{-1,0}d_{1,-1}}>0$. For every $t\in (0,1)$, there necessarily exists $\varepsilon_{t}\in (0,1)$ such that $y\mapsto \Delta^+(y,1,t)$ is strictly increasing in $[0,\varepsilon_{t})$. So $\Delta^+(\varepsilon_{t},1,t)> 0$, thereby proving \eqref{eq:limit_Delta+_0}. The case $d_{1,-1}=0$ is similar.

To prove \eqref{eq:limit_Delta+_1} we use the basic inequality of arithmetic and geometric means 
\begin{equation*}
     \sqrt{\widetilde{B}_{1}\,\widetilde{B}_{-1}}\leq \frac{\widetilde{B}_{1}+\widetilde{B}_{-1}}{2},
\end{equation*}
which eventually entails that
\begin{equation*}
     \Delta^+(1,1,t)=t \mathbf{B}_{0}(1,1,t)+ 2t\sqrt{\widetilde{B}_{1}(1,1)\,\widetilde{B}_{-1}(1,1)}\leq t\sum d_{i,j}-1<0.
\end{equation*}

Finally we prove our claim \eqref{eq:limit_Delta+_inf}. It is obvious when $d_{0,1}+2\sqrt{d_{-1,1}d_{1,1}}> 0$, as in this case $\Delta^+(y,1,t)$ is equivalent when $y\to+\infty$ to $t(d_{0,1}+2\sqrt{d_{-1,1}d_{1,1}})y^2$.
Since $d_{0,1}+2\sqrt{d_{-1,1}d_{1,1}}$ is nonnegative, the only nontrivial case is when the latter coefficient is zero. We must have $d_{0,1}=d_{-1,1}d_{1,1}=0$. If $d_{1,1}=0$ then $d_{-1,1}$ and $d_{1,0}$ are both nonzero, since otherwise the walk would be in one of the configurations of Lemma \ref{lem:doublezero}. Then $\widetilde{B}_{1}(y,1)\widetilde{B}_{-1}(y,1)$ has degree three with positive dominant term $d_{-1,1}d_{1,0}$, while $\mathbf{B}_{0}(y,1,t)$ has degree one, and  therefore \eqref{eq:limit_Delta+_inf} holds. The case $d_{-1,1}=0$ may be treated similarly. 
\end{proof}

\begin{figure}
\includegraphics[scale=0.35]{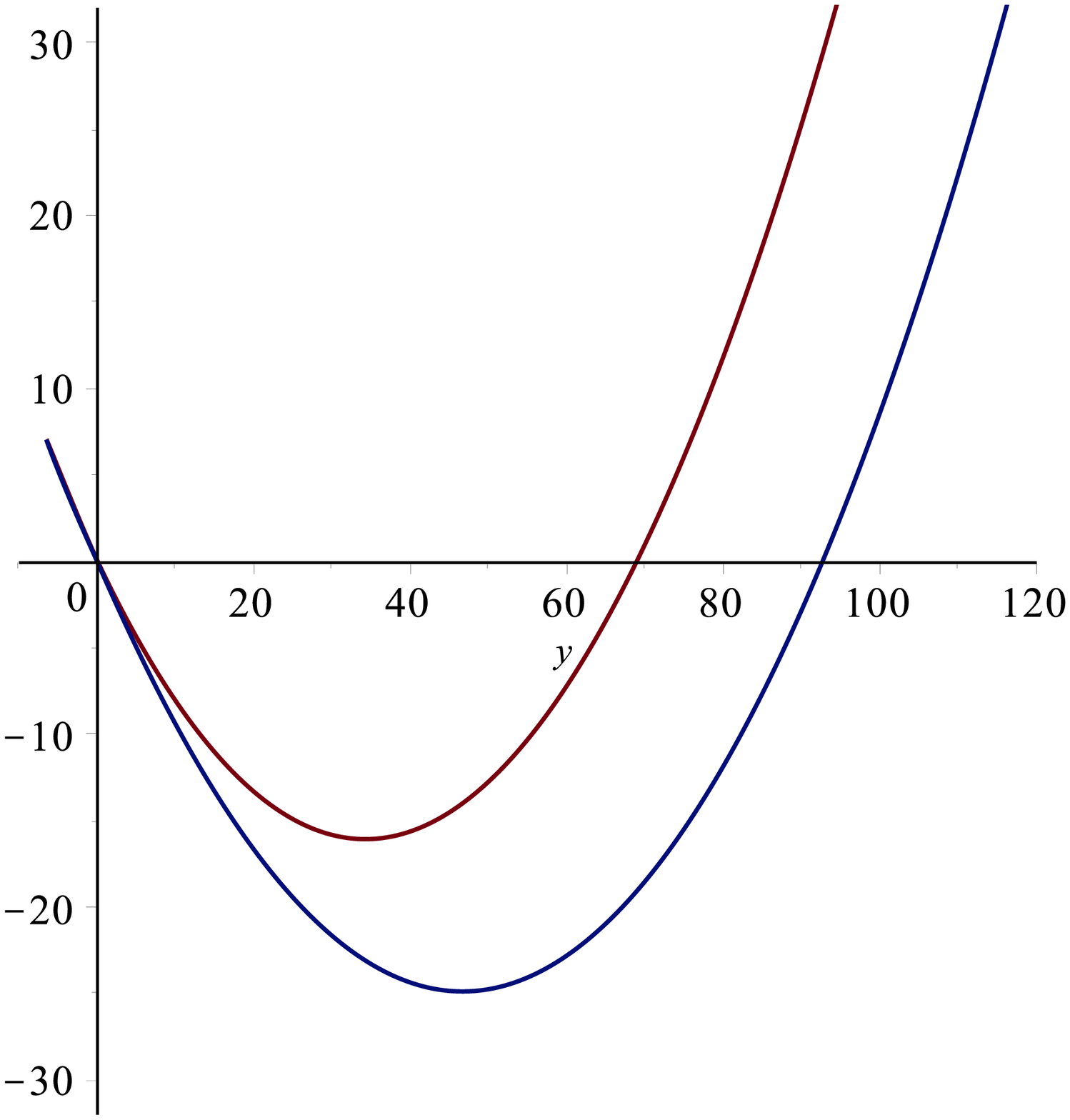}\qquad
\includegraphics[scale=0.35]{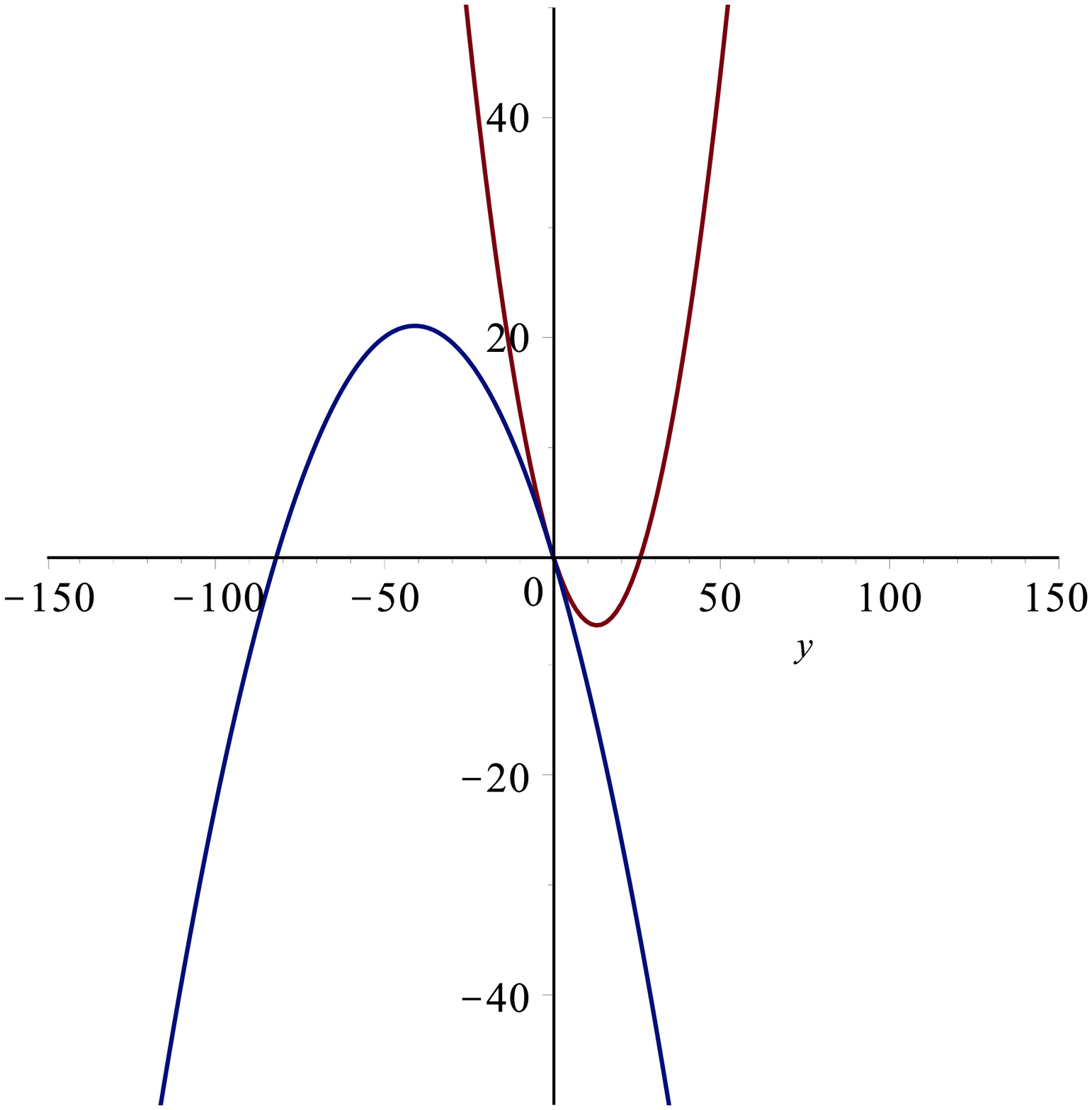}
\caption{Left: plots of $\Delta^-( y,1,t)$ (color red) and $\Delta^+( y,1,t)$ (color blue) in the case $\beta_4>0$. Right: same plots in the case $\beta_4<0$. In all cases, the functions admit one zero in $(-1,1)$ and one zero in $\vert y\vert>1$}
\label{fig:factors_a_4>0}
\end{figure}

\begin{lemma}\label{lem:roots_Delta-}
Assume that the walk is not in a configuration described in Lemma \ref{lem:doublezero}. For all $t\in (0,1)$, the factor $\Delta^-(y,1,t)$ has one zero in $(-1,1)$. Furthermore, 
\begin{itemize}
     \item If $\beta_{4}>0$, $\Delta^-(y,1,t)$ has one zero in $(1,+\infty)$;
     \item If $\beta_{4}=0$, $[1\!:\!0]$ is a zero of $\Delta^-(y_{0},y_{1},t)$;
     \item If $\beta_{4}<0$, $\Delta^-(y,1,t)$ has one zero in $(-\infty,-1)$, see Figure~\ref{fig:factors_a_4>0}. 
\end{itemize}
\end{lemma}
 
\begin{proof}
Though similar to that of Lemma \ref{lem:roots_Delta+}, the proof is more subtle for two reasons: first, different cases may occur according to the sign of $\beta_4$; second, the function $\Delta^-$ is not necessarily real on $(-\infty,0)$ and so one cannot directly apply the mean value theorem, as soon as negative values of $y$ are concerned.

Let us first consider the zero in $(-1,1)$. There are two cases, according to whether the condition below is satisfied or not:
\begin{equation}
\label{eq:cond_z}
     \exists z\in(-1,0]\text{ such that } \widetilde{B}_{1}(z,1) \widetilde{B}_{-1}(z,1)=0.
\end{equation}
We first assume that \eqref{eq:cond_z} is not satisfied. Remind that $\widetilde{B}_{1}$ and $\widetilde{B}_{-1}$ are positive on $[0,1]$. Then they are positive on $[-1,1]$, and $\Delta^-$ is well defined and real on $[-1,1]$. Let us prove the identities
\begin{align} 
     \Delta^-(-1,1,t)&>0,\label{eq:limit_Delta-_-1}\\
     \Delta^-(1,1,t)&<0.\label{eq:limit_Delta-_1}
\end{align}
The inequality of arithmetic and geometric means gives $-\sqrt{\widetilde{B}_{1}\,\widetilde{B}_{-1}}\geq -\frac{\widetilde{B}_{1}+\widetilde{B}_{-1}}{2}$, and finally entails that
\begin{equation*}
     \Delta^-(-1,1,t)=t \mathbf{B}_{0}(-1,1,t)- 2t\sqrt{\widetilde{B}_{1}(-1,1)\,\widetilde{B}_{-1}(-1,1)} \geq 1-t\sum_{i,j} d_{i,j}>0,
\end{equation*}
thereby proving \eqref{eq:limit_Delta-_-1}. On the other hand, $\widetilde{B}_{1}(1,1)$ and $\widetilde{B}_{-1}(1,1)$ are both positive and with \eqref{eq:limit_Delta+_1},
\begin{equation*}
     \Delta^-(1,1,t)\leq \Delta^+(1,1,t) < 0,
\end{equation*}
which shows \eqref{eq:limit_Delta-_1}. Then the mean value theorem easily implies that $\Delta^-(y,1,t)$ has one zero in $(-1,1)$.

Assuming now that the condition \eqref{eq:cond_z} is satisfied, let us consider the associated $z$ with biggest value, i.e., with smallest absolute value. The inequality \eqref{eq:limit_Delta-_-1} might not be true anymore (simply because $\Delta^-(-1,1,t)$ might be nonreal). Instead let us prove that
\begin{equation}
     \Delta^-(z,1,t)\geq0.\label{eq:limit_Delta-_z}
\end{equation}
Along the proof of Lemma \ref{lem:posneg}, we have seen that $\mathbf{B}_{0}(y,1,t)$ has no root in $(-\infty,0)$ and so has constant sign on this interval. If $d_{0,1}\neq 0$, then the quadratic polynomial $\mathbf{B}_{0}(y,1,t)$ has dominant term $d_{0,1}> 0$ and is therefore positive on $(-\infty,0)$. If $d_{0,1}=0$, then the affine function $\mathbf{B}_{0}(y,1,t)$ has dominant term $d_{0,0}-1/t < 0$ (we use $t\in (0,1)$ and $d_{0,0}\in [0,1]$) and is also positive on $(-\infty,0)$.
Therefore $\mathbf{B}_{0}(z,1,t)\geq0$. With $\widetilde{B}_{1}(z,1)\widetilde{B}_{-1}(z,1)=0$, we deduce \eqref{eq:limit_Delta-_z}. Note that \eqref{eq:limit_Delta-_1} still holds true in this context. By construction, $\widetilde{B}_{1}(z,1)\widetilde{B}_{-1}(z,1)\geq 0$ on $[z,0]\subset [z,1]$. Applying the mean value theorem we deduce that $\Delta^-(y,1,t)$ has one zero in $[z,1)\subset (-1,1)$.

We now focus on the other zero of $\Delta^-$. Consider first the subcase $\beta_{4}>0$. Recall that  by \eqref{eq:expanded_discriminant} the condition $\beta_{4}>0$ is equivalent to $d_{0,1}>2\sqrt{d_{1,1}d_{-1,1}}$. Therefore one obtains
\begin{equation} 
     \lim_{y\to +\infty} \Delta^-( y,1,t) =+\infty \quad\text{if } \beta_4>0.\label{eq:limit_Delta-_inf}
\end{equation}
Together with \eqref{eq:limit_Delta-_1}, the mean value theorem implies that for $\beta_{4}>0$, $\Delta^-( y,1,t)$ has at least one real zero in $(1,+\infty)$.

If $\beta_{4}=0$, $[1\!:\!0]$ is obviously a zero of $\Delta^-(y_{0},y_{1},t)$. If now $\beta_{4}<0$, i.e., 
\begin{equation}
\label{eq:beta_4<0_parameters}
    0 \leq  d_{0,1}<2\sqrt{d_{1,1}d_{-1,1}},
\end{equation}
then we must have $d_{1,1}d_{-1,1}\neq 0$, for otherwise \eqref{eq:beta_4<0_parameters} could not be true.
Then both $\widetilde{B}_{1}(y,1)$ and $\widetilde{B}_{-1}(y,1)$ have degree two with a positive dominant term. Therefore
\begin{equation*}
     \lim_{y\to -\infty} \widetilde{B}_{1}(y,1) \widetilde{B}_{-1}(y,1)=+\infty.
\end{equation*}
Consider now the condition 
\begin{equation}
\label{eq:cond_zbis}
     \exists z\leq0\text{ such that } \widetilde{B}_{1}(z,1) \widetilde{B}_{-1}(z,1)=0.
\end{equation}
Remind that  $\widetilde{B}_{1}(y,1) \widetilde{B}_{-1}(y,1)$ is strictly positive on $(0,+\infty)$. If there is no $z$ satisfying to \eqref{eq:cond_zbis}, then $\Delta^-$ is real on $\mathbb R$, and using \eqref{eq:beta_4<0_parameters} we easily compute
\begin{equation} 
     \lim_{y\to -\infty} \Delta^-( y,1,t) =-\infty\quad \text{if } \beta_4<0\label{eq:limit_Delta-_-inf}.
\end{equation}
We conclude by the mean value theorem applied to the function $\Delta^-$ on $(-\infty,-1)$, using further \eqref{eq:limit_Delta-_-1}.

If \eqref{eq:cond_zbis} holds true, let us consider the smallest $z$, i.e., with largest absolute value. Then \eqref{eq:limit_Delta-_z} is satisfied and we conclude by the same line of argument. The proof is complete.
 \end{proof}

\begin{lemma}
\label{lem:Delta+NOTDelta-}
Assume that the walk is not in a configuration described in Lemma \ref{lem:doublezero}. For all ${t\in (0,1)}$ and $[y_0 \!:\!y_{1}]\in\P1(\C)\setminus \{[0 \!:\!1] \}$ such that $\Delta^+(y_0,y_1,t)=0$,
\begin{equation}
\label{eq:Delta+NOTDelta-}
     \Delta^-(y_0,y_1,t)\neq \Delta^+(y_0,y_1,t).
\end{equation}
\end{lemma}
 
\begin{proof}
Thanks to \eqref{eq:limit_Delta+_inf}, $[1\!:\!0]$ cannot be a zero of the factor $\Delta^+(y_0,y_1,t)$. So we may assume that $[y_{0}\!:\!y_{1}]=[y\!:\!1]$ with $y\in \C^{*}$.
If $\Delta^-(y,1,t)=\Delta^+(y,1,t)$ for some $t$, then necessarily 
\begin{equation}
\label{eq:b=a=0_or_b=c=0}
     \mathbf{B}_{0}(y,1,t)=\widetilde{B}_{1}(y,1)=0\quad \text{or}\quad\mathbf{B}_{0}(y,1,t)=\widetilde{B}_{-1}(y,1)=0.
\end{equation} 
Lemma \ref{lem:posneg} yields $y=0$.  Thus for $t\in (0,1)$, $[y_{0}\!:\!y_{1}]\in\P1(\C)\setminus \{[0 \!:\!1] \}$, with ${\Delta^+(y_0,y_1,t)=0}$, we have $ \Delta^-(y_0,y_1,t)\neq \Delta^+(y_0,y_1,t)$. 
\end{proof}

We are now able to show Theorem \ref{thm:disczeroes}.
\begin{proof}[Proof of Theorem \ref{thm:disczeroes}]
By Lemmas \ref{lem:roots_Delta+} and \ref{lem:roots_Delta-}, $\Delta^{-}(y_0,y_1,t)$ has at least two distinct real zeros (resp.\ $\Delta^{+}(y_0,y_1,t)$ has at least two distinct nonzero real zeros). With Lemma \ref{lem:Delta+NOTDelta-}, the four zeros are distinct. Going back to Lemmas \ref{lem:roots_Delta+} and \ref{lem:roots_Delta-}, we obtain that they are located as stated in Theorem \ref{thm:disczeroes}.
\end{proof}

\begin{rem}
\label{rem:critical_point}
If the weights are not supported in any half-space (or equivalently, if the steps are not in a configuration of Lemma \ref{lem:doublezero}), then the step set polynomial $S(x,y)=\sum d_{i,j}x^{i}y^{j}$ defined in \eqref{eq:def_S} has a unique global minimizing point $(x_0,y_0)$ on $\mathbb R_{>0}^2$, which satisfies
\begin{equation*}
     \frac{\partial S}{\partial x}(x_0,y_0)=\frac{\partial S}{\partial y}(x_0,y_0)=0,
\end{equation*}
see \cite[Theorem 4]{BRS}. (This point $(x_0,y_0)$ is important, in particular $S(x_0,y_0)$ is related to the exponential growth of the total number of excursions by \cite[Theorem 4]{BRS}.) Obviously one has $S(x_0,y_0)\leq S(1,1)=1$, as $(x_0,y_0)$ is a global minimizer. Taking 
\begin{equation*}
     t=t_0=\frac{1}{S(x_0,y_0)}\geq 1,
\end{equation*}
the point $(x_0,y_0)$ is a singularity of $\Etproj$ (for $t=t_0$), and $\Etproj$ becomes of genus $0$, see Proposition~\ref{prop:genuscurvewalk}.

We conjecture that $t_0$ is the smallest value of $t>0$ for which $\Etproj$ has genus $0$. This holds true in case of a zero drift, for which $t_0=1$ and two branch points are equal to $1$ by \cite[Chapter~6]{FIM}.
\end{rem}
 
\begin{assumption}
\label{assumption} 
From now we assume that the model is nondegenerate and that $\Etproj$ is an elliptic curve. 
\end{assumption}
Recall that Proposition \ref{coro1} gives a very simple characterization of walks satisfying Assumption~\ref{assumption}.

\subsection{Group of the walk}

Following  \cite[Chapter 2]{FIM}, \cite[Section 3]{BMM} or \cite[Section~3]{KauersYatchak}, we attach to any model its group, which by definition is the group $\langle i_{1},i_{2} \rangle$ generated by the involutive birational transformations of $\P1(\C)^{2}$ given by 
\begin{equation}
\label{eq:generators_group}
     \iota_1(x,y) =\left(\frac{x_{0}}{x_{1}}, \frac{A_{-1}(\frac{x_{0}}{x_{1}}) }{A_{1}(\frac{x_{0}}{x_{1}})\frac{y_{0}}{y_{1}}}\right)\quad \text{and} \quad \iota_2(x,y)=\left(\frac{B_{-1}(\frac{y_{0}}{y_{1}})}{B_{1}(\frac{y_{0}}{y_{1}})\frac{x_{0}}{x_{1}}},\frac{y_{0}}{y_{1}}\right),
\end{equation}
with our notation \eqref{eq:def_S}.

\begin{figure}
\begin{center}
\includegraphics[scale=0.3]{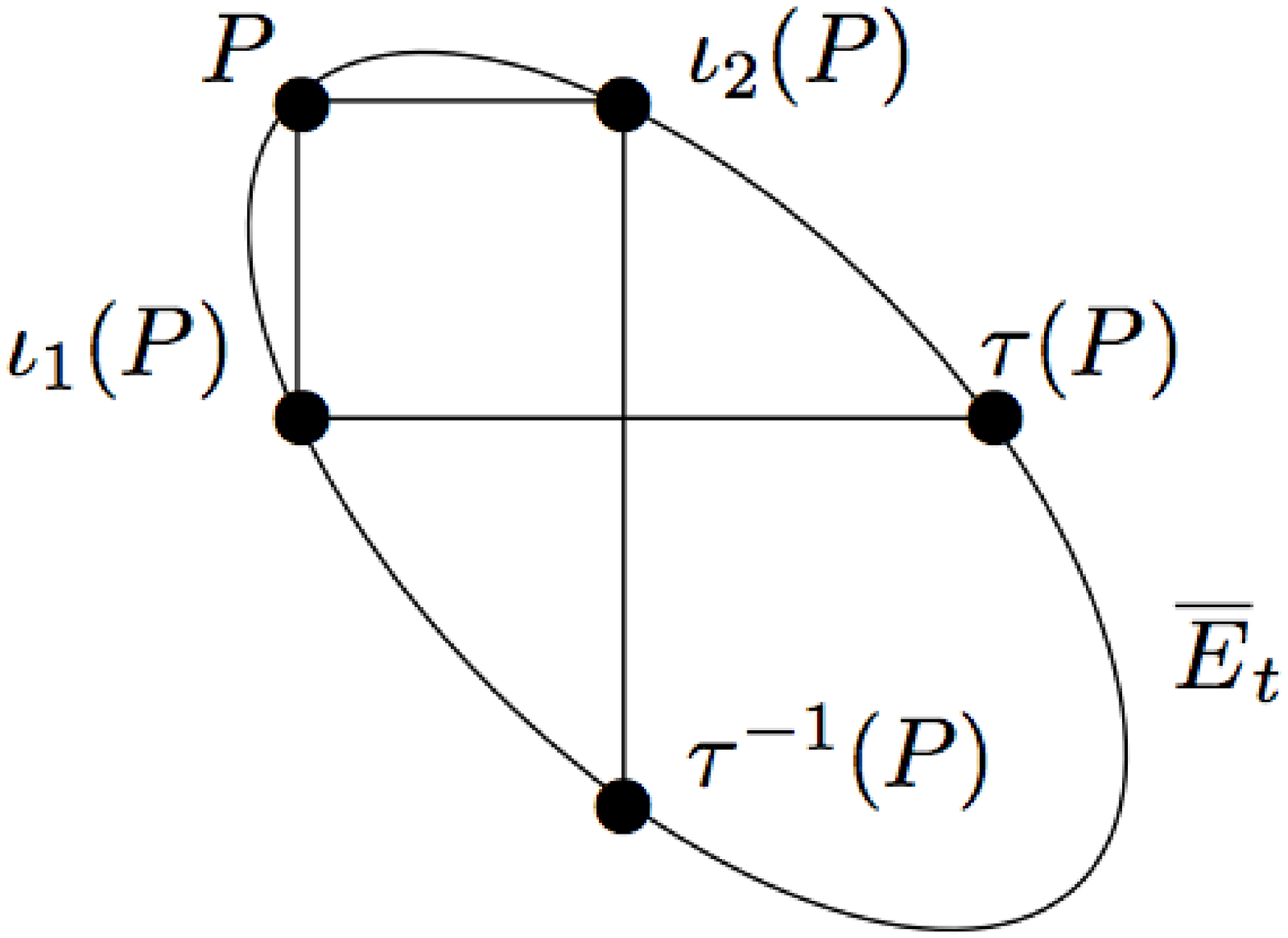}
\end{center}
\caption{The maps $\iota_{1}$ and $\iota_{2}$ restricted to the kernel curve $\Etproj $}\label{figiota}
\end{figure}
The kernel curve $\Etproj$ is left invariant by the natural action of this group, see \eqref{eq:expression_Et}. For a fixed value of $x$, there are at most  two possible values of $y$ such that $(x,y)\in\Etproj$. The involution $\iota_1$ corresponds to interchanging these values, see Figure \ref{figiota}. A similar interpretation can be given for $\iota_2$. 

Let us finally define
\begin{equation*}
     \tau=\iota_2 \circ \iota_1. 
\end{equation*}
Note that such a map is known as a QRT-map and has been widely studied, see \cite{DuistQRT}. As we will see later,  the algebraic nature of the series \eqref{eq:GF}, according to the classes depicted in \eqref{eq:classes}, highly depends on the fact that $\tau$ has finite or infinite order. Note that when the $d_{i,j}$ are fixed, the cardinality of the group on $\Etproj$ depends upon $t$, see \cite{fayolleRaschel} for concrete examples.

\section{Analytic continuation of the generating functions}
\label{sec:uniformization}

The goal of this section is to prove that $F^{1}(x;t)$ and $F^{2}(y;t)$ defined in \eqref{eq:funcequ} may be continued into multivalued meromorphic functions on the elliptic curve $\Etproj$. We are going to use a uniformization of $\Etproj$ via the Weierstrass elliptic function in order to see the multivalued functions as univalued meromorphic functions on $\C$. The starting point is that for $t$ fixed in $(0,1)$, the generating function $Q(x,y;t)$ is analytic for $\vert x\vert,\vert y\vert<1$, see Section \ref{sec:kernel}.  

\subsection{Uniformization of the elliptic curve $\Etproj$}
\label{subsec:explicit_parametrization}

Recall that Assumption~\ref{assumption} holds: the walk is nondegenerate and $\Etproj$ is an elliptic curve. This implies (Proposition \ref{prop:genuscurvewalk}) that the discriminants $\Delta^x_{[x_0:x_1]}$ and $\Delta^y_{[y_0:y_1]}$  have four distinct zeros. 

Since $\Etproj$ is an elliptic curve, we can identify $\Etproj$ with $\C/(\Z\omega_1 + \Z\omega_2)$, with $(\omega_1,\omega_2)\in \C^{2}$ basis of a lattice, via the $(\Z\omega_1 + \Z\omega_2)$-periodic map  
\begin{equation}
\label{eq:Lambda}
\begin{array}{llll}
\Lambda :& \C& \rightarrow &\overline{E}_t\\
 &\omega &\mapsto& (\mathfrak{q}_1(\omega), \mathfrak{q}_2(\omega)),
\end{array}
\end{equation}
where $\mathfrak{q}_1, \mathfrak{q}_2$ are rational functions of $\wp$ and its derivative $d\wp/d\omega$, and $\wp$ is the Weierstrass function associated with the lattice $\Z\omega_1 + \Z\omega_2$:
\begin{equation}
\label{eq:expression_wp_expanded}
     \wp(\omega )=\wp(\omega;\omega_1,\omega_2):=\frac{1}{\omega^{2}}+ \sum_{(\ell_{1},\ell_{2}) \in \Z^{2}\setminus \{(0,0)\}} \left(\frac{1}{(\omega +\ell_{1}\omega_{1}+\ell_{2}\omega_{2})^{2}} -\frac{1}{(\ell_{1}\omega_{1}+\ell_{2}\omega_{2})^{2}}\right).
\end{equation}     
Then the field of meromorphic functions on $\Etproj$ may be identified with the field of meromorphic functions on $\C/(\Z\omega_1 + \Z\omega_2)$, i.e., the field of meromorphic functions on $\C$ that are $(\omega_{1},\omega_{2})$-periodic (or elliptic). Classically, this latter field is equal to $\C(\wp, \wp')$, see \cite{WW}.
 
The goal of this subsection is to give explicit expressions for $\mathfrak{q}_1(\omega)$, $\mathfrak{q}_2(\omega)$, $\omega_{1}$ and $\omega_{2}$. Such computations have been already performed in \cite[Section 3.3]{FIM} when $t=1$, and in \cite{RaschelJEMS} in the unweighted case for $t\in(0,1/\vert\mathcal S\vert)$.
 
The maps $\iota_{1}$, $\iota_{2}$ and $\tau$ may be lifted to the $\omega$-plane. We will call them $\iup_{1}$, $\iup_{2}$ and $\widetilde{\tau}$, respectively. So we have the commutative diagrams 
\begin{equation*}
\xymatrix{
    \Etproj  \ar@{->}[r]^{\iota_k} & \Etproj  \\
    \C \ar@{->}[u]^\Lambda \ar@{->}[r]_{\iup_k} & \C \ar@{->}[u]_\Lambda 
  }
  \qquad\qquad\qquad
  \xymatrix{
    \Etproj  \ar@{->}[r]^{\tau} & \Etproj  \\
\C \ar@{->}[u]^\Lambda \ar@{->}[r]_{\widetilde{\tau}} & \C \ar@{->}[u]_\Lambda 
  } 
\end{equation*}
More precisely, following \cite{DuistQRT} (see in particular Proposition~2.5.2, Page 35 and Remark~2.3.8), there exist two complex numbers $\omega_{3}$ and $\omega_{4}$ such that 
\begin{equation}
\label{eq:expression_group_universal_cover}
     \iup_{1}(\omega)=-\omega+\omega_{4}, \quad\iup_{2}(\omega)=-\omega+\omega_{3}\quad \text{and} \quad\widetilde{\tau}(\omega)=\omega-\omega_{4}+\omega_{3}. 
\end{equation}
Up to a variable change of the form $\omega\mapsto \omega+\omega_{4}$, we may always reduce to the case $\omega_{4}=0$. So let us assume that $\omega_{4}=0$ 
in \eqref{eq:expression_group_universal_cover}. 
 
We have $\Lambda (\omega)=\Lambda(\omega+\omega_{1})=\Lambda(\omega+\omega_{2})$ for all $\omega\in \C$. This implies that for all $\ell_{1},\ell_{2}\in \Z$, $\iup_{2}$ may be replaced by $\omega\mapsto -\omega+\omega_{3}+\ell_{1}\omega_{1}+\ell_{2}\omega_{2}$. This shows that $\omega_{3}$ is not uniquely defined: it is only defined modulo the lattice $\Z\omega_{1}+\Z\omega_{2}$. A suitable choice of $\omega_{1}$, $\omega_{2}$ and $\omega_{3}$ will be made in Proposition \ref{prop:uniformization} and Lemma \ref{lem:formula_omega_3}.
 
Using \eqref{eq:expression_branch_points_x} and \eqref{eq:expression_D_0}, the discriminant $D(x)$ can be written as 
\begin{equation*}
     D(x)=B^{2}(x)-4A(x)C(x),
\end{equation*} 
with
\begin{equation*}
     \left\{\begin{array}{rl}
A(x)&\hspace{-3mm}=t(d_{-1,1} + d_{0,1} x + d_{1,1}x^2),\smallskip\\
B(x)&\hspace{-3mm}=t  (d_{-1,0} -  \frac{1}{t} x +d_{0,0}x + d_{1,0}x^2),\smallskip\\
C(x)&\hspace{-3mm}=t(d_{-1,-1} + d_{0,-1} x + d_{1,-1}x^2).
\end{array}\right.
\end{equation*}
The link between the above notations and the $\widetilde{A}_{i}$'s used in Section \ref{sec:kernel} is $A=t\widetilde{A}_{1}$, $B=t\widetilde{A}_{0}-x$ and $C=t\widetilde{A}_{-1}$.

The following proposition is the adaptation of \cite[Lemma 3.3.1]{FIM} to our context. Remind that we have ordered the $a_{i}$'s in such a way that the cycle of $\mathbb{P}^{1}(\R)$ starting from $-1$ and going to $+\infty$, and then from $-\infty$ to $-1$, crosses the $a_{i}$'s in the order $a_{1},a_{2},a_{3},a_{4}$, see Figure \ref{fig:ordering} and Remark~\ref{rem:ordering}. Before stating Proposition \ref{prop:uniformization}, note the following convention: a path of integration in an integral of the form $\int_{a_{i}}^{a_{j}}$ is the real path from $a_{i}$ to $a_{j}$ if $a_{i}\leq a_{j}$, and otherwise is the union of the real paths from $a_{i}$ to $+\infty$ and from $-\infty$ to $a_{j}$. 

\begin{propo}
\label{prop:uniformization}
The elliptic curve $\Etproj$ in \eqref{eq:expression_Et} admits a uniformization of the form
\begin{equation*}
     \Etproj
     = \bigl\{(x(\omega),y(\omega)):\omega\in\C/(\Z\omega_1 + \Z\omega_2)\bigr\},
\end{equation*}
where $x(\omega)$ and $y(\omega)$ are given by (note that $z=2A(x)y+B(x)$)
\begin{equation*}
    \begin{array}{|l|c|c|}\hline 
&x(\omega)&z(\omega)\\\hline
a_{4}\neq [1\!:\!0]&\phantom{\Bigg\|}\left[a_{4}+\frac{D'(a_{4})}{\wp (\omega)-\frac{1}{6}D''(a_{4})}:1\right]\phantom{\Bigg\|}&\phantom{\Bigg\|}\left[\frac{D'(a_{4})\wp'(\omega)}{2(\wp(\omega)-\frac{1}{6}D''(a_{4}))^{2}}:1\right]\phantom{\Bigg\|}\\\hline 
a_{4}= [1\!:\!0]&\phantom{\Big\|}\left[\wp(\omega)-\alpha_{2}/3:\alpha_{3}\right]\phantom{\Big\|}&\phantom{\Big\|}\left[-\wp'(\omega):2\alpha_{3}\right]\phantom{\Big\|}\\\hline 
\end{array}
\end{equation*}
The above formulas use the Weierstrass elliptic function $\wp(\omega )=\wp(\omega;\omega_1,\omega_2)$ defined in \eqref{eq:expression_wp_expanded}.
Moreover, a suitable choice for the periods $(\omega_1,\omega_2)$ is given by the elliptic integrals
\begin{equation}
\label{eq:expression_omega_1_omega_2}
     \omega_{1}=\mathbf{i}\int_{a_{3}}^{a_{4}} \frac{dx}{\sqrt{\vert D(x)\vert}}\in \mathbf{i}\R_{>0}\quad\text{and}\quad
\omega_{2}=\int_{a_{4}}^{a_{1}} \frac{dx}{\sqrt{D(x)}}\in \R_{>0}.
\end{equation} 
\end{propo}

Before proving Proposition \ref{prop:uniformization}, let us do a series of remarks.

\smallskip

$\bullet$ Formula \eqref{eq:expression_omega_1_omega_2} is stated in \cite[Lemma 3.3.2]{FIM} for $t=1$. Note a small misprint in \cite[Lemma 3.3.2]{FIM}, namely a (multiplicative) factor of $2$ that should be $1$.

\smallskip

$\bullet$ Note that the expression \eqref{eq:expression_omega_1_omega_2} for the periods is not unique (however, it is unique up to a unimodular transform). Our motivation for doing this particular choice is to have a real positive period (namely, $\omega_2$) and a purely imaginary one ($\omega_1$).

\smallskip

$\bullet$ As we shall see along the proof, the elliptic curve $\Etproj$ admits the classical Weierstrass canonical form
\begin{equation}
\label{eq:Weierstrass_canonical_form} 
     \Etproj=\{(\mathbf{u},v)\in\P1(\C)^2:v^{2}=4\mathbf{u}^{3}-g_{2}\mathbf{u}-g_{3}\},
\end{equation}
with real invariants $g_2,g_3$ defined in \eqref{eq:expression_g2g3_case1} (resp.\ \eqref{eq:expression_g2g3_case0}) in the case $a_{4}\neq [1\!:\!0]$ (resp.\ $a_{4}=[1\!:\!0]$). Furthermore, the discriminant of $4\mathbf{u}^{3}-g_{2}\mathbf{u}-g_{3}$ is positive.

\smallskip

$\bullet$ Obviously, replacing $z$ by $-z$ in the statement of Proposition \ref{prop:uniformization} gives another equivalent uniformization, see \eqref{eq:Riemann_surface_sqrt}. We have arbitrarily fixed the sign of $z$ by demanding that $z$ and $\wp'$ should have the same sign, as in \cite[Lemma 3.3.1]{FIM}.

\smallskip

$\bullet$ An alternative, more symmetric formula for the second coordinate of the uniformization $y(\omega)$ will be given in \eqref{eq6}.

\smallskip 

$\bullet$ Remind that $\omega_3$ is defined modulo the lattice. In what follows, with the above expressions \eqref{eq:expression_omega_1_omega_2} of $\omega_{1}$ and $\omega_{2}$, we will fix $\omega_3$ in the fundamental parallelogram $[0,\omega_1)+[0,\omega_2)$. Remark that an explicit formula for $\omega_3$ will be obtained in \eqref{eq:expression_omega_3}.

\begin{proof}[Proof of Proposition \ref{prop:uniformization}]
Letting $z=2A(x)y+B(x)$, the equality $\overline{K}(x,y;t)=0$ can easily be reformulated as 
\begin{equation}
\label{eq:Riemann_surface_sqrt}
     z^{2}=D(x).
\end{equation}     

Assume first that $a_{4}= [1\!:\!0]$. We have $\alpha_{4}=0$ in \eqref{eq:expression_D_0} and thus $\alpha_{3}\neq 0$ (otherwise $[1\!:\!0]$ would be a double zero of the discriminant, contradicting our assumption). In this case we perform the changes of variable
\begin{equation*}
     x=\frac{\mathbf{u}-\alpha_{2}/3}{\alpha_{3}}\quad \text{and} \quad z=\frac{-v}{2\alpha_{3}}
\end{equation*}
in order to recover the Weierstrass canonical form \eqref{eq:Weierstrass_canonical_form}. We just have to set $\mathbf{u}=\wp$ and $v=\wp '$ to obtain the result announced in Proposition \ref{prop:uniformization}.
In this case the invariants are given by
\begin{equation}
\label{eq:expression_g2g3_case0}
     g_2=\frac{4}{3}\alpha_2^2-4\alpha_1\alpha_3\quad\text{and}\quad g_3=-\frac{8}{27}\alpha_2^3+\frac{4}{3}\alpha_1\alpha_2\alpha_3-4\alpha_0\alpha_3^2.
\end{equation}
In particular, $g_2$ and $g_3$ are clearly polynomial functions of $t$.

Assume now that $a_{4}\neq [1\!:\!0]$. The main idea is to reduce to the case $a_{4}=[1\!:\!0]$ by performing a fractional linear transformation. The branch point $a_4$ is a simple zero of $D(x)$ and hence $D'(a_{4})\neq 0$. Introduce
\begin{equation*}
     u=\frac{D'(a_{4})}{x-a_{4}}\quad\text{and}\quad v=\frac{2zD'(a_{4})}{(x-a_{4})^{2}}.
\end{equation*}
Then the Taylor formula $D(x)=\sum_{j=1}^{4}(x-a_{4})^{j}\frac{D^{(j)}(a_{4})}{j!}$ allows us to express \eqref{eq:Riemann_surface_sqrt} as  
\begin{equation*}
     v^{2}=4u^{3}+2D''(a_{4})u^{2} +\frac{2u}{3}D^{(3)}(a_{4})D'(a_{4})+\frac{D^{(4)}(a_{4})D'(a_{4})^{2}}{6}.
\end{equation*}
Letting finally $\mathbf{u}=u+\frac{D''(a_{4})}{6}$, we obtain the Weierstrass canonical form \eqref{eq:Weierstrass_canonical_form}, where
\begin{multline}
\label{eq:expression_g2g3_case1}
   g_2=\frac{D''(a_{4})^2}{3}-\frac{2D'(a_{4})D^{(3)}(a_{4})}{3}\\ 
   \text{and}\quad g_3=-\frac{D''(a_{4})^3}{27}+\frac{D'(a_{4})D''(a_{4})D^{(3)}(a_{4})}{9}-\frac{D'(a_{4})^2D^{(4)}(a_{4})}{6}.
\end{multline}    
Again, we just have to set $\mathbf{u}=\wp$ and $v=\wp '$ to obtain the result.

In both cases the invariants $g_{2},g_{3}$ are real and the discriminant of $4\mathbf{u}^{3}-g_{2}\mathbf{u}-g_{3}$ is strictly positive (this follows from the fact that the $a_{i}$'s are real and distinct). Let $e_{1}>e_{2}>e_{3}$ be the three real roots of the discriminant. By \cite[Section 20.32, Example 1]{WW} we may choose $\omega_{1},\omega_{2}$ as follows (here we replace $\omega_{1}$ by $-\omega_{1}$ in order to have a period in $\mathbf{i}\R_{>0}$):
\begin{equation}
\label{eq:expression_periods_WW}
     \omega_1=2\mathbf{i}\int_{-\infty}^{e_3}\frac{d\mathbf{u}}{\sqrt{g_{2}\mathbf{u}+g_{3}-4\mathbf{u}^{3}}}\quad\text{and}\quad
     \omega_2=2\int_{e_1}^{+\infty}\frac{d\mathbf{u}}{\sqrt{4\mathbf{u}^{3}-g_{2}\mathbf{u}-g_{3}}}.
\end{equation}
In particular $\omega_2$ is real while $\omega_1$ is a pure imaginary.

Note that since  $4\mathbf{u}^{3}-g_{2}\mathbf{u}-g_{3}$ has three distinct real roots, one has by \cite[Section 20.32]{WW}
\begin{equation}
\label{eq:Weierstrass_factorization}
    e_1=\wp\left(\frac{\omega_2}{2}\right)>e_2=\wp\left(\frac{\omega_1+\omega_2}{2}\right)>e_3=\wp\left(\frac{\omega_1}{2}\right).
\end{equation}

We now prove that, with $x(\omega)$ as in the statement of Proposition \ref{prop:uniformization}, 
\begin{equation}
\label{eq:correspondence_ai_ei}
     x(0)=a_{4},\quad x\left(\frac{\omega_{1}}{2}\right)=a_{3},\quad x\left(\frac{\omega_{1}+\omega_{2}}{2}\right)=a_2,\quad x\left(\frac{\omega_{2}}{2}\right)=a_1.
\end{equation}
The correspondence in \eqref{eq:correspondence_ai_ei} is illustrated on Figure \ref{fig:The_fund_par}. Let us first note that the change of variable between the $x$-variable in \eqref{eq:Riemann_surface_sqrt} and the $\mathbf{u}$-variable in \eqref{eq:Weierstrass_factorization} is just
\begin{equation}
\label{eq:variable_change}
     x=\left\{\begin{array}{ll}
     \displaystyle a_{4}+\frac{D'(a_{4})}{\mathbf{u}-\frac{D''(a_{4})}{6}} & \text{when } a_{4}\neq [1\!:\!0],\medskip\\
     \displaystyle \frac{\mathbf{u}-\alpha_{2}/3}{\alpha_{3}} & \text{otherwise}.
     \end{array}\right.
\end{equation}
Remind that $a_{1},a_{2},a_{3},a_{4}$ are the roots of the discriminant. By construction  the roots of $\Delta^x_{[x_0:x_1]}$ correspond to the double roots of $y\mapsto \overline{K}(a_{j},y;t)$. Therefore, the $(a_{j},Y_{\pm}(a_{j}))$ are the fixed points by $\iota_{1}$. Since  $a_{1},a_{2},a_{3},a_{4}$ are distinct and $\{0,\frac{\omega_{1}}{2},\frac{\omega_{2}}{2},\frac{\omega_{1}+\omega_{2}}{2}  \}$ are the four distinct fixed points modulo $\Z\omega_1 + \Z\omega_2$ of  $\iup_{1}(\omega)=-\omega$, we find the equality of sets
\begin{equation*}
     \Lambda \left\{0,\frac{\omega_{1}}{2},\frac{\omega_{2}}{2},\frac{\omega_{1}+\omega_{2}}{2}\right\}=\bigl\{(a_{1},Y_{\pm}(a_{1})),(a_{2},Y_{\pm}(a_{2})),(a_{3},Y_{\pm}(a_{3})),(a_{4},Y_{\pm}(a_{4}))\bigr\}.
\end{equation*}
By construction we have $x(0)=a_{4}$.
 To pursue, note that
\begin{equation*}
     \alpha_3=2d_{1,0}(d_{0,0}-1/t)-4d_{1,1}d_{0,-1}-4d_{0,1}d_{1,-1},
\end{equation*} 
see \eqref{eq:expanded_discriminant}. Since the $d_{i,j}$'s are positive and $d_{0,0}-1/t< 0$ (we use $t\in (0,1)$), we obtain $\alpha_3 \leq 0$. 

In the case $a_{4}= [1\!:\!0]$ we must have $\alpha_{4}=0$ and $\alpha_{3}\neq 0$. Hence $\alpha_3 < 0$, and the function in \eqref{eq:variable_change} is decreasing. With \eqref{eq:Weierstrass_factorization} we conclude that 
\begin{equation*}
    x\left(\frac{\omega_{2}}{2}\right)<x\left(\frac{\omega_1+\omega_{2}}{2}\right)<x\left(\frac{\omega_{1}}{2}\right),
\end{equation*}
and thus \eqref{eq:correspondence_ai_ei} holds since $a_{1}<a_{2}<a_{3}<a_{4}=+\infty$, see Figure \ref{fig:ordering}, Theorem \ref{thm:disczeroes} and Remark~\ref{rem:ordering}.

Let us now consider the case $a_{4}\neq [1\!:\!0]$ and $\alpha_4>0$. 
With the assertion on the sign of $a_{3},a_{4}$ in Theorem \ref{thm:disczeroes}, we deduce that $a_{4}$ is the biggest root, see Figure~\ref{fig:ordering}. Hence $D'(a_4)>0$. Since all of $a_1,a_2,a_3$ are smaller than $a_4$ we must have, due to \eqref{eq:variable_change}, $e_i<\frac{D''(a_{4})}{6}$. The function \eqref{eq:variable_change} being decreasing for $\mathbf{u}\in(-\infty,\frac{D''(a_{4})}{6})$ and the $e_i$'s being ordered as in \eqref{eq:Weierstrass_factorization}, we deduce \eqref{eq:correspondence_ai_ei}, similarly to the case $a_{4}= [1\!:\!0]$.
 
Consider finally the case $a_{4}\neq [1\!:\!0]$ and $\alpha_4<0$. With the assertion on the sign of $a_{3},a_{4}$ in Theorem~\ref{thm:disczeroes}, we deduce that $a_{4}$ is the smallest root, see Figure~\ref{fig:ordering}. Hence $D'(a_4)>0$. Since all of $a_1,a_2,a_3$ are bigger than $a_4$ we must have, due to \eqref{eq:variable_change}, $e_i>\frac{D''(a_{4})}{6}$. The function \eqref{eq:variable_change} being decreasing for $\mathbf{u}\in(\frac{D''(a_{4})}{6},+\infty)$ and the $e_i$'s being ordered as in \eqref{eq:Weierstrass_factorization}, we deduce \eqref{eq:correspondence_ai_ei}.

\medskip
 
We now move to the proof of the formulas \eqref{eq:expression_omega_1_omega_2} for the periods. To that purpose, we perform in \eqref{eq:expression_periods_WW} the variable change \eqref{eq:variable_change}. We first assume that $a_{4}\neq [1\!:\!0]$. A straightforward computation shows that  
\begin{equation}\label{eq5}
D(x)=z^{2}=v^{2}\frac{(x-a_{4})^{4}}{4D'(a_{4})^{2}}=\frac{(x-a_{4})^{4}}{4D'(a_{4})^{2}}(4\mathbf{u}^{3}-g_{2}\mathbf{u}-g_{3}).
\end{equation}
In particular $\sqrt{4\mathbf{u}^{3}-g_{2}\mathbf{u}-g_{3}}=\frac{2D'(a_{4})}{(x-a_{4})^{2}}\sqrt{D(x)}$.
 Therefore with \eqref{eq:expression_periods_WW} $d\mathbf{u}=\frac{-D'(a_{4})}{(x-a_{4})^{2}}dx$, and with \eqref{eq:correspondence_ai_ei} we find
\begin{equation*}
 \omega_2=2\int_{\wp (\frac{\omega_2}{2})}^{+\infty}\frac{d\mathbf{u}}{\sqrt{4\mathbf{u}^{3}-g_{2}\mathbf{u}-g_{3}}}
 =\int_{a_{4}}^{a_1}\frac{dx}{\sqrt{D(x)}}.
\end{equation*}
 Similarly, $-4\mathbf{u}^{3}+g_{2}\mathbf{u}+g_{3}$ is positive for $\mathbf{u}\in(-\infty,e_{3})$ and 
\begin{equation*}
 \omega_1=2\mathbf{i}\int_{-\infty}^{\wp (\frac{\omega_1}{2})}\frac{d\mathbf{u}}{\sqrt{-4\mathbf{u}^{3}+g_{2}\mathbf{u}+g_{3}}}=2\mathbf{i}\int_{-\infty}^{\wp (\frac{\omega_1}{2})}\frac{d\mathbf{u}}{\sqrt{\vert 4\mathbf{u}^{3}-g_{2}\mathbf{u}-g_{3}\vert }}=\mathbf{i}\int_{a_3}^{a_{4}}\frac{dx}{\sqrt{\vert D(x)\vert }}.
\end{equation*}
The computations are very similar in the case $a_{4}= [1\!:\!0]$ and we omit them.
\end{proof}

\subsection{Further properties of the uniformization}
We start by studying the real and non-real points of $x(\omega)$ and $y(\omega)$.
\begin{lemma}
\label{lem:real_points}
The following holds:
\begin{itemize}
     \item $\omega\in\{\frac{\omega_{1}}{2}\Z+\omega_{2}\R\} \Longrightarrow x(\omega),y(\omega)\in \mathbb{P}^{1}(\R)$ (dashed line on Figure \ref{fig:The_fund_par});
     \item $\omega\in \{\omega_{1}\R+\frac{\omega_{2}}{2}\Z\}\Longrightarrow x(\omega)\in \mathbb{P}^{1}(\R)$ (dotted line on Figure \ref{fig:The_fund_par});
     \item $\omega\in \{\omega_{1}\R+\frac{\omega_{2}}{2}\Z\}\setminus \{\frac{\omega_{1}}{2}\Z+\frac{\omega_{2}}{2}\Z \}\Longrightarrow y(\omega)\notin \R$.
\end{itemize}
\end{lemma}
     \unitlength=0.52cm
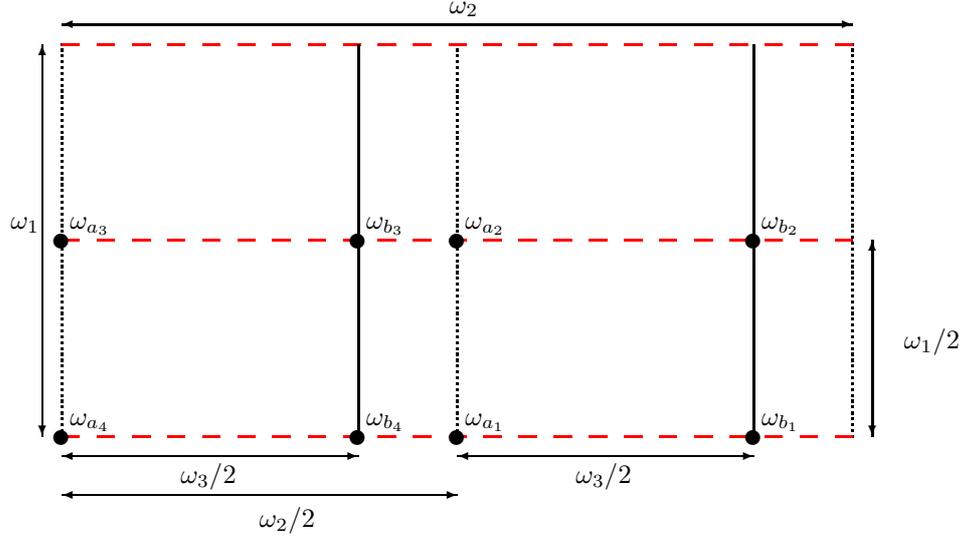
\begin{figure}[t]
    \vspace{5mm}
  \begin{center}
\begin{tabular}{cccc}
    \hspace{-7.3cm}

 \begin{tikzpicture}(19.55,6.5)
           \linethickness{2mm}
 \draw[draw=none,fill=black!00] (-5,5.2) -- (5.42,5.2) -- (5.42,0) -- (-5,0) -- cycle;
     \end{tikzpicture}

\hspace{-33mm}\begin{picture}(0,0)(0,0)
\thicklines
\dottedline{0.15}(-10,0)(-10,10)
\dottedline{0.15}(0,0)(0,10)
\dottedline{0.15}(10,0)(10,10)
\textcolor{red}{\dashline[30]{0.5}(-10,0)(10,0)
\dashline[30]{0.5}(-10,5)(10,5)
\dashline[30]{0.5}(-10,10)(10,10)}
\thinlines
\put(-2.5,0){\line(0,1){10}}
\put(7.5,0){\line(0,1){10}}
\put(-10.27,-0.27){{\LARGE$\bullet$}}
\put(-2.77,-0.27){{\LARGE$\bullet$}}
\put(-0.27,-0.27){{\LARGE$\bullet$}}
\put(7.23,-0.27){{\LARGE$\bullet$}}
\put(-10.27,4.73){{\LARGE$\bullet$}}
\put(-2.77,4.73){{\LARGE$\bullet$}}
\put(-0.27,4.73){{\LARGE$\bullet$}}
\put(7.23,4.73){{\LARGE$\bullet$}}
\put(-9.8,0.3){{$\omega_{a_{4}}$}}
\put(-2.3,0.3){{$\omega_{b_{4}}$}}
\put(0.2,0.3){{$\omega_{a_{1}}$}}
\put(7.7,0.3){{$\omega_{b_{1}}$}}
\put(-9.8,5.3){{$\omega_{a_{3}}$}}
\put(-2.3,5.3){{$\omega_{b_{3}}$}}
\put(0.2,5.3){{$\omega_{a_{2}}$}}
\put(7.7,5.3){{$\omega_{b_{2}}$}}
\put(-2.5,-0.5){\vector(-1,0){7.5}}
\put(-10,-0.5){\vector(1,0){7.5}}
\put(-7,-1.2){{$\omega_{3}/2$}}
\put(0,-0.5){\vector(1,0){7.5}}
\put(7.5,-0.5){\vector(-1,0){7.5}}
\put(3,-1.2){{$\omega_{3}/2$}}
\put(-10,10.5){\vector(1,0){20}}
\put(10,10.5){\vector(-1,0){20}}
\put(-0.2,10.8){{$\omega_{2}$}}
\put(-10.5,0){\vector(0,1){10}}
\put(-10.5,10){\vector(0,-1){10}}
\put(-11.3,5.3){{$\omega_{1}$}}

\put(10.5,0){\vector(0,1){5}}
\put(10.5,5){\vector(0,-1){5}}
\put(11.3,2.25){{$\omega_{1}/2$}}

\put(-10,-1.5){\vector(1,0){10}}
\put(0,-1.5){\vector(-1,0){10}}
\put(-5,-2.3){{$\omega_{2}/2$}}
\end{picture}
    \end{tabular}
  \end{center}
  \vspace{8mm}
\caption{Real points of $x(\omega)$ and $y(\omega)$ on the fundamental parallelogram. Here $\Lambda(\omega_{a_{i}})=(a_{i},Y_{\pm}(a_{i}))$ and $\Lambda(\omega_{b_{j}})=(X_{\pm}(b_{j}),b_{j})$}
\label{fig:The_fund_par}
\end{figure}

\begin{proof}
The proof follows from the (well-known) location of the real points of $\wp,\wp'$ on the fundamental parallelogram, and the location of the purely imaginary points of $\wp'$ when one period is real and the other one purely imaginary. In particular it is known \cite[Section 20.32, Example 2]{WW} that $\wp$ is real on the perimeter of the fundamental parallelogram (and on the half-perimeter as well).

Let $\overline{\omega}$ denote the complex conjugate number of $\omega\in \C$. Since by Proposition \ref{prop:uniformization} the period $\omega_{1}$ is purely imaginary and $\omega_{2}$ is real, we have for all $\omega\in \C$
\begin{equation}
\label{eq:conjug}
     \overline{\wp(\omega)}=\wp(\overline{\omega})\quad 
     \text{and} \quad \overline{\wp'(\omega)}=\wp'(\overline{\omega}),
\end{equation}
see \eqref{eq:expression_wp_expanded}. Moreover, once again by \eqref{eq:expression_wp_expanded}, we have
\begin{equation}
\label{eq:parity_Weierstrass}
     \wp(-\omega)=\wp(\omega)\quad\text{and} \quad \wp' (-\omega)=-\wp' (\omega).
\end{equation}

Let $\omega\in\frac{\omega_{1}}{2}\Z+\omega_{2}\R$. With \eqref{eq:parity_Weierstrass} and the $(\omega_{1},\omega_{2})$-periodicity, we get that
$\overline{\wp(\omega)}=\wp(\overline{\omega})=\wp(\omega)$ and $\overline{\wp'(\omega)}=\wp'(\overline{\omega})=\wp'(\omega)$. This shows that  $\wp(\omega),\wp'(\omega)\in \mathbb{P}^{1}(\R)$, and thereby proves the first item of Lemma \ref{lem:real_points}.

Let $\omega\in \omega_{1}\R+\frac{\omega_{2}}{2}\Z$. Using \eqref{eq:parity_Weierstrass} and the $(\omega_{1},\omega_{2})$-ellipticity, we get
$\overline{\wp(\omega)}=\wp(\overline{\omega})=\wp(-\omega)=\wp(\omega)$ and $\overline{\wp'(\omega)}=\wp'(\overline{\omega})=\wp'(-\omega)=-\wp'(\omega)$. As a consequence, $\wp(\omega)\in \mathbb{P}^{1}(\R)$ and $\wp'(\omega)\in \mathbf{i}\R\bigcup [1\!:\!0]$, and thus the second item of Lemma \ref{lem:real_points} is proved.

Using Proposition \ref{prop:uniformization}, we deduce that $x(\omega)\in \mathbb{P}^{1}(\R)$ for $\omega\in \{\omega_{1}\R+\frac{\omega_{2}}{2}\Z\}\bigcup\{\frac{\omega_{1}}{2}\Z+\omega_{2}\R\} $. Let us remind that the three zeros of $\wp'$ modulo $\omega_{1}\Z+\omega_{2}\Z$ are $\frac{\omega_{1}}{2},\frac{\omega_{2}}{2},\frac{\omega_{1}+\omega_{2}}{2}$  and that its unique (triple) pole is at $0$. So $\wp'(\omega)\in \mathbf{i}\R^{*}$ for $\omega$ that belongs to $\{\omega_{1}\R+\frac{\omega_{2}}{2}\Z\}\setminus \{\frac{\omega_{1}}{2}\Z+\frac{\omega_{2}}{2}\Z \}$ (third item of Lemma \ref{lem:real_points}). We then conclude with Lemma~\ref{prop:uniformization} that $\omega\in \{\frac{\omega_{1}}{2}\Z+\omega_{2}\R\}$ implies $z(\omega)\in \mathbb{P}^{1}(\R)$, and  $\omega\in \{\omega_{1}\R+\frac{\omega_{2}}{2}\Z\}\setminus \{\frac{\omega_{1}}{2}\Z+\frac{\omega_{2}}{2}\Z \}$ yields $z(\omega)\notin \mathbb{P}^{1}(\R)$. The result on $y(\omega)$ follows by combining the results on $x(\omega)$ and $z(\omega)$.
\end{proof}

Recall that $\omega_3$ is introduced in \eqref{eq:expression_group_universal_cover}. The statement hereafter is illustrated on Figure \ref{fig:The_fund_par}.
\begin{lemma}
\label{lem:correspondence_ai_ei}
The following holds:
\begin{itemize}
 \item $(x(0),y(0))=(a_{4},Y_{\pm}(a_{4}))$;
     \item $(x(\frac{\omega_{1}}{2}),y(\frac{\omega_{1}}{2}))=(a_{3},Y_{\pm}(a_{3}))$;
     \item $(x(\frac{\omega_{1}+\omega_{2}}{2}),y(\frac{\omega_{1}+\omega_{2}}{2}))=(a_{2},Y_{\pm}(a_{2}))$;
     \item $(x(\frac{\omega_{2}}{2}),y(\frac{\omega_{2}}{2}))=(a_{1},Y_{\pm}(a_{1}))$;
     \item $(x(\frac{\omega_{3}}{2}),y(\frac{\omega_{3}}{2}))=(X_{\pm}(b_{4}),b_{4})$;
     \item $(x(\frac{\omega_{1}+\omega_{3}}{2}),y(\frac{\omega_{1}+\omega_{3}}{2}))=(X_{\pm}(b_{3}),b_{3})$;
     \item$(x(\frac{\omega_{1}+\omega_{2}+\omega_{3}}{2}),y(\frac{\omega_{1}+\omega_{2}+\omega_{3}}{2}))=(X_{\pm}(b_{2}),b_{2})$;
     \item $(x(\frac{\omega_{2}+\omega_{3}}{2}),y(\frac{\omega_{2}+\omega_{3}}{2}))=(X_{\pm}(b_{1}),b_{1})$.
\end{itemize}
\end{lemma}

\begin{proof}
The statements for the $a_{i}$'s have been shown in the proof of Proposition \ref{prop:uniformization}. Those concerning the $b_i$'s are a priori unclear, as the coordinates $x(\omega)$ and $y(\omega)$ do not play a symmetric role in Proposition~\ref{prop:uniformization}. However, using the exact same ideas as in \cite[Equation (3.3)]{KurkRasch}, we can rewrite the formula $y(\omega)$ in Proposition \ref{prop:uniformization} more symmetrically, as follows:
\begin{equation}\label{eq6}
          y(\omega)=\left\{\begin{array}{ll}
     \displaystyle b_{4}+\frac{E'(b_{4})}{\wp(\omega-\omega_3/2)-\frac{E''(b_{4})}{6}} & \text{when } b_{4}\neq [1\!:\!0],\medskip\\
     \displaystyle \frac{\wp(\omega-\omega_3/2)-\beta_{2}/3}{\beta_{3}} & \text{otherwise},
     \end{array}\right.
\end{equation}
with the help of the notation \eqref{eq:expression_D_0}.
The result follows. 
\end{proof}

Let $\PPP$ be the counterclockwise oriented boundary of the half-parallelogram with vertices $0,\frac{\omega_{2}}{2}$, $\frac{\omega_{1}+\omega_{2}}{2},\frac{\omega_{1}}{2}$, i.e., the union of the four segments $[0,\omega_{2}/2]$, $[\omega_{2}/2,(\omega_{1}+\omega_{2})/2]$, $[(\omega_{1}+\omega_{2})/2,\omega_{1}/2]$ and $[\omega_{1}/2,0]$, see the left display on Figure \ref{fig:half_parallelogram}.
\begin{lemma}
\label{lem:half-parallelogram}
The function $\omega\mapsto x(\omega)$ is continuous and one-to-one from $\PPP$ to $\mathbb{P}^{1}(\R)$.
 \end{lemma}

\begin{proof}
It is most well known that $\wp$ is one-to-one from the boundary $\PPP$ of the half-parallelogram to $\mathbb{P}^{1}(\R)$. Indeed $\wp$ is real on $\PPP$ (Example 2 in \cite[20.32]{WW}) and goes from $+\infty$ to $-\infty$ when $\PPP$ is oriented counterclockwise. If $\wp$ was not strictly decreasing along $\PPP$ this would contradict the fact that $\wp$ has order $2$ (i.e., the fact that $\wp$ takes each value of $\mathbb{P}^{1}(\C)$ twice within a fundamental parallelogram).

To conclude, let us notice that $x(\omega)$ is a fractional linear transform with real coefficients of $\wp(\omega)$ (see Proposition~\ref{prop:uniformization}), thus the same result holds for $x$, thereby proving Lemma \ref{lem:half-parallelogram}.
\end{proof}

The following remark describes the behavior of $x(\omega)$ in the half-parallelogram; see also Figure~\ref{fig:half_parallelogram}.
\begin{rem}
\label{rem1}
Remind, see Remark \ref{rem:ordering}, that we have ordered the $a_{i}$'s in such a way that the cycle of $\mathbb{P}^{1}(\R)$ starting from $-1$ and going to $+\infty$, and then from $-\infty$ to $-1$, crosses the $a_{i}$ in the order $a_{1},a_{2},a_{3},a_{4}$, see Figure \ref{fig:ordering}. Remind also, see Theorem \ref{thm:disczeroes}, that $-1<a_{1}<  a_{2}<1$ and $1< \vert a_{3}\vert,\vert a_{4}\vert$.

The following is a byproduct of the proof of Lemma \ref{lem:half-parallelogram}: using the monotonicity of $x(\omega)$ along $\PPP$ shown in Lemma \ref{lem:half-parallelogram}, we have proven that for all $\omega\in [\omega_{2}/2,(\omega_{1}+\omega_{2})/2]$, $x(\omega)\in [a_{1},a_{2}]$ and then $\vert x(\omega)\vert<1$. Similarly, we have also proven that for all $\omega\in [0,\omega_{1}/2]$, $\vert x(\omega)\vert>1$. With the same arguments,  for all $\omega\in [\omega_{1}/2,(\omega_{1}+\omega_{2})/2]$, $x(\omega)\in [a_{2},a_{3}]$, and then there exists $\omega_{0}\in (\omega_{1}/2,(\omega_{1}+\omega_{2})/2)$ such that $x(\omega_{0})=1$, see Figure \ref{fig:half_parallelogram}. Similarly,  there exists $\omega_{0}'\in (0,\omega_{2}/2)$ such that $x(\omega_{0}')=-1$.

By \eqref{eq6}, a similar statement holds true for $y(\omega)$: for all $\omega\in [(\omega_{2}+\omega_{3})/2,(\omega_{1}+\omega_{2}+\omega_{3})/2]$, $\vert y(\omega)\vert<1$, and for all $\omega\in [\omega_{3}/2,(\omega_{1}+\omega_{3})/2]$, $\vert y(\omega)\vert>1$. 
\end{rem}

\unitlength=0.5cm
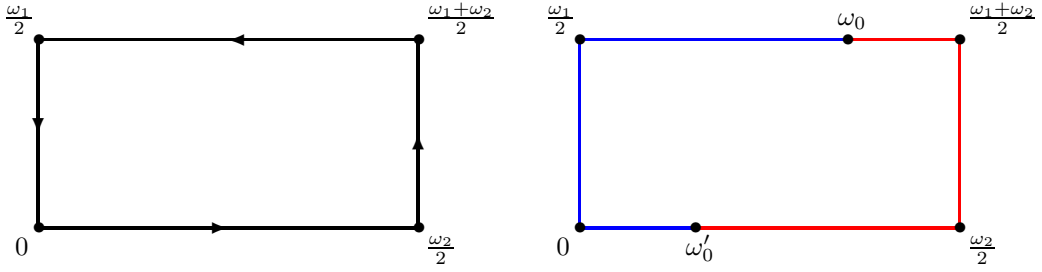
\begin{figure}[t]
    \vspace{29mm}
  \begin{center}\hspace{-60mm}
\begin{picture}(0,0)(0,0)
\thicklines
\put(0,0){\vector(1,0){5}}
\put(5,0){\line(1,0){5}}
\put(10,5){\vector(-1,0){5}}
\put(5,5){\line(-1,0){5}}
\put(0,5){\vector(0,-1){2.5}}
\put(0,2.5){\line(0,-1){2.5}}
\put(10,0){\vector(0,1){2.5}}
\put(10,2.5){\line(0,1){2.5}}
\put(-0.15,-0.15){$\bullet$}
\put(9.85,-0.15){$\bullet$}
\put(-0.15,4.85){$\bullet$}
\put(9.85,4.85){$\bullet$}
\put(-0.6,-0.7){$0$}
\put(10.2,-0.7){$\frac{\omega_2}{2}$}
\put(10.2,5.4){$\frac{\omega_1+\omega_2}{2}$}
\put(-0.9,5.4){$\frac{\omega_1}{2}$}
\end{picture}\hspace{70mm}
\begin{picture}(0,0)(0,0)
\thicklines
\put(0,0){\textcolor{blue}{\line(1,0){3}}}
\put(3,0){\textcolor{red}{\line(1,0){7}}}
\put(10,5){\textcolor{red}{\line(-1,0){3}}}
\put(7,5){\textcolor{blue}{\line(-1,0){7}}}
\put(0,5){\textcolor{blue}{\line(0,-1){5}}}
\put(10,0){\textcolor{red}{\line(0,1){5}}}
\put(-0.15,-0.15){$\bullet$}
\put(2.9,-0.15){$\bullet$}
\put(6.9,4.85){$\bullet$}
\put(9.85,-0.15){$\bullet$}
\put(-0.15,4.85){$\bullet$}
\put(9.85,4.85){$\bullet$}
\put(-0.6,-0.7){$0$}
\put(2.8,-0.7){$\omega_0'$}
\put(10.2,-0.7){$\frac{\omega_2}{2}$}
\put(6.8,5.4){$\omega_0$}
\put(10.2,5.4){$\frac{\omega_1+\omega_2}{2}$}
\put(-0.9,5.4){$\frac{\omega_1}{2}$}
\end{picture}
  \end{center}
  \vspace{5mm}
\caption{Left: orientation of the half-parallelogram $\PPP$. Right: location of $\omega_0$ and $\omega_0'$ on the half-parallelogram, such that $x(\omega_0)=1$ and $x(\omega_0')=-1$. On the blue part of $\PPP$, one has $\vert x(\omega)\vert\geq 1$ and on the red part, $\vert x(\omega)\vert\leq 1$}
\label{fig:half_parallelogram}
\end{figure}

We now compute an explicit expression for $\omega_{3}$.
\begin{lemma}
\label{lem:formula_omega_3}
One has
\begin{equation}
\label{eq:expression_omega_3}
     \omega_{3}=\int_{a_{4}}^{X_{\pm}(b_{4})} \frac{dx}{\sqrt{D(x)}}\in (0,\omega_{2}).
\end{equation}     
\end{lemma}

\begin{proof}
We first show that $\omega_3\in (0,\omega_2)$. By Theorem \ref{thm:disczeroes}, $b_{4}\in \mathbb{P}^{1}(\R)$. Since the real polynomial $x\mapsto K(x,b_{4};t)$ has a double root and $b_{4}\in \mathbb{P}^{1}(\R)$, we obtain that $X_{\pm}(b_{4})\in \mathbb{P}^{1}(\R)$. With the same notation as in Lemma~\ref{lem:half-parallelogram}, there exists $\widetilde{\omega}\in \PPP$ such that $x(\widetilde{\omega})=X_{\pm}(b_{4})$. We have two possibilities: 
\begin{itemize}
     \item $(x(\widetilde{\omega}),y(\widetilde{\omega}))=(X_{\pm}(b_{4}),b_{4})$;
     \item $(x(\widetilde{\omega}),y(\widetilde{\omega}))=\iota_{1}(X_{\pm}(b_{4}),b_{4})$. 
\end{itemize}
Remind that $(x(\frac{\omega_{3}}{2}),y(\frac{\omega_{3}}{2}))=(X_{\pm}(b_{4}),b_{4})$, see Lemma \ref{lem:correspondence_ai_ei}. Since $\iup_{1}(\widetilde{\omega})=-\widetilde{\omega}$ by \eqref{eq:expression_group_universal_cover}, we obtain  $\frac{\omega_{3}}{2}\in \pm \widetilde{\omega}+\omega_{1}\Z+\omega_{2}\Z$. Further, the fact that 
${\widetilde{\omega}\in \PPP}$ yields ${\frac{\omega_{3}}{2}\in \{\omega_{1}\R+\frac{\omega_{2}}{2}\Z\}\bigcup\{\frac{\omega_{1}}{2}\Z+\omega_{2}\R\}}$. Since $y(\frac{\omega_{3}}{2})=b_4$ is also real, Lemma~\ref{lem:real_points} implies that  ${\frac{\omega_{3}}{2}\in\{\frac{\omega_{1}}{2}\Z+\omega_{2}\R\}}$. So $\omega_{3}\in\{\omega_{1}\Z+\omega_{2}\R\}$.  Since $\omega_{3}$ belongs to the fundamental parallelogram $\omega_{1}[0,1)+\omega_{2}[0,1)$, we find that $0\leq \omega_{3}< \omega_{2}$. Note that $\omega_{3}\neq  0$, since otherwise $\iup_{1}= \iup_{2}$, which is not possible by definition of $\iota_{1}$ and $\iota_{2}$. So $\omega_{3}\in (0,\omega_{2})$. With $\frac{\omega_{3}}{2}\in \pm \widetilde{\omega}+\omega_{1}\Z+\omega_{2}\Z$ and $\widetilde{\omega}\in \PPP$, we also deduce that $0\leq \widetilde{\omega}< \omega_{2}/2$.
 
Introduce
\begin{equation}
\label{eq:Omega_3}
     \Omega_{3}=2\int_{\wp (\omega_{3}/2)}^{+\infty}\frac{d\mathbf{u}}{\sqrt{4\mathbf{u}^{3}-g_{2}\mathbf{u}-g_{3}}} = \int_{a_{4}}^{X_{\pm}(b_{4})} \frac{dx}{\sqrt{D(x)}},
\end{equation}
where the second equality follows from a similar reasoning as in the proof of Proposition \ref{prop:uniformization}. Our aim is to prove that $\omega_3=\Omega_3$, see \eqref{eq:expression_omega_3}, and we first prove that $\Omega_3\in (0,\omega_2)$. 
Using together the facts that $0\leq \widetilde{\omega}< \omega_{2}/2$, $(x(\widetilde{\omega}),y(\widetilde{\omega}))=(X_{\pm}(b_{4}),b_{4})$ and Remark \ref{rem1}, we find $X_{\pm}(b_{4})\in (a_4,a_1)$. By Theorem \ref{thm:disczeroes} and its proof, the discriminant is positive on this interval. With $(a_4, X_{\pm}(b_{4}))\subset  (a_4,a_1)$, this implies that $\Omega_3=\int_{a_{4}}^{X_{\pm}(b_{4})} \frac{dx}{\sqrt{D(x)}}\in (0,\int_{a_{4}}^{a_1} \frac{dx}{\sqrt{D(x)}})=(0,\omega_2)$.

Using an inverse of the Weierstrass function, see \cite[$\S$ 20.221]{WW}, we get that ${\wp(\frac{\omega_3}{2})=\wp(\frac{\Omega_3}{2})}$. This finally entails 
\begin{equation*}
     \frac{\omega_3}{2}=\frac{\Omega_3}{2}\mod \omega_{1}\Z+\omega_{2}\Z, \quad \text{or} \quad \frac{\omega_3}{2}=-\frac{\Omega_3}{2} \mod \omega_{1}\Z+\omega_{2}\Z.
\end{equation*}
The only solution satisfying to the constraints $\Omega_{3},\omega_3\in(0,\omega_2)$ is $\Omega_{3}=\omega_3$, which completes the proof.\end{proof}

\subsection{Analytic continuation}
\label{sec:analcont}

The goal of this subsection is to prove that the functions $F^{1}(x;t)$ and $F^{2}(y;t)$ introduced in \eqref{eq:def_F1_F2} admit multivalued meromorphic continuations on $\Etproj$, that we will call $\rx$ and $\ry$. 
 
Define the domains 
\begin{equation}
\label{eq:three_domains}
     \mathcal{D}_{x} := \Etproj\bigcap \{\vert x\vert < 1 \},\quad \mathcal{D}_{y} := \Etproj\bigcap \{\vert y\vert < 1 \}\quad \text{and}\quad \mathcal{D}_{x,y}:=\mathcal{D}_{x}\bigcap \mathcal{D}_{y}. 
\end{equation}     
Remind that for $t$ fixed in $(0,1)$, $Q(x,y;t)$ converges for $\vert x\vert\leq1$ and $\vert y\vert\leq1$, being a generating function of probabilities. The same holds for $F^{1}(x;t)$ and $F^{2}(y;t)$. Remind also that we have defined in \eqref{def:Gamma-x_Gamma-y}
\begin{equation*}
     \Gamma_{x}=\Etproj\bigcap \{\vert x\vert = 1 \}\quad \text{and}\quad
     \Gamma_{y}=\Etproj\bigcap \{\vert y\vert = 1 \}.
\end{equation*}
$\Gamma_{x}$ is the union of two disjoint connected paths $\Gamma_{x}^{\pm}$ such that ${(x,y)\in \Gamma_{x}^{+}\Rightarrow |y|>1}$ and $(x,y)\in \Gamma_{x}^{-}\Rightarrow |y|<1$, see Lemma \ref{lem:curve_x1}. A similar statement holds for $\Gamma_{y}$. 

 We first prove a few topological properties of the domains  \eqref{eq:three_domains}.

\begin{lemma}
\label{lem:Dxy_nonempty}
One has $\mathcal{D}_{x,y}\neq \varnothing$.
\end{lemma}
\begin{lemma}
\label{lem:Dx_Dy_connected}
The sets $\mathcal{D}_{x}$ and $\mathcal{D}_{y}$ are connected. 
\end{lemma}
\begin{proof}[Proof of Lemmas \ref{lem:Dxy_nonempty} and \ref{lem:Dx_Dy_connected}]
Lemma \ref{lem:Dxy_nonempty} is obvious from Lemma \ref{lem:curve_x1}, as $\mathcal{D}_{x,y}$ contains $\Gamma_{x}^{-}$ and $\Gamma_{y}^{-}$. Let us do the proof of Lemma \ref{lem:Dx_Dy_connected} for $\mathcal{D}_{x}$, the other case being similar. By definition, 
\begin{equation}\label{eq1}
\mathcal{D}_{x}=\{(x,Y_{-}(x)): \vert x\vert < 1\} \bigcup \{(x,Y_{+}(x)): \vert x\vert <1\}.
\end{equation}
Since both sets $\{(x,Y_{-}(x)): \vert x\vert < 1\}$  and  $\{(x,Y_{+}(x)): \vert x\vert < 1\}$ are obviously connected, it suffices to prove that they have a nonempty intersection. 
 
As we can see in Theorem \ref{thm:disczeroes}, there exists $\vert a_{1}\vert <1$ such that $y\mapsto \overline{K}(a_{1},y;t)$ has a double root. This means that  $Y_{-}(a_{1})=Y_{+}(a_{1})$, proving that 
\begin{equation*}
     (a_{1},Y_{\pm}(a_{1}))\in \{(x,Y_{-}(x)): \vert x\vert < 1\}\bigcap \{(x,Y_{+}(x)): \vert x\vert < 1\}.\qedhere
\end{equation*}
\end{proof}

Let us examine the consequences of Lemmas \ref{lem:Dxy_nonempty} and \ref{lem:Dx_Dy_connected}. We may define the three generating functions $Q(x,y;t)$, $F^{1}(x;t)$ and $F^{2}(y;t)$ on $\mathcal{D}_{x,y}$. Restricting the main functional equation \eqref{eq:funcequ} on $\mathcal{D}_{x,y}$, we obtain  
\begin{equation}
\label{eq:funcequaonthecurve2}
  0  = F^{1}(x;t) +F^{2}(y;t)-K(0,0;t) Q(0,0;t)+xy.
\end{equation}
Since $F^{1}(x;t)$ is analytic for $\vert x\vert<1$, it is analytic in $\mathcal{D}_{x}$. We may define $F^{2}(y;t)$ on $\mathcal{D}_{x}\setminus \mathcal{D}_{x,y}$ in the following way: 
\begin{equation*}
     F^{2}(y;t)=-F^{1}(x;t)+K(0,0;t) Q(0,0;t)-xy.
\end{equation*}     
Using \eqref{eq:funcequaonthecurve2}, we obtain an analytic continuation of $F^{2}(y;t)$ on  $\mathcal{D}_{x}$.
 Similarly $F^{2}(y;t)$ is analytic in $\mathcal{D}_{y}$, and we may continue $F^{1}(x;t)$ on $\mathcal{D}_{y}$.

Since the union of two connected sets ($\mathcal{D}_{x}$ and $\mathcal{D}_{y}$) with nonempty intersection (viz, $\mathcal{D}_{x,y}$) is connected, we have proved that we may use \eqref{eq:funcequaonthecurve2} to lift $F^{1}(x;t)$ and $F^{2}(y;t)$ as meromorphic functions on the open connected domain 
\begin{equation*}
     {\mathcal{D}:=\mathcal{D}_{x}\bigcup\mathcal{D}_{y}}.
\end{equation*}
The domain $\mathcal{D}\neq \Etproj$ admits the paths $\Gamma_{x}^{+},\Gamma_{y}^{+}$ as boundary, and obviously contains $\Gamma_{x}^{-},\Gamma_{y}^{-}$.

\medskip
 
The next step is to apply $\Lambda$ in \eqref{eq:Lambda} so as to lift $F^{1}(x;t)$ and $F^{2}(y;t)$ on the universal covering $\mathbb C$ of $\Etproj$.
The map $\Lambda$ being $(\omega_{1},\omega_{2})$-periodic, it induces a map $\widetilde{\Lambda}$ from $\C/( \omega_{1}\Z+\omega_{2}\Z)$ to $\overline{E}_t$, which is an homeomorphism. Then 
\begin{equation*}
     \widetilde{\Lambda}^{-1}(\mathcal{D}),\quad \widetilde{\Lambda}^{-1}(\mathcal{D}_{x})\quad \text{and}\quad \widetilde{\Lambda}^{-1}(\mathcal{D}_{y})
\end{equation*}
are connected domains.     
Using the homeomorphism property of $\widetilde{\Lambda}$, we obtain that the boundary of the open set $\widetilde{\Lambda}^{-1}(\mathcal{D}_{x})$ is the image under $\widetilde{\Lambda}^{-1}$ of the boundary of $\mathcal{D}_{x}$. By Lemma~\ref{lem:curve_x1}, we deduce that the boundary of $\widetilde{\Lambda}^{-1}(\mathcal{D}_{x})$ is composed by the two paths $\widetilde{\Lambda}^{-1}(\Gamma_{x}^{\pm})$. A similar statement holds for $\widetilde{\Lambda}^{-1}(\mathcal{D}_{y})$. Finally, the boundary of the open connected set $\widetilde{\Lambda}^{-1}(\mathcal{D})$ is composed of the paths $\widetilde{\Lambda}^{-1}(\Gamma_{x}^{+})$ and $\widetilde{\Lambda}^{-1}(\Gamma_{y}^{+})$, and by construction $\widetilde{\Lambda}^{-1}(\mathcal{D})$ contains $\widetilde{\Lambda}^{-1}(\Gamma_{x}^{-})$ and $\widetilde{\Lambda}^{-1}(\Gamma_{y}^{-})$. See Figure \ref{The_fundamental_parallelogram} for an illustration.

Consider now $\Lambda^{-1}(\mathcal{D})$, the preimage of $\mathcal{D}$ via $\Lambda$. It is $(\omega_{1},\omega_{2})$-periodic but not necessarily connected anymore. Let us fix $\widetilde{\mathcal{D}}\subset \C$, a connected component of $\Lambda^{-1}(\mathcal{D})$ in the $\omega$-plane that intersects the fundamental parallelogram $\omega_1[0,1)+\omega_2[0,1)$. In particular we have $\Lambda (\widetilde{\mathcal{D}})=\mathcal{D}$. Similarly, let us define $\widetilde{\mathcal{D}}_{x}\subset \C$ and $\widetilde{\mathcal{D}}_{y}\subset \C$ such that 
\begin{equation*}
     \Lambda (\widetilde{\mathcal{D}}_{x})=\mathcal{D}_{x},\quad \Lambda (\widetilde{\mathcal{D}}_{y})=\mathcal{D}_{y}\quad \text{and}\quad
     \widetilde{\mathcal{D}}=\widetilde{\mathcal{D}}_{x}\bigcup \widetilde{\mathcal{D}}_{y}.
\end{equation*}
Note that by definition, any pair of (distinct) paths among the four paths $\Gamma_{x}^{\pm},\Gamma_{y}^{\pm}$ has an empty intersection.
 From what precedes, there exist connected paths $\widetilde{\Gamma}_{x}^{\pm}$ and $\widetilde{\Gamma}_{y}^{\pm}$ such that
\begin{itemize}
     \item $\Lambda (\widetilde{\Gamma}_{x}^{\pm})=\Gamma_{x}^{\pm}$ and $\Lambda (\widetilde{\Gamma}_{y}^{\pm})=\Gamma_{y}^{\pm}$;
    \item  $\widetilde{\mathcal{D}}_{x}$ is delimited by $\widetilde{\Gamma}_{x}^{+}$ and $\widetilde{\Gamma}_{x}^{-}$;
     \item $\widetilde{\mathcal{D}}_{y}$ is delimited by $\widetilde{\Gamma}_{y}^{+}$ and $\widetilde{\Gamma}_{y}^{-}$;
     \item $\widetilde{\mathcal{D}}$ is delimited by $\widetilde{\Gamma}_{x}^{+}$ and $\widetilde{\Gamma}_{y}^{+}$, and contains $\widetilde{\Gamma}_{x}^{-}$ and $\widetilde{\Gamma}_{y}^{-}$.
\end{itemize}
In the lemma hereafter, we derive some properties of the paths $\widetilde{\Gamma}_{x}^{\pm}$ and $\widetilde{\Gamma}_{y}^{\pm}$. See Figure \ref{The_fundamental_parallelogram} for a typical example.

\begin{lemma}\label{lem1}
The following holds:
\begin{enumerate}
     \item\label{it:1}The paths $\widetilde{\Gamma}_{x}^{\pm}$ are $\omega_{1}$-periodic, and do not cross the vertical straight lines going through $\ell\omega_{2}/2$, for any $\ell\in \Z$;
     \item\label{it:2}The paths $\widetilde{\Gamma}_{y}^{\pm}$ are $\omega_{1}$-periodic, and do not cross the vertical straight lines going through $\omega_3/2+\ell\omega_{2}/2$, for any $\ell\in \Z$;
     \item\label{it:3}The domain $\widetilde{\mathcal{D}}$ is $\omega_{1}$-periodic;          
     \item\label{it:4}The domain $\widetilde{\mathcal{D}}$ is delimited by a left boundary, namely $\widetilde{\Gamma}_{x}^{+}$, and a right boundary, $\widetilde{\Gamma}_{y}^{+}$.
\end{enumerate}
\end{lemma}
 
Our proof of Lemma \ref{lem1} will use various properties of the uniformization. In comparison, let us note that the proof of \cite{FIM} (valid for $t=1$) is based on more topological arguments, as the homology classes of $\widetilde{\Gamma}_{x}^{\pm}$, see \cite[Chapter~3]{FIM}.

\begin{proof}
Let us begin with the paths $\widetilde{\Gamma}_{x}^{\pm}$.  
By Remark \ref{rem1}, for all $\omega$ belonging to the vertical segment $[\omega_{2}/2,(\omega_{1}+\omega_{2})/2]$ one has $\vert x(\omega)\vert<1$, and similarly on $[0,\omega_{1}/2]$ one has $\vert x(\omega)\vert>1$.
 This implies that for all $\ell\in \Z$, $\widetilde{\Gamma}_{x}^{\pm}$ do not cross the vertical segments ${[\omega_{2}/2+\ell\omega_{2},(\omega_{1}+\omega_{2})/2)+\ell\omega_{2}]}$ and $[\ell\omega_{2},\omega_{1}/2+\ell\omega_{2}]$. Then, for all $\ell\in \Z$, they do not cross the vertical segments ${[\ell\omega_{2}/2,\omega_{1}/2+\ell\omega_{2}/2]}$.
It follows from the same Remark~\ref{rem1} that the preimage of  $-1$ by $x(\omega)$ belongs to the horizontal lines going through $\ell \omega_{1}$, $\ell\in \Z$, see the right display on Figure \ref{fig:half_parallelogram}. Similarly, the preimage of $1$ belongs to the horizontal lines going through $(2\ell+1) \omega_{1}/2$, $\ell\in \Z$.
Since for all $\ell\in \Z$, the connected paths $\widetilde{\Gamma}_{x}^{\pm}$ do not cross the vertical segments ${[\ell\omega_{2}/2,\omega_{1}/2+\ell\omega_{2}/2]}$, we deduce the existence of $\ell'\in \Z$, such that one of the paths $\widetilde{\Gamma}_{x}^{\pm}$ crosses the horizontal segments $[\ell'\omega_{2},\omega_{2}/2+\ell'\omega_{2}]$ and $[\omega_{1}/2+\ell'\omega_{2},(\omega_{1}+\omega_{2})/2+\ell'\omega_{2}]$. Let us first assume that $\widetilde{\Gamma}_{x}^{+}$ is the crossing path.

Using Proposition~\ref{prop:uniformization} together with the fact that $\wp(\overline{\omega})=\overline{\wp(\omega)}$, see \eqref{eq:conjug}, we deduce that $x(\overline{\omega})=\overline{x(\omega)}$, and then $\widetilde{\Gamma}_{x}^{+}$ is sent by complex conjugation to a translation of $\widetilde{\Gamma}_{x}^{\pm }$  by the periods $\omega_{1},\omega_{2}$. Since $\Gamma_{x}^{+}\bigcap \Gamma_{x}^{-}=\varnothing$ and  the crossing point of $\widetilde{\Gamma}_{x}^{+}$  with the horizontal line $\R \omega_{2}$ is fixed by complex conjugation, it follows that  $\widetilde{\Gamma}_{x}^{+}$ is setwise fixed by complex conjugation. We deduce that
 $\widetilde{\Gamma}_{x}^{+}$ is symmetric w.r.t.\ the horizontal line going through $0$. It follows that the restriction of $\widetilde{\Gamma}_{x}^{+}$ to 
the horizontal strip $\R+(-1/2,1/2)\omega_{1}$  does not cross the vertical segments going through $\ell\omega_{2}/2$, $\ell\in \Z$. Furthermore, there exists $\omega'$ belonging to the horizontal line going through
$-\frac{\omega_{1}}{2}$, such that $\omega '$ and $\omega '+\omega_{1}$ belong to $\widetilde{\Gamma}_{x}^{+}$. Since $\Lambda (\omega)=\Lambda (\omega+\omega_{1})$ we deduce that $\widetilde{\Gamma}_{x}^{+}$ is $\omega_{1}$-periodic, and does not cross the vertical lines going through $\ell\omega_{2}/2$, $\ell\in \Z$. 

Since $\iota_{1}(\Gamma_{x}^{-})=\Gamma_{x}^{+}$ and since $\widetilde{\Gamma}_{x}^{\pm}$ are connected, there exist $\ell_{1},\ell_{2}\in \Z$ such that  $\iup_{1}(\widetilde{\Gamma}_{x}^{+})=\widetilde{\Gamma}_{x}^{-}+\ell_{1}\omega_{1}+\ell_{2}\omega_{2}$. As $\widetilde{\Gamma}_{x}^{+}$ is $\omega_{1}$-periodic, we may take $\ell_{1}=0$. Using $\iup_{1}(\omega)=-\omega$, we deduce from what we proved for $\widetilde{\Gamma}_{x}^{+}$ that similar results hold for $\widetilde{\Gamma}_{x}^{-}$: it is an $\omega_{1}$-periodic path of $\C$ that does not cross the vertical lines going through $\ell\omega_{2}/2$, $\ell\in \Z$. The proof in the case where $\widetilde{\Gamma}_{x}^{-}$ is the crossing path is totally similar, and the proof of (\ref{it:1}) is complete.

\medskip

Let us now show (\ref{it:2}). With \eqref{eq6}, a similar statement is valid for $\widetilde{\Gamma}_{y}^{\pm}$: they are $\omega_{1}$-periodic paths of $\C$ that do not cross the vertical lines going through $\omega_{3}/2+\ell\omega_{2}/2$, $\ell\in \Z$.

\medskip
 
We prove (\ref{it:3}). We deduce the $\omega_{1}$-periodicity of $\widetilde{\mathcal{D}}$ from the fact (proved above the statement of Lemma \ref{lem1}) that $\widetilde{\mathcal{D}}$ is delimited by $\widetilde{\Gamma}_{x}^{+}$ and $\widetilde{\Gamma}_{y}^{+}$, which are $\omega_{1}$-periodic. 

\medskip

Let us conclude with the proof of (\ref{it:4}). By the first two points, the curves $\widetilde{\Gamma}_{x}^{+}$ and $\widetilde{\Gamma}_{y}^{+}$ do not cross convenient vertical lines. Furthermore, by construction $\Gamma_{x}^{+}\bigcap \Gamma_{y}^{+}=\varnothing$, proving that the two boundaries do not cross. Then, one path is the left boundary of $\widetilde{\mathcal{D}}$ while the other one is the right boundary. It remains to determine which one is the left one. 

 We use $\iup_{1}(\widetilde{\Gamma}_{x}^{+})=\widetilde{\Gamma}_{x}^{-}+\ell_{2}\omega_{2}$,  for some $\ell_{2}\in \Z$, and $\iup_{1}(\omega)=-\omega$, to deduce that there exists $\ell_{x}\in \Z$ such that the two paths $\widetilde{\Gamma}_{x}^{\pm}$ are symmetric w.r.t.\ the vertical line going through $\ell_{x}\omega_{2}/2$.
Furthermore, since $\mathcal{D}_{x}\neq \Etproj$ and since $\widetilde{\Gamma}_{x}^{\pm}$ are the boundaries of $\widetilde{\mathcal{D}}_{x}$, we deduce that one of $\widetilde{\Gamma}_{x}^{\pm}$ belongs to one of the vertical strips
\begin{equation*}
     \R\omega_{1}+(\ell_{x}-1,\ell_{x})\frac{\omega_{2}}{2}\quad\text{and}\quad \R\omega_{1}+(\ell_{x},\ell_{x}+1)\frac{\omega_{2}}{2},
\end{equation*}
and the remaining path of $\widetilde{\Gamma}_{x}^{\pm}$ lies in the other strip.
Remind that  $\widetilde{\Gamma}_{x}^{\pm}$ are symmetric the one of the other w.r.t.\ the vertical line going through $\ell_{x}\omega_{2}/2$. Moreover, with Remark~\ref{rem1}, one has $\vert x(\ell\omega_{2}/2)\vert<1$ (resp.\ $\vert x(\ell\omega_{2}/2)\vert>1$) for odd (resp.\ even) values of $\ell$. We deduce that $\ell_{x}$ should be odd.

Similarly, using ${\iota_{2}(\Gamma_{y}^{+})=\Gamma_{y}^{-}}$ and $\iup_{2}(\omega)=-\omega+\omega_{3}$, there exists an odd integer $\ell_{y}$ such that each of $\widetilde{\Gamma}_{y}^{\pm}$ belongs to one of the vertical strips
\begin{equation*}
     \R\omega_{1}+(\ell_{y}-1,\ell_{y})\frac{\omega_{2}}{2}+\frac{\omega_{3}}{2}\quad \text{and}\quad  \R\omega_{1}+(\ell_{y},\ell_{y}+1)\frac{\omega_{2}}{2}+\frac{\omega_{3}}{2},
\end{equation*}
and $\widetilde{\Gamma}_{y}^{\pm}$ are symmetric w.r.t.\ the vertical line going through $\ell_{y}\omega_{2}/2+\omega_{3}/2$. 

To conclude the proof of (\ref{it:4}), we first note that $\widetilde{\mathcal{D}}_{x}\bigcap \widetilde{\mathcal{D}}_{y}\neq \varnothing$ since otherwise  $\widetilde{\mathcal{D}}=\widetilde{\mathcal{D}}_{x}\bigcup \widetilde{\mathcal{D}}_{y}$ would not be connected. Remind that $\widetilde{\mathcal{D}}_{x}$ is delimited by $\widetilde{\Gamma}_{x}^{+}$ and $\widetilde{\Gamma}_{x}^{-}$, and $\widetilde{\mathcal{D}}_{y}$ is delimited by $\widetilde{\Gamma}_{y}^{+}$ and $\widetilde{\Gamma}_{y}^{-}$. This shows that the domain delimited by $\widetilde{\Gamma}_{x}^{\pm}$ has nonempty intersection with the domain bounded by $\widetilde{\Gamma}_{y}^{\pm}$.
Since $\omega_{3}\in (0,\omega_{2})$, the odd integer $\ell$ such that $\ell\omega_{2}/2+\omega_{3}/2$ is the closest to $\ell_{x}\omega_{2}/2$ is $\ell_{x}$. This shows that $\ell_{y}=\ell_{x}$. Then one of the two curves $\widetilde{\Gamma}_{x}^{\pm}$ has to be on the left to $\widetilde{\Gamma}_{y}^{\pm}$, and the latter has to be the left boundary of $\widetilde{\mathcal{D}}$. Since $\widetilde{\Gamma}_{x}^{+}$ and $\widetilde{\Gamma}_{y}^{+}$ are the boundaries of $\widetilde{\mathcal{D}}$, the result follows.
\end{proof}
 
\unitlength=0.6cm
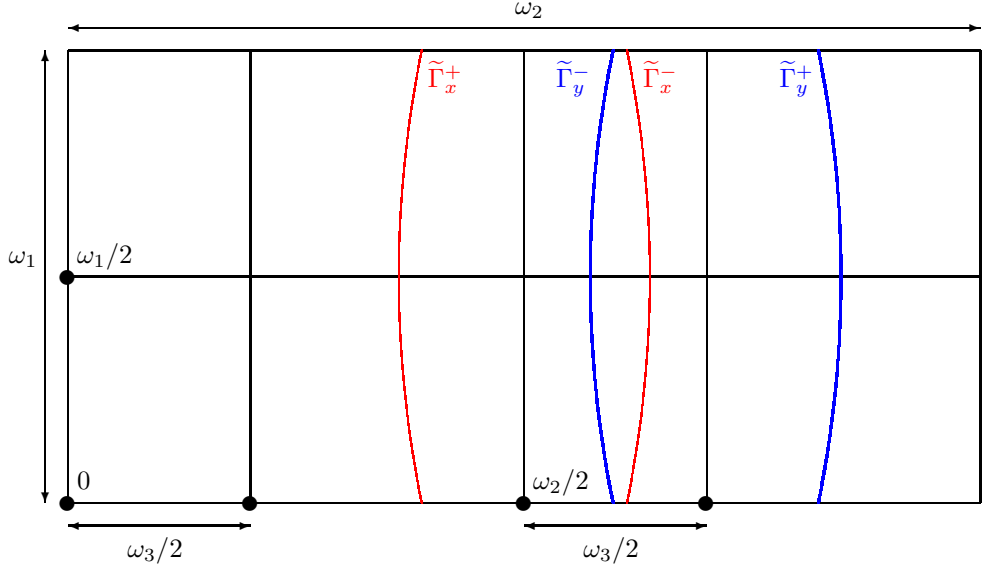
\begin{figure}[t]
    \vspace{65mm}
  \begin{center}
\begin{tabular}{cccc}

\begin{picture}(0,0)(0,0)
\put(-10,0){\line(1,0){20}}
\put(-10,10){\line(1,0){20}}
\put(-10,5){\line(1,0){20}}
\put(-10,0){\line(0,1){10}}
\put(10,0){\line(0,1){10}}
\put(0,0){\line(0,1){10}}
\put(4,0){\line(0,1){10}}
\put(-6,0){\line(0,1){10}}
\put(-10.23,-0.23){{\LARGE$\bullet$}}
\put(-6.23,-0.23){{\LARGE$\bullet$}}
\put(-0.23,-0.23){{\LARGE$\bullet$}}
\put(3.77,-0.23){{\LARGE$\bullet$}}
\put(-10.23,4.77){{\LARGE$\bullet$}}
\put(-9.8,0.3){{$0$}}
\put(0.2,0.3){{$\omega_{2}/2$}}
\put(-9.8,5.3){{$\omega_{1}/2$}}
\put(-10,-0.5){\vector(1,0){4}}
\put(-6,-0.5){\vector(-1,0){4}}
\put(-8.7,-1.2){{$\omega_{3}/2$}}
\put(0,-0.5){\vector(1,0){4}}
\put(4,-0.5){\vector(-1,0){4}}
\put(1.3,-1.2){{$\omega_{3}/2$}}
\put(-10,10.5){\vector(1,0){20}}
\put(10,10.5){\vector(-1,0){20}}
\put(-0.2,10.8){{$\omega_{2}$}}
\put(-10.5,0){\vector(0,1){10}}
\put(-10.5,10){\vector(0,-1){10}}
\put(-11.3,5.3){{$\omega_{1}$}}
\textcolor{red}{
\qbezier(-2.25,10)(-2.75,7.5)(-2.75,5)
\qbezier(2.25,10)(2.75,7.5)(2.75,5)
\qbezier(-2.25,0)(-2.75,2.5)(-2.75,5)
\qbezier(2.25,0)(2.75,2.5)(2.75,5)
\put(-2.11,9.2){{$\widetilde{\Gamma}_{x}^{+}$}}
\put(2.6,9.2){{$\widetilde{\Gamma}_{x}^{-}$}}
}
\thicklines
\textcolor{blue}{
\qbezier(1.75,10)(1.25,7.5)(1.25,5)
\qbezier(6.25,10)(6.75,7.5)(6.75,5)
\qbezier(1.75,0)(1.25,2.5)(1.25,5)
\qbezier(6.25,0)(6.75,2.5)(6.75,5)
\put(0.5,9.2){{$\widetilde{\Gamma}_{y}^{-}$}}
\put(5.4,9.2){{$\widetilde{\Gamma}_{y}^{+}$}}
}
\end{picture}
    \end{tabular}
  \end{center}
  \vspace{10mm}
\caption{The parallelogram $\omega_1[0,1)+\omega_2[0,1)$, and important points and domains on it}
\label{The_fundamental_parallelogram}
\end{figure} 

\begin{lemma}
\label{lem:strip}
There exists a nonempty open connected set $\mathcal{U}\subset \C$ such that $\widetilde{\tau}(\mathcal{U})\subset \widetilde{\mathcal{D}}\bigcap \widetilde{\tau}(\widetilde{\mathcal{D}})$. Furthermore,
we have 
\begin{equation*} 
     \bigcup_{\ell\in \Z} \widetilde{\tau}^{\ell}\bigl(\widetilde{\mathcal{D}}\bigr)=\C.
\end{equation*}
\end{lemma}

\begin{proof}
 Remind that $\widetilde{\mathcal{D}}$ is delimited by $\widetilde{\Gamma}_{x}^{+}$ and $\widetilde{\Gamma}_{y}^{+}$. With $\iota_{1}(\Gamma_{x}^{+})=\Gamma_{x}^{-}$ and $\iota_{2}(\Gamma_{x}^{-})\in \mathcal{D}_{y}$ we deduce that $\tau (\Gamma_{x}^{+})\subset \mathcal{D}$. Using $\widetilde{\tau}(\omega)=\omega+\omega_{3}$, we deduce the existence of $\ell_{1},\ell_{2}\in \Z$ such that $\widetilde{\tau}(\widetilde{\Gamma}_{x}^{+})\in \widetilde{\mathcal{D}}+\ell_{1}\omega_{1}+\ell_{2}\omega_{2}$. By Lemma \ref{lem1}, $\widetilde{\mathcal{D}}$ is $\omega_{1}$-periodic and we may assume $\ell_{1}=0$. 
 By Lemma \ref{lem1}, the left boundary of $\widetilde{\mathcal{D}}$ is  $\widetilde{\Gamma}_{x}^{+}$. By Lemma \ref{lem:formula_omega_3}, $\omega_{3}\in (0,\omega_{2})$, proving that $\ell_{2}=0$, since otherwise $\widetilde{\mathcal{D}}=\C$ which contradicts $\widetilde{\Gamma}_{x}^{+}\bigcap  \widetilde{\mathcal{D}}=\varnothing$. 
 We thus have 
\begin{equation}\label{eq7}
\widetilde{\tau} (\widetilde{\Gamma}_{x}^{+})\subset \widetilde{\mathcal{D}}.
\end{equation}
The map $\widetilde{\tau}$ being continuous, $\widetilde{\tau}(\widetilde{\mathcal{D}})$ is an open connected domain. Since $\widetilde{\Gamma}_{x}^{+}$ is the left boundary of the open set $\widetilde{\mathcal{D}}$, we deduce from \eqref{eq7} that there exists a nonempty open connected set $\mathcal{U}\subset\widetilde{\mathcal{D}}$ such that $\widetilde{\tau}(\mathcal{U})\subset \widetilde{\mathcal{D}}$. This shows that
$\widetilde{\tau}(\mathcal{U})\subset \widetilde{\mathcal{D}}\bigcap \widetilde{\tau}(\widetilde{\mathcal{D}})$.

A straightforward induction  yields that for all $k\in \N$, $\bigcup_{\ell=-k }^{k} \widetilde{\tau}^{\ell}(\widetilde{\mathcal{D}})$ is a connected domain with left boundary $\widetilde{\Gamma}_{x}^{+}-k\omega_{3}$ and right boundary  $\widetilde{\Gamma}_{y}^{+}+k\omega_{3}$. Since $\omega_{3}>0$, we deduce the result by making $k$ going to infinity.
\end{proof}

\begin{theo}
\label{theo:analytic_continuation}
The functions $F^1(x;t)$ and  $F^2(y;t)$ may be lifted to the universal cover of $\Etproj$. We will call respectively $\rx$ and $\ry$ the continuations. Seen as functions of $\omega$, they are meromorphic on $\C$ and satisfy 
\begin{align}
     \ry(\omega+\omega_{3};t)   &=\ry(\omega;t)+x(\omega)(y(\omega)-y(-\omega)), \label{eq:omega_3_per_ry} \\ 
     \ry(\omega+\omega_{1};t)&=  \ry(\omega;t),\label{eq:omega_1_per_ry}\\
     \rx(\omega+\omega_{3};t) & =\rx(\omega;t) +y(-\omega)(x(\omega)-x(\omega+\omega_{3})),\label{eq:omega_3_per_rx} \\
     \rx(\omega+\omega_{1};t)&=  \rx(\omega;t).\label{eq:omega_1_per_rx}
\end{align}
\end{theo}

\begin{proof}
By Lemma \ref{lem1}, $\widetilde{\mathcal{D}}$ is $\omega_{1}$-periodic, proving that $\widetilde{\mathcal{D}}$ is stable by addition by $\omega_{1}$. Since at this step of the continuation, $F^{1}$ is univalued as a function on $\mathcal{D}$, we find that the analytic continuation of $\rx$ on $\widetilde{\mathcal{D}}$ is $\omega_{1}$-periodic, see \eqref{eq:omega_1_per_rx}. The same holds for $\ry$, see \eqref{eq:omega_1_per_ry}.

Consider $x(\omega)$ and $y(\omega)$ as defined in Proposition \ref{prop:uniformization}.
From \eqref{eq:funcequaonthecurve2} we deduce that for all $\omega\in \widetilde{\mathcal{D}}$,
\begin{equation}
\label{eq3}
     x(\omega)y(\omega)+\rx(\omega;t)+\ry (\omega;t)-K(0,0;t)Q(0,0;t)=0. 
\end{equation}
With the same reasons as in the proof of  Lemma \ref{lem1}, there exist $\ell_{1},\ell_{2}\in \Z$ such that  $\iup_{1}(\widetilde{\Gamma}_{x}^{+})=\widetilde{\Gamma}_{x}^{-}+\ell_{1}\omega_{1}+\ell_{2}\omega_{2}$. Since $\widetilde{\Gamma}_{x}^{+}$ is $\omega_{1}$-periodic, we may take $\ell_{1}=0$. Let us set $\ell_{2}=\ell$ and apply $\omega\mapsto -\omega+\ell \omega_{2}$ to both sides of \eqref{eq3}. Using $\rx(\omega;t)=\rx(-\omega+\ell \omega_{2};t)$ (this follows from $\iup_{1}(\omega)=-\omega$ and $\iota_{1}(F^{1})=F^{1}$) we deduce that
\begin{equation}
\label{eqn:equgroupaction}
     x(-\omega)y(-\omega)+\rx(\omega;t)+\ry (-\omega+\ell\omega_{2};t)-K(0,0;t)Q(0,0;t)=0. 
\end{equation}
Consider the open connected set of Lemma \ref{lem:strip} such that $\widetilde{\tau}(\mathcal{U})\subset \widetilde{\mathcal{D}}\bigcap \widetilde{\tau}(\widetilde{\mathcal{D}})$. The shift $\omega\mapsto \omega+\omega_{3}$ maps $\mathcal{U}$ into $\widetilde{\mathcal{D}}$. Apply $\omega\mapsto -\omega+\omega_{3}+\ell\omega_{2}$ to both sides of \eqref{eqn:equgroupaction}. Using $\ry(-\omega+\ell\omega_{2};t)=\ry(\omega+\omega_{3};t)$ (this follows from $\iup_{2}(\omega)=-\omega+\omega_{3}$ and $\iota_{2}(F^{2})=F^{2}$), we obtain that for all $\omega\in \mathcal{U}$
\begin{equation}\label{eqn:equgroupaction2}
     x(\omega+\omega_{3})y(\omega+\omega_{3}) + \rx (\omega+\omega_{3};t)+\ry (-\omega+\ell\omega_{2};t)-K(0,0;t)Q(0,0;t)=0.
\end{equation}
Subtracting  \eqref{eq3} to \eqref{eqn:equgroupaction} (resp.\ \eqref{eqn:equgroupaction} to \eqref{eqn:equgroupaction2}), we find
\begin{align*}
\ry (-\omega+\ell\omega_{2};t)  - \ry (\omega;t)     & =x(\omega)y(\omega)-x(-\omega)y(-\omega) ,   \\
\rx (\omega+\omega_{3};t)  - \rx (\omega;t) &=  x(-\omega)y(-\omega) -x(\omega+\omega_{3})y(\omega+\omega_{3}).   
\end{align*}
We now use $\ry(-\omega+\ell\omega_{2};t)=\ry(\omega+\omega_{3};t)$, $x(-\omega)=x(\omega)$ and $y(-\omega)=y(\omega+\omega_{3})$ to deduce that \eqref{eq:omega_3_per_ry} and \eqref{eq:omega_3_per_rx} are valid for every $\omega\in \mathcal{U}$. Let us now define $\rx'$ and $\ry'$, analytic on $\widetilde{\tau}(\widetilde{\mathcal{D}})$, by means of the formulas, for all 
$\omega\in \widetilde{\mathcal{D}}$, 
\begin{align*}
     \rx' (\omega+\omega_{3})&=\rx (\omega)+y(-\omega)(x(\omega)-x(\omega+\omega_{3})),\\
     \ry' (\omega+\omega_{3})&=\ry (\omega)+x(\omega)(y(\omega)-y(-\omega)).
\end{align*}
From what precedes, $\rx$ and $\rx'$ (resp.\ $\ry$ and $\ry'$) are equal on $\widetilde{\tau} (\mathcal{U})\subset \widetilde{\mathcal{D}}\bigcap \widetilde{\tau}(\widetilde{\mathcal{D}})$, see Lemma~\ref{lem:strip}, so the analytic continuation principle allows us to continue $\rx$ and $\ry$ on the connected set $\widetilde{\mathcal{D}}\bigcup \widetilde{\tau}(\widetilde{\mathcal{D}})$. Iterating the same reasoning, we find that generating functions $\rx$ and $\ry$ as well as the identities \eqref{eq:omega_3_per_ry} and \eqref{eq:omega_3_per_rx} can be extended to the whole connected domain $ \bigcup_{\ell\in \Z} \widetilde{\tau}^{\ell}(\widetilde{\mathcal{D}})$. With Lemma \ref{lem:strip}, the latter is $\C$. By construction, the lifts are $\omega_{1}$-periodic, as already stated. 
\end{proof}

\begin{rem}
\label{rem2}
So far we have considered walks that start at the origin $(0,0)$. We could similarly handle models of walks starting at the point $(i,j)$ with probability $p_{i,j}$, such that $\sum_{i,j} p_{i,j}=1$. In this situation, the functional equation satisfied by the generating function is 
\begin{equation*}
     K(x,y;t)Q(x,y;t)=\sum_{i,j\geq 0} p_{i,j}x^{i+1}y^{j+1}+F^{1}(x;t) +F^{2}(y;t)-K(0,0;t) Q(0,0;t).
\end{equation*}     
Using exactly the same strategy, we may prove that the functions $\rx$ and  $\ry$ may be continued to $\C$. They are $\omega_{1}$-periodic and satisfy
\begin{align*}
&\ry (\omega+\omega_{3};t)  - \ry (\omega;t) \hspace{0.3mm} =\sum_{i,j\geq 0} p_{i,j}x^{i+1}(\omega)y^{j+1}(\omega)-\sum_{i,j\geq 0} p_{i,j}x^{i+1}(-\omega)y^{j+1}(-\omega), \\ 
&\rx (\omega+\omega_{3};t)  - \rx (\omega;t)  = \hspace{0.8mm} \sum_{i,j\geq 0} p_{i,j}x^{i+1}(-\omega)y^{j+1}(-\omega)-\sum_{i,j \geq 0} p_{i,j}x^{i+1}(\omega+\omega_{3})y^{j+1}(\omega+\omega_{3}). 
\end{align*}
\end{rem}

\section{Sufficient conditions for differential transcendence}
\label{sec:transcendence}

Throughout this section, we assume that $\widetilde{\tau}$ has infinite order, which is equivalent to doing the hypothesis that the group is infinite.  Our main results are to derive differential transcendence criteria for $\rx$ and $\ry$. 
By Theorem \ref{theo:analytic_continuation}, these functions satisfy  difference equations of the form 
\begin{equation*}
     {\widetilde{\tau}(f) - f = b},
\end{equation*}
with $\widetilde\tau$ defined in \eqref{eq:expression_group_universal_cover}. Galois theoretic methods to study the differential properties of such functions have been developed in \cite{HS,DHR,DHRS,DHRS2}, see also \cite{HAR16}. In this section we describe a consequence of the latter theory and show how it will be used to prove that in many cases, $x\mapsto Q(x,0;t)$ and $y\mapsto Q(0,y;t)$ are differentially transcendental.

\subsection{Background of difference Galois theory}

We remind that $\Etproj$ is an elliptic curve. The field $\mer(\Etproj)$ of meromorphic functions on the elliptic curve may be identified, via the Weierstrass elliptic function, to the field of meromorphic functions on $\C$ which are $(\omega_{1},\omega_{2})$-periodic. We have a natural derivation on this field given by the $\omega$-derivative $\partial_{\omega}$. As Theorem~\ref{theo:analytic_continuation} shows, the continuations of $\rx$ and $\ry$ belong to $\mer(\C)$, the field of meromorphic functions on $\mathbb C$. The latter may be equipped with the derivation $\partial_{\omega}$, and the inclusion 
\begin{equation*}
     \bigl(\mer(\Etproj),\partial_{\omega}\bigr)\subset \bigl(\mer(\C),\partial_{\omega}\bigr)
\end{equation*}
of differential fields holds. 

\begin{defi}\label{defi1}
Let $(E,\partial_{\omega})\subset (F,\partial_{\omega})$ be differential fields. 
We say that $f\in F$ is {\rm  differentially algebraic} over $E$ if it satisfies a nontrivial algebraic differential equation with coefficients in $E$, i.e., if for some $m$ there exists a nonzero polynomial ${P(y_0, \ldots , y_m) \in E[y_0, \ldots , y_m]}$ such that
$$
P(f,\partial_{\omega}(f), \ldots, \partial_{\omega}^m(f)) = 0.
$$
 We say that  $f$  is differentially transcendental over $E$ if it is not differentially algebraic.  
\end{defi}

Our first remark is that  by \cite[Lemmas~6.3 and 6.4]{DHRS}, the series $x\mapsto Q(x,0;t)$ is differentially algebraic over $\C(x)$ if and only if $\omega\mapsto\rx(\omega;t)$ is differentially algebraic over $\mer(\Etproj)$. And symmetrically the same holds for $y\mapsto Q(0,y;t)$ and $\omega\mapsto\ry(\omega;t)$. We may therefore focus on $\rx$ and $\ry$.

\begin{defi}
\label{defi-field}
A $(\partial_{\omega},\widetilde{\tau})$-field is a triple $(K,\partial_{\omega}, \widetilde{\tau})$, where $K$ is a field, $\partial_{\omega}$ is a derivation on $K$, $\widetilde{\tau}$ is an automorphism of $K$, and where $\partial_{\omega}$ and $\widetilde{\tau}$ commute on $K$. 
\end{defi}  

Since by \eqref{eq:expression_group_universal_cover} $\widetilde{\tau}(\omega)=\omega+\omega_{3}$, we deduce that on $\mer(\C)$, 
\begin{equation*}
     \widetilde{\tau} \circ \partial_{\omega} =\partial_{\omega}\circ \widetilde{\tau}.
\end{equation*}
Then the triples $\bigl(\mer(\Etproj), \partial_{\omega},\widetilde{\tau}\bigr)$ and  $\bigl(\mer(\C), \partial_{\omega},\widetilde{\tau}\bigr)$ provide examples for Definition \ref{defi-field}.

Proposition~2.6 of \cite{DHR} gives a criterion for differential transcendence in the above general setting; we are now going to translate it in our context. The result of \cite{DHR} only requires the assumption to embed the solutions $\rx$ and $\ry$ into a $(\partial_{\omega},\widetilde{\tau})$-field. This is done due to Theorem~\ref{theo:analytic_continuation} and, remarkably, this is the only point where analytic tools are needed in this section. 

\begin{propo}[\cite{DHR}]
\label{prop:caract_hyperalg} 
Let $b\in \mer(\Etproj)$ and $f \in \mer(\C)$. Assume that 
\[\widetilde{\tau}(f) - f = b.\]
If $f$ is differentially algebraic over $\mer(\Etproj)$, then there exist an integer $n \geq 0$, ${c_0,\ldots,c_{n-1} \in \C}$ and $g \in \mer(\Etproj)$, such that 
\begin{equation}
\label{eq:lindiffb}
    \partial_{\omega}^n(b)+c_{n-1}\partial_{\omega}^{n-1}(b)+\cdots+c_{1}\partial_{\omega}(b)+c_{0} b =\widetilde{\tau}(g)-g. 
\end{equation}
\end{propo}

Let
\begin{equation}
\label{eq:first_def_b_1_b_2}
     b_{1}= \iota_1(y)(x-\tau (x))\quad \text{and} \quad b_{2}=x(y-\iota_1(y))
\end{equation}     
be the quantities that appear in the right-hand sides of Equations \eqref{eq:omega_3_per_rx} and \eqref{eq:omega_3_per_ry} of Theorem \ref{theo:analytic_continuation}.
By \cite[Corollary 3.9 and Proposition 3.10]{DHRS}, Proposition \ref{prop:caract_hyperalg} has the following consequence:

\begin{coro}
\label{cor:prop:caract_hyperalg}
Assume that $b_{1}$ (resp.\ $b_{2}$) in \eqref{eq:first_def_b_1_b_2} has a pole $P \in \Etproj$ of order $m \geq 1$ such that none of the $\tau^k(P)$ with $k \in \Z \setminus \{0\}$ is a pole of order $\geq m$ of $b_{1}$ (resp.\ $b_{2}$). Then $x\mapsto Q(x,0;t)$ and $y\mapsto  Q(0,y;t)$ are differentially transcendental over $\C(x)$ and $\C(y)$, respectively. 
\end{coro}

\begin{rem}
As in the end of Section \ref{sec:analcont}, we may consider walks starting at the point $(i,j)$ with probability $p_{i,j}$, $\sum_{i,j} p_{i,j}=1$. A similar criterion to Corollary \ref{cor:prop:caract_hyperalg} might be derived but it would be less effective, since the second member of the equation appearing in Remark \ref{rem2} may have many poles. 
\end{rem}

\subsection{Differential transcendence for genus one walks}

In this section we consider walks of genus one and derive criteria that ensure that the functions $x\mapsto Q(x,0;t)$  and $y\mapsto Q(0,y;t)$ are differentially transcendental over $\C(x)$ and $\C(y)$, respectively. These criteria are strong and tractable enough to show that these functions are differentially transcendental in many concrete weighted cases. 
 
As suggested by Corollary \ref{cor:prop:caract_hyperalg}, we are now interested in the poles of $b_{1}$ and $b_{2}$ in \eqref{eq:first_def_b_1_b_2}. In the unweighted case, see \cite{DHRS}, a set of simple criteria to apply Corollary~\ref{cor:prop:caract_hyperalg} has been given. As we wrote in the introduction, the only serious obstruction to the extension of the results to the weighted case was the analytic continuation of the generating functions, which was unknown in this setting. Since the latter has been proved in this paper (our Theorem \ref{theo:analytic_continuation}), the arguments of \cite{DHRS} may be straightforwardly adapted to the weighted case. For this reason, the results will be given without proof. 

Let us compute the poles of $b_{2}=x(y-\iota_1(y))$. Let $P_{1},P_{2}$ be the poles of $x$ and $Q_{1},Q_{2}$ be the poles of $y$. We may prove that the poles of $b_{2}$ are
\begin{equation*}
     \{P_1, P_2, Q_1, Q_2 , \iota_1(Q_1), \iota_1(Q_2)\}.
\end{equation*}
We are going to see the poles of  $b_{1}$ and $b_{2}$ as elements of $\mathbb{P}^{1}(\C)^2$. 

We split the analysis in three cases: generic case (Theorem \ref{thm:generic1}), double pole case (Theorem~\ref{thm:generic2}) and triple pole case (Theorem \ref{thm:generic3}).
 
\subsubsection*{Generic case} 
This is the situation where at least one of the two quantities 
\begin{equation}
\label{gc}
     \frac{\Delta^x_{[1:0]}}{t^2}=d_{1,0}^2-4d_{1,-1}d_{1,1}
     \quad \text{and}\quad
     \frac{\Delta^y_{[1:0]}}{t^2}=d_{0,1}^2-4d_{-1,1}d_{1,1}
\end{equation}   
is not a square in $\mathbb{Q}(d_{i,j},t)$. Note that almost every choice of the $d_{i,j}$ leads to a generic case.
\begin{theo}[\cite{DHRS}, Proposition 5.1]
\label{thm:generic1} 
If at least one of of the quantities in \eqref{gc} is not a square in $\mathbb{Q}(d_{i,j},t)$, then the generating functions $x\mapsto Q(x,0;t)$ and $y\mapsto Q(0,y;t)$ are differentially transcendental over $\C(x)$ and $\C(y)$, respectively.
\end{theo}

\begin{coro}
Assume $d_{i,j}\in \Q$ and $t$ transcendental. If at least one of of the quantities in \eqref{gc} is not a square in $\mathbb{Q}$, then $x\mapsto Q(x,0;t)$ and $y\mapsto Q(0,y;t)$ are differentially transcendental over $\C(x)$ and $\C(y)$, respectively.
\end{coro}

\begin{ex}
Consider the leftmost model on Figure \ref{fig:illustration_theorems_infinite}. This is an example that comes from a projection of a three-dimensional model, and to which Theorem~\ref{thm:generic1} applies. We have 
\begin{equation*}
     \frac{\Delta^y_{[1:0]}}{t^2}=d_{0,1}^2-4d_{-1,1}d_{1,1}=-\frac{1}{3},
\end{equation*}
which obviously is not a square in $\mathbb{Q}(d_{i,j},t)$ (remind that $t$ is real).
\end{ex}

\begin{ex}\label{ex1}
Consider the walk with $d_{-1,1} = d_{1,1} = d_{1,-1} = d_{0,-1} = 1/4$ and all other $d_{i,j} = 0$. By \eqref{eq:kernelwalk} the kernel curve $\Etproj$ is defined by
\begin{equation*}
     \overline{K}(x_0,x_1, y_0, y_1; t) = x_0x_1y_0y_1 - \frac{t}{4}(x_1^2y_0^2 + x_0^2y_0^2 + x_0^2 y_1^2+x_0x_1y_1^2)=0.
\end{equation*}
To compute the poles $P_{1}$ and $P_{2}$ of $x$, we have to solve 
$\overline{K}(1,0, y_0, y_1; t)=0$, which gives the simple equation
$$ \frac{t}{4}(y_0^2 + y_1^2)=0.$$
Its solutions in $\P1(\C)$ are $[\mathbf{i}\!:\!1]$ and  $[- \mathbf{i}\!:\!1]$, so 
$$P_{1},P_{2}=([ 1\!:\!0],[ \pm \mathbf{i}\!:\!1]).$$
Likewise, in order to compute the poles $Q_{1},Q_{2}$ of $y$, we solve $\overline{K}(x_0,x_1, 1, 0; t)=0$ and find 
 $$Q_{1},Q_{2}=([ \pm \mathbf{i}\!:\!1],[ 1\!:\!0]).$$
To compute $\iota_1(Q_1), \iota_1(Q_2)$ we have to find the other root of $\overline{K}(\pm \mathbf{i},1, y_{0}, y_{1}; t)=0$. We obtain 
\begin{equation*}
     \iota_{1}(Q_{1}),\iota_{1} (Q_{2})=([ +\mathbf{i}\!:\!1],[ t(1+\mathbf{i})\!:\! 4]),([ -\mathbf{i}\!:\!1],[ t(1-\mathbf{i})\!:\! 4]).\qedhere
\end{equation*}
\end{ex}

\medskip

As Example \ref{ex1} shows, although the kernel curve $\Etproj$ has coefficients in $\mathbb{Q}(d_{i,j},t)$, the poles of $b_1$ or $b_{2}$ may belong to a nontrivial intermediate field extension $\mathbb{Q}(d_{i,j},t) \subset L \subset \C$.  When this is the case, the authors of \cite{DHRS} use a Galois action to derive that at least one pole of $b_{1}$ or $b_{2}$ should be alone in its orbit with respect to $\tau$, allowing them to apply Corollary \ref{cor:prop:caract_hyperalg}. 

\medskip

Let us consider unweighted quadrant walks. Then  \cite[Proposition 5.1]{DHRS} allows us to conclude that when Assumption \ref{assumption} is satisfied, $26$ over the $51$ unweighted quadrant models listed in \cite{KurkRasch} have a differentially transcendental generating function. In the weighted context, we find that almost every choice of the $d_{i,j}$'s leads to a differentially transcendental counting function. 

\subsubsection*{Double pole case}
Assume that $d_{1,1}=0$ and $d_{1,0}d_{0,1}\neq 0$. 
By \cite[Section~5.2.1]{DHRS}, the proof being similar in our situation, the function $b_{2}$ admits at least two double poles, and to deduce that $x\mapsto Q(x,0;t)$ and $y\mapsto Q(0,y;t)$ are differentially transcendental over $\C(x)$ and $\C(y)$, respectively, it suffices to prove that one of the double poles is alone in its orbit with respect to $\tau$. Let us give a criterion ensuring such a result.

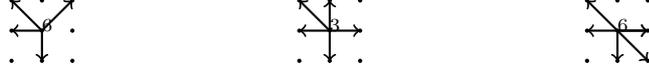
\begin{figure}\bigskip
\begin{tikzpicture}[scale=.4, baseline=(current bounding box.center)]
\foreach \x in {-1,0,1} \foreach \y in {-1,0,1} \fill(\x,\y) circle[radius=2pt];
\draw[thick,->](0,0)--(-1,0);
\draw[thick,->](0,0)--(1,1);
\draw[thick,->](0,0)--(0,-1);
\draw[thick,->](0,0)--(-1,1);
\put(-2,1){{$\small{1/2}$}}
\put(1,1){{$\small{1/6}$}}
\put(-2,-0.2){{$\small{1/6}$}}
\put(-0.5,-1.5){{$\small{1/6}$}}
\end{tikzpicture}
\quad\quad\quad\quad\quad\quad\quad\quad
\begin{tikzpicture}[scale=.4, baseline=(current bounding box.center)]
\foreach \x in {-1,0,1} \foreach \y in {-1,0,1} \fill(\x,\y) circle[radius=2pt];
\draw[thick,->](0,0)--(-1,0);
\draw[thick,->](0,0)--(1,0);
\draw[thick,->](0,0)--(0,-1);
\draw[thick,->](0,0)--(-1,1);
\draw[thick,->](0,0)--(0,1);
\put(-2,1){{$\small{1/6}$}}
\put(1,-0.2){{$\small{1/6}$}}
\put(-2,-0.2){{$\small{1/6}$}}
\put(-0.5,-1.5){{$\small{1/6}$}}
\put(-0.5,1){{$\small{1/3}$}}
\end{tikzpicture}
\quad\quad\quad\quad\quad\quad\quad\quad
\begin{tikzpicture}[scale=.4, baseline=(current bounding box.center)]
\foreach \x in {-1,0,1} \foreach \y in {-1,0,1} \fill(\x,\y) circle[radius=2pt];
\draw[thick,->](0,0)--(-1,0);
\draw[thick,->](0,0)--(1,0);
\draw[thick,->](0,0)--(0,-1);
\draw[thick,->](0,0)--(-1,1);
\draw[thick,->](0,0)--(1,0);
\draw[thick,->](0,0)--(1,-1);
\put(-2,1){{$\small{1/3}$}}
\put(1,-0.2){{$\small{1/6}$}}
\put(1,-1.5){{$\small{1/6}$}}
\put(-2,-0.2){{$\small{1/6}$}}
\put(-0.5,-1.5){{$\small{1/6}$}}
\end{tikzpicture}
\bigskip

\caption{Three models, to which (from left to right) Theorems \ref{thm:generic1}, \ref{thm:generic2} and \ref{thm:generic3} apply}
\label{fig:illustration_theorems_infinite}
\end{figure}

\begin{theo}[\cite{DHRS}, Theorem 5.3]
\label{thm:generic2} 
Assume that $d_{1,1}=0$ and $d_{1,0}d_{0,1}\neq 0$. Assume further that there are no point of $\Etproj (\mathbb{Q}(d_{i,j},t))$ that is fixed by $\iota_{1}$ or $\iota_{2}$. Then $x\mapsto Q(x,0;t)$ and $y\mapsto Q(0,y;t)$ are differentially transcendental over $\C(x)$ and $\C(y)$, respectively.
\end{theo}

\begin{ex}
Consider the second model on Figure \ref{fig:illustration_theorems_infinite}; it provides an example of model to which Theorem~\ref{thm:generic2} applies. Assume that $t$ is transcendental. We have to check that there are no point of $\Etproj (\mathbb{Q}(t))$ fixed by $\iota_{1}$ or $\iota_{2}$.
 
If we take the notations of Section \ref{secdisc}, we find that the fixed points of $\iota_{1}$ and $\iota_{2}$ are of the form 
$(X_{\pm}(b_{i}),b_{i})$ and $(a_{i},Y_{\pm }(a_{i}))$. Furthermore, the $a_{i}$'s and $b_{i}$'s are roots of the discriminant. Since $t$ is transcendental over $\mathbb{Q}$, we may identify $\mathbb{Q}(t)$ with the field of rational functions in $t$, and the discriminants become polynomials with coefficients in $\mathbb{Q}(t)$. We now have to prove that they have no root in $\Q(t)$, i.e., that the two following polynomials have no root in $\Q(t)$:
\begin{equation*}
 \big( x_1^2 -  \frac{6}{t} x_0x_1 + x_0^2\big)^2 
 - 4 x_0x_1(d_{-1,1} x_1^2 +2 x_0x_1)\quad \text{and}\quad \big( y_1^2 -  \frac{6}{t} y_0y_1 + 2y_0^2\big)^2  
- 4 y_0y_1 (  y_0y_1 + y_0^2).
\end{equation*}
We begin by the first one. To the contrary, assume the existence of $x_{0},x_{1}\in \Q(t)$ with $(x_{0},x_{1})\neq (0,0)$ that cancel the discriminant. If $x_{1}=0$ then necessarily $x_{0}=0$, and so we may assume that $x_{1}=1$ in the projective coordinates. We have to consider
 \begin{equation}\label{eq4}
  \big(1 -  \frac{6}{t} x_0 + x_0^2\big)^2 
 - 4 x_0(d_{-1,1}  +2 x_0).
 \end{equation}
Obviously, \eqref{eq4} implies that $x_{0}$ is nonzero, and so seen as a rational function of $t$, $x_0$ has no zero. Moreover, $t=0$ may be its only pole, of order at most one. So a rational solution must take the form $x_{0}=a/t$ for some $a\in \Q$. However, replacing $x_0$ by $a/t$ in \eqref{eq4}, we easily notice that no value of $a$ works. This proves that the first discriminant has no root in $\Q(t)$.

Let us now consider the second discriminant. As above let us assume that it has a root in $\Q(t)$, which leads to a root in $\Q(t)$ of 
\begin{equation*}
     \big(1-  \frac{6}{t} y_0 + 2y_0^2\big)^2 - 4 y_0 (  y_0 + y_0^2) .
\end{equation*}
Again, $y_{0}$ has no zero and $t=0$ is its only pole, of order one. So $y_{0}=a/t$ and we find $a= 3$, which however turns out to be impossible.  We then conclude that Theorem~\ref{thm:generic2} applies.
\end{ex}

Let us consider unweighted quadrant walks. Then \cite[Theorem 5.3]{DHRS} concludes that when Assumption \ref{assumption} is satisfied and $t$ is transcendental, $5$ over the $51$ unweighted quadrant models listed in \cite{KurkRasch} have a differentially transcendental generating function.

\subsubsection*{Triple pole case}

Assume that $d_{1,1}=d_{1,0}=0$ and $d_{0,1}\neq 0$. By \cite[Section~5.2.2]{DHRS}, the proof being similar in our situation, the function $b_{2}$ admits exactly one triple pole, and the other poles are at most double. Then, Corollary \ref{cor:prop:caract_hyperalg} applies and we obtain the result below:

\begin{theo}[\cite{DHRS}, Theorem 5.4]
\label{thm:generic3} 
Assume that $d_{1,1}=d_{1,0}=0$ and $d_{0,1}\neq 0$.  Then $x\mapsto Q(x,0;t)$ and $y\mapsto Q(0,y;t)$ are differentially transcendental over $\C(x)$ and $\C(y)$, respectively.
 
Similarly, assume that  $d_{1,1}=d_{0,1}=0$ and $d_{1,0}\neq 0$.  Then $x\mapsto Q(x,0;t)$ and $y\mapsto Q(0,y;t)$ are differentially transcendental over $\C(x)$ and $\C(y)$, respectively.
\end{theo}

See Figure \ref{fig:illustration_theorems_infinite} for an example of model to which Theorem~\ref{thm:generic3} applies.\\ \par 

Let us consider unweighted quadrant walks. Then \cite[Theorem 5.4]{DHRS} concludes that when Assumption \ref{assumption} is satisfied, $9$ over the $51$ unweighted quadrant models listed in \cite{KurkRasch} have a differentially transcendental generating function. Thus the combination of the three above theorems allows us to conclude that when $t$ is transcendental, $40$ over the $51$ unweighted quadrant models listed in \cite{KurkRasch} have a differentially transcendental generating function. Among the $11$ remaining cases, $9$ are differentially algebraic, see \cite{BBMR16,DHRS}, and $2$ are differentially transcendental, see \cite{DHRS}.

\section{A sufficient condition for algebraicity}
\label{sec:algebraicity}

The aim of the present section is to give sufficient conditions for the algebraicity of the generating functions. Throughout the section it will be assumed that $t\in(0,1)$ is such that the kernel curve is nondegenerate and elliptic (Assumption \ref{assumption}), see Proposition \ref{prop:singcases} and Lemma \ref{lem:doublezero}. Let $\langle \iota_{1},\iota_{2} \rangle$ be the group introduced in Section~\ref{sec:kernel}, see \eqref{eq:generators_group}. In this section we restrict ourselves to the case of a finite group. Extending results of \cite{BMM,FaRa-10,KurkRasch,KuRa-15} to the weighted case, we prove that the generating function $(x,y)\mapsto Q(x,y;t)$ is systematically holonomic, and even algebraic if the orbit-sum
\begin{equation}
\label{eq:orbit-sum_formal}
     \mathcal{O}(x,y)=\sum_{\theta\in\langle \iota_{1},\iota_{2} \rangle}  \textnormal{sign}(\theta)\cdot\theta(xy)
\end{equation}
identically vanishes (in the definition \eqref{eq:orbit-sum_formal} above, $\textnormal{sign}(\theta)=1$ (resp.\ $-1$) if the number of elements $\iota_1,\iota_2$ used to write $\theta\in\langle \iota_{1},\iota_{2} \rangle$ is even (resp.\ odd)). 

\begin{theo}
\label{theo:alg_crit1}
Let $0< t <1$ be such that the group $\langle \iota_{1},\iota_{2} \rangle$ restricted to the kernel curve $\Etproj$ is finite (i.e., $\omega_3/\omega_2=k/\ell\in\mathbb Q$, $\gcd(k,\ell)=1$, see Proposition \ref{prop:uniformization}). Then the function $(x,y)\mapsto Q(x,y;t)$ is holonomic. Moreover, it is algebraic if and only if the orbit-sum $\mathcal{O}(x(\omega),y(\omega))$ is identically zero.
\end{theo}

\begin{coro}
\label{coro:alg_crit2}
If the group $\langle \iota_{1},\iota_{2} \rangle$ of birational transformations of $\mathbb{P}^{1}(\C)^2$ is finite, then for all $0<t<1$, 
 the function $(x,y)\mapsto Q(x,y;t)$ is holonomic. Moreover, if the orbit-sum
\eqref{eq:orbit-sum_formal} is identically zero, then the above generating function is algebraic for all $0<t<1$. 
\end{coro}

Before proving these results, let us do a series of remarks.

\smallskip

$\bullet$ As an example, Corollary \ref{coro:alg_crit2} applies to the three models of Figure \ref{fig:illustration_theorems_finite}, which admit a group of order $10$, as shown in \cite{KauersYatchak}. The orbit-sum \eqref{eq:orbit-sum_formal} is $0$, see again \cite{KauersYatchak}, from which it follows that these models are algebraic (as functions of $x$ and $y$). This gives another proof of this fact, after the recent proof given in \cite{BBMR16} (note, \cite{BBMR16} proves the algebraicity of $Q(x,y;t)$ in the three variables).

\smallskip

$\bullet$  Theorem \ref{theo:alg_crit1} and Corollary \ref{coro:alg_crit2} apply exactly the same, for walks starting at point $(i,j)$ with probability $p_{i,j}$, $\sum_{i,j} p_{i,j}=1$. In the orbit sum \eqref{eq:orbit-sum_formal}, $xy$ should be replaced by $\sum_{i,j}p_{i,j}x^{i+1}y^{j+1}$, and $b_1$ and $b_2$ in \eqref{eq:first_def_b_1_b_2_bis} become as in Remark \ref{rem2}.

\smallskip

$\bullet$ For all $23$ finite group models of unweighted walks listed in \cite{BMM}, Corollary \ref{coro:alg_crit2} is proved in \cite{BMM} for $22$ out of the $23$ models, while the article \cite{BostanKauersTheCompleteGenerating} concludes the proof for the last model (Gessel's walk). 

Moreover, for certain families of weighted walks, Corollary \ref{coro:alg_crit2} is shown in \cite{KauersYatchak}. Note that these families in \cite{KauersYatchak} completely describe the models having a group of order $4$, $6$ and $8$.

This is actually a refined version of Corollary \ref{coro:alg_crit2} which is proved in \cite{BMM,BostanKauersTheCompleteGenerating,KauersYatchak}, as the holonomy and algebraicity of the generating functions are proved in the three variables $x,y,t$, whereas we show these properties only in the variables $x$ and $y$.

\smallskip

$\bullet$ The above results hold when a certain group is finite, so it is natural to ask how often this property happens to be satisfied. Note a first subtlety: in Theorem \ref{theo:alg_crit1} the group is defined on $\Etproj$, while in Corollary \ref{coro:alg_crit2} it is considered as acting on $\mathbb{P}^{1}(\C)^2$. As obviously $\Etproj$ is strictly included in $\mathbb{P}^{1}(\C)^2$, the second group has an order larger than or equal to that of the first one. A more precise comparison of the two groups may be found in \cite[Section 2]{FaRa-10}.

When the group $\langle \iota_{1},\iota_{2} \rangle$ acts on $\mathbb{P}^{1}(\C)^2$ (Corollary \ref{coro:alg_crit2}), its only known possible orders are $4$, $6$, $8$, $10$ and $+\infty$, see \cite{KauersYatchak}. It is believed that the cardinality of finite groups in the weighted case may be bounded by $10$. On the other hand, restricted to $\Etproj$ (as in Theorem \ref{theo:alg_crit1}), the group can take any order (even and equal to or larger than $4$), see \cite{fayolleRaschel}. 

\smallskip

$\bullet$ From a methodological viewpoint, the proof of Theorem \ref{theo:alg_crit1} is largely inspired by \cite[Section 9]{KuRa-15}. 
\begin{figure}\bigskip
\begin{tikzpicture}[scale=.4, baseline=(current bounding box.center)]
\foreach \x in {-1,0,1} \foreach \y in {-1,0,1} \fill(\x,\y) circle[radius=2pt];
\draw[thick,->](0,0)--(-1,0);
\draw[thick,->](0,0)--(1,1);
\draw[thick,->](0,0)--(0,-1);
\draw[thick,->](0,0)--(-1,1);
\draw[thick,->](0,0)--(1,0);
\draw[thick,->](0,0)--(0,1);
\draw[thick,->](0,0)--(1,-1);
\put(-2,1){{$\small{1/9}$}}
\put(1,1){{$\small{1/9}$}}
\put(1,-0.2){{$\small{2/9}$}}
\put(1,-1.5){{$\small{1/9}$}}
\put(-2,-0.2){{$\small{1/9}$}}
\put(-0.5,-1.5){{$\small{1/9}$}}
\put(-0.5,1){{$\small{2/9}$}}
\end{tikzpicture}
\quad\quad\quad\quad\quad\quad\quad\quad
\begin{tikzpicture}[scale=.4, baseline=(current bounding box.center)]
\foreach \x in {-1,0,1} \foreach \y in {-1,0,1} \fill(\x,\y) circle[radius=2pt];
\draw[thick,->](0,0)--(-1,0);
\draw[thick,->](0,0)--(-1,-1);
\draw[thick,->](0,0)--(0,-1);
\draw[thick,->](0,0)--(-1,1);
\draw[thick,->](0,0)--(1,0);
\draw[thick,->](0,0)--(0,1);
\draw[thick,->](0,0)--(1,-1);
\put(-2,1){{$\small{1/9}$}}
\put(-2,-1.5){{$\small{1/9}$}}
\put(1,-0.2){{$\small{1/9}$}}
\put(1,-1.5){{$\small{1/9}$}}
\put(-2,-0.2){{$\small{2/9}$}}
\put(-0.5,-1.5){{$\small{2/9}$}}
\put(-0.5,1){{$\small{1/9}$}}
\end{tikzpicture}
\quad\quad\quad\quad\quad\quad\quad\quad
\begin{tikzpicture}[scale=.4, baseline=(current bounding box.center)]
\foreach \x in {-1,0,1} \foreach \y in {-1,0,1} \fill(\x,\y) circle[radius=2pt];
\draw[thick,->](0,0)--(-1,0);
\draw[thick,->](0,0)--(-1,-1);
\draw[thick,->](0,0)--(0,-1);
\draw[thick,->](0,0)--(-1,1);
\draw[thick,->](0,0)--(1,0);
\draw[thick,->](0,0)--(0,1);
\draw[thick,->](0,0)--(1,1);
\put(-2,1){{$\small{1/9}$}}
\put(1,1){{$\small{1/9}$}}
\put(1,-0.2){{$\small{1/9}$}}
\put(-2,-1.5){{$\small{1/9}$}}
\put(-2,-0.2){{$\small{2/9}$}}
\put(-0.5,-1.5){{$\small{1/9}$}}
\put(-0.5,1){{$\small{2/9}$}}
\end{tikzpicture}\bigskip
\caption{Three models with a group of order $10$. These models, proposed in \cite{KauersYatchak} and conjectured to be algebraic, were recently solved in \cite{BBMR16}}
\label{fig:illustration_theorems_finite}
\end{figure}
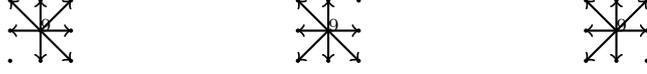

\medskip

Before starting the proof of Theorem \ref{theo:alg_crit1} we need some notation. Let $b_1(\omega),b_2(\omega)$ be the liftings of $b_1,b_2$ in \eqref{eq:first_def_b_1_b_2} on the universal cover of $\Etproj$:
\begin{equation}
\label{eq:first_def_b_1_b_2_bis}
   b_1 (\omega)=y(\omega+\omega_{3})(x(\omega)-x(\omega+\omega_{3})) \quad
     \text{and}\quad
       b_2 (\omega)=x(\omega)(y(\omega)-y(-\omega)).
\end{equation} 
Remind that $\widetilde{\tau}(\omega)=\omega+\omega_{3}$ with $\omega_{3}\in (0,\omega_{2})$, see Lemma \ref{lem:formula_omega_3}. By \eqref{eq:omega_3_per_ry} and \eqref{eq:omega_3_per_rx} one has
\begin{equation}
\label{eq:tau_rx_ry}
     \widetilde{\tau}(r_x)-r_x=b_1\quad\text{and}\quad\widetilde{\tau}(r_y)-r_y=b_2.
\end{equation}
Finally we introduce
\begin{equation}
\label{eq:orbit-sums_omega}
     \mathcal O_1(\omega)=b_1(\omega)+\widetilde{\tau}(b_1(\omega))+\cdots+\widetilde{\tau}^{\ell-1}(b_1(\omega)),
     \quad
     \mathcal O_2(\omega)=b_2(\omega)+\widetilde{\tau}(b_2(\omega))+\cdots+\widetilde{\tau}^{\ell-1}(b_2(\omega)),
\end{equation}
and as we will see below, we have, with $\mathcal O$ as in \eqref{eq:orbit-sum_formal}, 
\begin{equation*}
     \mathcal O_2(\omega)=-\mathcal O_1(\omega)=\mathcal O(x(\omega),y(\omega)).
\end{equation*}
As sums of $(\omega_1,\omega_2)$-elliptic functions, $\mathcal O_1$ and $\mathcal O_2$ are $(\omega_1,\omega_2)$-elliptic.

\begin{proof}[Proof of Theorem \ref{theo:alg_crit1}]
Remind that $\ell$ is defined by $\omega_3/\omega_2=k/\ell\in\mathbb Q$. For $0\leq n\leq \ell-1$, apply $\widetilde\tau^{n}$ to Equation \eqref{eq:tau_rx_ry} and sum these $\ell$ identities. We easily obtain 
\begin{equation}
\label{eq:O_1-O_2}
     \widetilde{\tau}^{\ell} (\rx(\omega;t))  - \rx(\omega;t)=\mathcal O_1(\omega)\quad\text{and}\quad
     \widetilde{\tau}^{\ell} (\ry(\omega;t))  - \ry(\omega;t)=\mathcal O_2(\omega).
\end{equation}
We deduce that for every $(\omega_1,\omega_2)$-elliptic function $\widetilde{\tau}^{\ell}f(\omega)=f(\omega+\ell\omega_{3})=f(\omega+k\omega_{2})=f(\omega)$, showing that $\widetilde{\tau}^{\ell}$ restricted to the field of $(\omega_1,\omega_2)$-elliptic functions is equal to the identity. Consequently $\widetilde{\tau}^{\ell}(x(\omega)y(\omega))=x(\omega)y(\omega)$. As we may see in the proof of Theorem \ref{theo:analytic_continuation}, we have  $b_{1}(\omega)+b_{2}(\omega)=\widetilde{\tau}(x(\omega)y(\omega))-x(\omega)y(\omega)$. We conclude from \eqref{eq:orbit-sums_omega} that $\mathcal O_1 (\omega)=-\mathcal O_2 (\omega)$.

Let us first assume that the orbit-sum $\mathcal O_1 (\omega)$ identically vanishes. Let us rewrite the first identity of \eqref{eq:O_1-O_2} as 
\begin{equation*}
     \rx(\omega+\ell\omega_3;t)- \rx(\omega;t)=0, 
\end{equation*}
which reads that $\rx$ is $\ell\omega_3$-periodic. Being in addition $\omega_1$-periodic by \eqref{eq:omega_1_per_rx}, we deduce that $\rx$ is $(\omega_1,\ell\omega_3)$-elliptic, and therefore $(\omega_1,k\omega_2)$-elliptic since $\omega_3/\omega_2=k/\ell$.
 Using Lemma~\ref{lem12} below, we obtain that  $\rx(\omega;t)$ is an algebraic function of $x(\omega)$.
For the exact same reasons, $\ry(\omega;t)$ is an algebraic function of $y(\omega)$. This shows that $x\mapsto Q(x,0;t)$ and $y\mapsto Q(0,y;t)$ are algebraic. We conclude with the main functional equation \eqref{eq:funcequ} that $(x,y)\mapsto Q(x,y;t)$ is algebraic.
 
Conversely, assume that the function $(x,y)\mapsto Q(x,y;t)$ is algebraic. Then $\rx(\omega;t)$ (resp.\ $\ry(\omega;t)$) is algebraic in $x(\omega)$ (resp.\ $y(\omega)$) and $\rx(\omega;t)$ is algebraic in $\wp(\omega)$, due to Proposition \ref{prop:uniformization}. By \cite[Proposition 6]{DreR}, there exist $\ell_{1},\ell_{2}\in \N^{*}$ such that $\rx(\omega;t)$ is $(\ell_{1}\omega_{1},\ell_{2}\omega_{2})$-elliptic. Note that due to Theorem \ref{theo:analytic_continuation} ($\omega_1$-periodicity, see \eqref{eq:omega_1_per_rx}), we may take $\ell_{1}=1$. We have ${\rx(\omega+k \ell_{2}\omega_2;t)- \rx(\omega;t)=0}$, and therefore $\rx(\omega+\ell\ell_{2}\omega_3;t)- \rx(\omega;t)=0$. For $0\leq n\leq \ell\ell_{2}-1$, apply $\widetilde\tau^{n}$ to Equations \eqref{eq:omega_3_per_ry} and \eqref{eq:omega_3_per_rx} and make the sum of these $\ell\ell_{2}$ identities. We easily obtain that 
$$b_1(\omega)+\widetilde{\tau}(b_1(\omega))+\cdots+\widetilde{\tau}^{\ell\ell_{2}-1}(b_1(\omega))=0.$$
Using $\widetilde{\tau}^{\ell}(b_1(\omega))=b_1(\omega)$, we deduce that 
$$
0=b_1(\omega)+\widetilde{\tau}(b_1(\omega))+\cdots+\widetilde{\tau}^{\ell\ell_{2}-1}(b_1(\omega))=\ell_{2}\big(b_1(\omega)+\widetilde{\tau}(b_1(\omega))+\cdots+\widetilde{\tau}^{\ell-1}(b_1(\omega))\big)=\ell_{2}\mathcal O_1 (\omega).
$$
Then $\mathcal{O}_1(\omega)=0$. We conclude with $\mathcal O_1 (\omega)=-\mathcal O_2 (\omega)$ that $\mathcal O_2 (\omega)=0$.
\par 

We now assume that the orbit-sum $\mathcal O_1 (\omega)$ is nonzero and want to prove the holonomy. 
Set $\Omega_1=\omega_1$, $\Omega_2=k\omega_2$ and consider 
 $\zeta(\omega;\Omega_1,\Omega_2)$ that is meromorphic on the complex plane and is the (opposite of the) antiderivative of the Weierstrass $\wp$-function: $\zeta'=-\wp$, coupled with the condition $\lim_{\omega\to 0} \zeta (\omega;\Omega_1,\Omega_2)-\frac{1}{\omega}=0$, see \cite[20.4]{WW}:
\begin{equation*}
     \zeta(\omega;\Omega_1,\Omega_2):=\frac{1}{\omega}+ \sum_{(\ell_{1},\ell_{2}) \in \Z^{2}\setminus \{(0,0)\}} \left(\frac{1}{\omega +\ell_{1}\Omega_{1}+\ell_{2}\Omega_{2}} -\frac{1}{\ell_{1}\Omega_{1}+\ell_{2}\Omega_{2}}+\frac{\omega}{(\ell_{1}\Omega_{1}+\ell_{2}\Omega_{2})^{2}}\right).
\end{equation*}   
As in \cite[Equation (4.3.7)]{FIM}, the key idea is to introduce 
\begin{equation}
\label{eq:def_phi}
     \phi(\omega) = \frac{\Omega_1}{2i\pi}\zeta(\omega;\Omega_1,\Omega_2)-\frac{\omega}{i\pi}\zeta(\Omega_1/2;\Omega_1,\Omega_2)
\end{equation} 
and to prove that
\begin{equation}
\label{eq:prop_phi}
     \phi(\omega+\Omega_1)=\phi(\omega) \quad \text{and}\quad \phi(\omega+\Omega_2)=\phi(\omega)+1.
\end{equation}
 Using the quasi-periodicity property of $\zeta$, that is for $i\in\{1,2\}$,
\begin{equation*}
     \zeta(\omega+\Omega_{i};\Omega_1,\Omega_2)-\zeta(\omega;\Omega_1,\Omega_2)=2\zeta({\Omega_i}/{2};\Omega_1,\Omega_2),
\end{equation*}
see \cite[20.41]{WW}, we obtain that $\phi$ is $\Omega_1$-periodic (first identity in \eqref{eq:prop_phi}). Using further the relation
\begin{equation*}
     \zeta({\Omega_1}/{2};\Omega_1,\Omega_2)\Omega_2-\zeta({\Omega_2}/{2};\Omega_1,\Omega_2)\Omega_1=-\pi i,
\end{equation*}
see \cite[20.411]{WW} (the minus sign in front of $\pi i$ differs from \cite{WW}, as the role of $\Omega_1$ and $\Omega_2$ is reversed here), we deduce that the function $\phi$ satisfies $\phi(\omega+k\omega_2)=\phi(\omega)+1$ (second identity in \eqref{eq:prop_phi}).

With the help of the property \eqref{eq:prop_phi} satisfied by $\phi$ and the $(\omega_1,\omega_2)$-ellipticity of $\mathcal O_1$, we may rewrite the first identity in \eqref{eq:O_1-O_2} as the fact that the function
\begin{equation}
\label{eq:function_difference}
     \psi (\omega)=\rx (\omega)-\mathcal O_1(\omega)\phi (\omega)
\end{equation}
is $k\omega_2$-periodic:
\begin{equation*}
     \psi(\omega+k\omega_2)=\psi(\omega).
\end{equation*} 
With the $\omega_{1}$-periodicity, we find using Lemma \ref{lem12} that $\psi$ is algebraic in the variable $x(\omega)$. 

Writing \eqref{eq:function_difference} under the form $\rx (\omega)=\psi (\omega)+\mathcal O_1 (\omega)\phi (\omega)$, one sees that $\rx (\omega)$ is the sum of an algebraic function in $x(\omega)$ and the function $\mathcal O_1(\omega)\phi(\omega)$.  We claim that $\mathcal O_1(\omega)\phi(\omega)$ is holonomic in the variable $x(\omega)$. First, since it is $(\omega_1,\omega_2)$-elliptic, $\mathcal O_1 (\omega)$ is algebraic in $x(\omega)$. Second, $\phi' (\omega)$ is a rational function of $x(\omega)$, as $\zeta' (\omega)=-\wp (\omega)$. So $\mathcal O_1(\omega)\phi(\omega)$ is holonomic in the variable $x(\omega)$, proving the claim. We conclude using closure properties of holonomic functions.
\end{proof}

\begin{proof}[Proof of Corollary \ref{coro:alg_crit2}]
If the group $\langle \iota_{1},\iota_{2} \rangle$ of birational transformation of $\mathbb{P}^{1}(\C)^2$ is finite, then for any $0< t <1$, the group $\langle \iota_{1},\iota_{2} \rangle$ restricted to the kernel curve $\Etproj$ is finite too, see \cite[Section 2]{FaRa-10}. By Proposition \ref{coro1}, Assumption \ref{assumption} is satisfied for every $t\in (0,1)$. The proof of Corollary \ref{coro:alg_crit2} is then a straightforward consequence of Theorem~\ref{theo:alg_crit1}.
\end{proof}

\begin{lemma}\label{lem12}
Let $f(\omega)$ be a $(\omega_1, k\omega_2)$-elliptic function. Then $f(\omega)$ is algebraic in $x(\omega)$.
\end{lemma}

\begin{proof}
The functions $f(\omega)$ and $\wp (\omega)$ are  $(\omega_1, k\omega_2)$-elliptic. Using a well-known property of elliptic functions, there must exist a nonzero polynomial $P\in \mathbb C[X,Y]$ such that $P(f,\wp)=0$, see \cite[20.54]{WW}. Thus $f(\omega)$ is algebraic in $\wp(\omega)$, and hence also in $x(\omega)$ due to Proposition \ref{prop:uniformization}. 
\end{proof}

\bibliography{walkbib}

\end{document}